\documentclass{article}

\usepackage{amsmath,amsthm,verbatim,amssymb,amsfonts,amscd,graphicx,graphics, hyperref}
\usepackage{pdfpages}
\usepackage{amsmath, amssymb, graphics, setspace, amsfonts}
\usepackage{amsmath,amsthm,enumerate,amssymb,enumitem,tikz,float,array,multirow,todonotes}
\usepackage{mathtools}
\usepackage{tikz-cd}
\usetikzlibrary{matrix,arrows,backgrounds}
\usepackage{imakeidx}
\usepackage{kantlipsum}

\allowdisplaybreaks

\usepackage{mathrsfs}
\usepackage{import}
\usepackage{example}
\usepackage{makeidx}
\usepackage{caption,subcaption}

\graphicspath{ {./figures/} }

\DeclarePairedDelimiter\floor{\lfloor}{\rfloor}

\numberwithin{equation}{section}

\counterwithin*{equation}{section}

\theoremstyle{definition} \newtheorem{theorem}[equation]{Theorem}
\theoremstyle{definition} 
\theoremstyle{definition} \newtheorem{definition}[equation]{Definition}
\theoremstyle{definition} \newtheorem{lemma}[equation]{Lemma}
\theoremstyle{definition} 
\theoremstyle{definition} \newtheorem{corollary}[equation]{Corollary}

\theoremstyle{remark} \newtheorem{remark}[equation]{Remark}
\theoremstyle{remark} \newtheorem{claim}[equation]{Claim}

\title{Round Trees and Conformal Dimension in Random Groups: low density to high density}
\author{Jordan Frost}

\begin{document}

\maketitle

\begin{abstract}
We investigate conformal dimension for the class of infinite hyperbolic groups in the Gromov density model $\mathcal{G}^d_{m,l}$ of random groups with $m \geq 2$ fixed generators, density $0 < d < 1/2$ and relator length $l \to \infty$.
Our main result is a lower bound linear in $l$ at all densities $0 < d < 1/2$ achieved by building undistorted round trees coming directly from lower density Gromov random groups.
\end{abstract}

\tableofcontents

\section{Introduction}\label{s : introduction}

An important invariant of the large scale geometry of a hyperbolic group $G$ is the conformal dimension of its boundary $\partial_{\infty}G$.
It captures both the metric and analytic properties of the space.
Conformal dimension is a quasi-isometric invariant of hyperbolic groups \cite{Mackay-Tyson-survey-conformal-dimension} having been introduced by Pansu \cite{Pansu-conformal-dimension} as a tool for the study of the classical rank one symmetric spaces and their lattices.
It also has various links to actions on $L_p$-spaces \cite{Bourdon-cohomologie-and-isometric-actions-on-Lp-spaces}.

Following the Gromov-Ollivier hyperbolicity $1/2$-theorem \cite{Gromov-invariants, ollivier-survey} it is thus a fruitful problem to see what can be said about the conformal dimension of the boundary of `generic groups' in the Gromov random group model $\mathcal{G}^d_{m,l}$.
They are  known to have boundaries at infinity homeomorphic to the Menger sponge \cite{Dahmani-Guirardel-Przytycki-random-groups-do-not-split}.
To distinguish these groups via the boundary it is natural to consider the conformal dimension of the boundaries.
In this paper we are interested in estimating (from below) this analytic invariant given the algebraic information for the random group model $\mathcal{G}^d_{m,l}$.

This work was carried out initially in \cite{Mackay-conf-rand-12, Mackay-conf-rand-16} culminating in the result of, for $d < \frac{1}{8}$, $\text{Confdim}(\partial_\infty G) \asymp_C dl / \lvert \log d \rvert$   \cite[Thm.1.3]{Mackay-conf-rand-16} by using the small cancellation properties of low density random groups.
The notation $A \asymp_C B$ stands for $ B / C \leq A \leq C B$.
This implies that as $l \to \infty$ for $d < 1/8$ the model $\mathcal{G}^d_{m,l}$ passes through infinitely many quasi-isometry classes.
As a corollary, at the appropriate density, the density can be recovered, up to a power, by the conformal dimension of the boundary and the Euler characteristic of the group.

The methods used to obtain those results are a mixture of $l_p$-cohomology (following Bourdon-Kleiner) \cite{Bourdon-Kleiner-combinatorial-modulus-lowener-property-coxeter-groups, Bourdon-Kleiner-applications-lpcohomology-boundaries-gromov-hyp} and walls in the Cayley complex (building on Wise and Ollivier-Wise) \cite{Wise-cubulating-small-cancellation, Ollivier-Wise-cubulating-low-density-random-groups} to find upper bounds.
To find lower bounds the author of \cite{Mackay-conf-rand-16} constructs `combinatorial round trees', see Definition \ref{d : round tree}, in the Cayley complex analogous to the `round trees' Gromov defines \cite[Sec.7.C3]{Gromov-invariants}.

For introductions and more background on random groups we refer the reader to \cite{ollivier-survey}, and on conformal dimension to \cite{Mackay-Tyson-survey-conformal-dimension}.

The culmination of that work and ours in this paper is the following Theorem which is our headline result.

\begin{theorem}(Conformal dimension bound)\label{t : headline}
There exists $C\geq 1 $ so that for all $m \geq 2$ and $0 < d < 1/2$ a group $G \sim \mathcal{G}^{d}_{m,l}$ in the Gromov density model is infinite hyperbolic and satisfies:

$$ \frac{d (1-2d)^5 l }{C \lvert \log(d(1/2 - d) ) \rvert} \leq \frac{\text{Confdim}(\partial_\infty G)}{\log(2m-1)} \leq \frac{C d l}{ (1 - 2d)\lvert \log d \rvert } $$
\noindent
with overwhelming probability.
Therefore, as $l \to \infty$ the collection of groups in $\mathcal{G}^{d}_{m,l}$ passes through infinitely many distinct quasi-isometry classes at all densities.
\end{theorem}

Our contribution to Theorem \ref{t : headline} is the following result with the previous bounds being covered by \cite[Prop.1.7]{Mackay-conf-rand-12} for the upper at all densities and \cite[Thm.1.3]{Mackay-conf-rand-16} for the lower at small densities.

\begin{theorem}(Linear lower bound)\label{t : main result}
There exists $C\geq 1 $ so that for all $m \geq 2$ and $1/8 \leq d < 1/2$ a group $G \sim \mathcal{G}^{d}_{m,l}$ in the Gromov density model is infinite hyperbolic and satisfies

$$ \frac{\text{Confdim}(\partial_\infty G)}{\log(2m-1)} \geq \frac{(1-2d)^5 l}{C \lvert \log(1/2 - d) \rvert}$$ 

\noindent
with overwhelming probability.
\end{theorem}

Another approach of finding lower bounds of $\text{Confdim}(\partial_{\infty} G)$ is via by finding fixed points for groups acting on $L^p$ noticed by Bourdon \cite{Bourdon-cohomologie-and-isometric-actions-on-Lp-spaces}.
This was used by \cite{Drutu-Mackay-random-groups-random-graphs-eigenvalies-Lp-laplacians, Laat-Salle-banach-space-actions-L2-spectral-gap} to find lower bounds for the Triangle model of random groups.
For details about this model see \cite{Zuk-propertyT-kazhdan-constants-for-discrete-groups, Kotowski-Kotowski-random-groups-and-propertyT-Zuk-revisited, Antoniuk-Luczak-Swicatkowski-random-triangular-groups-at-a-third}.

A very recent major development of this method by Oppenheim to the Gromov density model, that came out during the writing up of our work, applies to the Gromov density
model and finds new fixed point results for random groups at density $ 1/3 <
d < 1/2 $ \cite[Thm.1.5]{Oppenheim-banach-fixed-point-theorems-groups}. 
This results in a linear in $l$ lower bound \cite[Thm.1.8]{Oppenheim-banach-fixed-point-theorems-groups}.
Our linear in $l$ lower bound on $\text{Confdim}(\partial_{\infty} G)$ is entirely new for the range $1/8 \leq d \leq 1/3$ and applies for all $d < 1/2$; unlike Oppenheim we say nothing about fixed points.

On the level of constants, our lower bounds are stronger than Oppenheim's for $d$ close to $1/3$ whilst his are stronger for $d$ close to $1/2$.
Combining that work with Theorem \ref{t : headline} lets us report the following as the state of the art for the Gromov density model.

\begin{corollary}(Theorem \ref{t : headline} with \cite[Thm.1.8]{Oppenheim-banach-fixed-point-theorems-groups})\label{t : state of the art}
There exists $C\geq 1 $ so that for all $m \geq 2$ and $0 < d < 1/2$ a group $G \sim \mathcal{G}^{d}_{m,l}$ in the Gromov density model is infinite hyperbolic and satisfies

$$ \frac{d l}{C \lvert \log d \rvert} \leq \frac{\text{Confdim}(\partial_\infty G)}{\log(2m-1)} \leq \frac{C d l}{ (1 - 2d) \lvert \log d \rvert}$$ 

\noindent
with overwhelming probability as $l \to \infty$.
\end{corollary}

Theorem \ref{t : main result} is shown by building undistorted combinatorial round trees (Definition \ref{d : round tree}) at any density $d_t \in [1/8,1/2)$.
We show there is a smaller density $d_s \in (0,1/8)$ where a suitable round tree exists (in fact, many do) and this round tree remains undistorted under the natural quotient map. 
This is done in Section \ref{s : build round trees all densities} where parameters $\beta, \eta$ and $H$ of the round tree are also defined.

\begin{theorem}(High density round trees)\label{t : intro round trees everywhere}
For all $m \geq 2$ and $1/8 \leq d_t < 1/2$ there exists $\beta, \eta, d_s, H$ defining a round tree $A \rightarrow K_{d_s} \rightarrow K_{d_t}$ which is an undistorted round tree with overwhelming probability in $\mathcal{G}^{d_s}_{m,l}$ and $\mathcal{G}^{d_t}_{m,l}$.
\end{theorem}

\begin{remark}
The geometric properties of $A$ depend on the parameters $\beta, \eta$ and $H$ (see Definition \ref{d : round tree} with $V = (2m-1)^{\beta \eta l}$ for details).
\end{remark}

The impetus of these ideas is Calegari-Walker's methods of building quasi-convex surfaces in Gromov's model \cite{Calegari-Waker-surfsub-rand-15}.
They first build surface subgroups in random one-relator groups, and then show how to ensure the subgroup remains embedded at positive densities by suitably modifying their construction. 
Of particular relevance is the diagrams they rule out in \cite[Thm.6.4.1]{Calegari-Waker-surfsub-rand-15} and the restrictions on the surfaces the defining fatgraph stucture gives \cite[Sec.3]{Calegari-Waker-surfsub-rand-15}.
See Lemma \ref{l : filling diagrams along paths} for our version of this method to rule diagrams out.

Our methods may be useful in showing other quasi-convex geometric structures in low-density random groups remain quasi-convex under the natural quotient map to high-density groups.
The author hopes to make this precise in subsequent work. 


\subsection*{Organisation}

In the preliminaries, Section \ref{s : prelim}, we start by discussing Gromov random groups and their important properties, then diagrams at low-density groups highlighting what we will use later.
Finally, we have a short discussion on conformal dimension introducing the invariant.

In the section on round trees, Section \ref{s : round trees}, we first introduce in Section \ref{ss : round trees geometry} the geometric object of a combinatorial round tree, describe its geometry and state how we can obtain lower bounds in conformal dimension.
We then outline in Section \ref{ss : low density undistorted round trees} our adapted construction of `thin' round trees at low-density random groups where in Section \ref{ss : emanating restrictions} we detail what is meant by `thin'.

In the section on probabilistic diagram rule outs, Section \ref{s : prob diagram rule out}, we define the notion of a restricted abstract diagram, Definition \ref{d : abstract-vkmp}, and following Ollivier \cite{ollivier-survey} in Theorem \ref{t : restricted abstract diagram rule out} we show if the number of restrictions is `large enough' the diagram is not fillable.

In the section on building round trees at high densities, Section \ref{s : build round trees all densities}, we show the round trees constructed in Section \ref{s : round trees} are quasi-isometrically embedded under the natural quotient map from a low-denisity random group.

\subsection*{Acknowledgements}

I would like to thank my PhD supervisor, John Mackay, for his patience, guidance and insights.
I feel extremely lucky for his suggestion to pursue this project and the many interesting and helpful conversations we had during its working out and completion.

\section{Preliminaries}\label{s : prelim}

Our motivation is studying the conformal dimension of random groups in the Gromov density model $\mathcal{G}^{d}_{m,l}$.
We construct lower bounds by building conbinatorial round trees in low-density random groups and prove they stay quasi-isometrically embedded under the natural quotient to higher density groups.

In this section we introduce the model of random groups that we are interested in, what low-density properties we will later use and the conformal dimension of a hyperbolic group.

\subsection{Random Groups}\label{ss : random group notions}

Our random groups will be defined by presentations with a fixed finite generating set $S$ along with a probability measure defining the set of relations whose support is a subset of words in $S$.
Throughout we will only consider finitely presented groups $G$.
This is often what is meant by a random group, invoking Gromov, in the geometric group theory literature, \cite[Sec.9.B]{Gromov-invariants} and \cite[Sec.1]{Gromov-random-walk-in-random-groups}.

The particular model of interest in this paper is the so called Gromov density model from \cite[Sec.9.B]{Gromov-invariants}.

\begin{definition}(Gromov density model)\label{d : gromov density model}
For all $0 \leq d \leq 1$ and $m \geq 2$, $l \geq 2$ integers, \emph{Gromov's model of random groups at density} $d$,  $\mathcal{G}^{d}_{m,l} $, is a probability distribution over groups defined by presentations $\langle S |R \rangle$ where $\lvert S \rvert = m$ and $R$ is a list of $\floor*{(2m-1)^{dl}}$ cyclically reduced words of length $l$.
 
The probability of obtaining a group $G \sim \mathcal{G}^d_{m,l}$ is the probability we obtain $G$ as a presentation $\langle S | R \rangle$ where each cyclically reduced word in $R$ is
chosen with uniform probability and independently among all such words in $S \cup S^{-1}$.
\end{definition}



We will use the symbolism of $G \sim \mathcal{G}^{d}_{m,l}$ to take a random group $G$ from the model $\mathcal{G}^{d}_{m,l}$ under the determining probability measure.
For notational convenience when using densities to count words we will drop the floor function but implicitly count using it.

The parameters $m$, $l$ and $d$ are called the number of generators, relator length and the density of the model.
Random groups were introduced to talk about `generic groups' \cite[Sec.9B]{Gromov-invariants}. 
We need a notion of a property being generic.
A property $\mathcal{P}$ in our context can be viewed as a random variable defined on the state space of $\mathcal{G}^d_{m,l}$ with image in $\{0,1\}$.
Write $\mathcal{P}(G) = 1$ if the property holds for group $G$, $0$ otherwise.

\begin{definition}(Genericity in relator length in Gromov's density model)\label{d : generic notion}
The property $\mathcal{P}$ is said to be \emph{generic in relator length} in the Gromov density model $\mathcal{G}^d_{m,l}$ if we have
$$ \mathbb{P}(\mathcal{P}(G) = 1 \text{ for } G \sim \mathcal{G}^{d}_{m,l} ) \to 1 $$
as $l \to \infty$.
\end{definition}

If a property is generic in relator length in $\mathcal{G}^d_{m,l}$ we will simply say it is generic in $\mathcal{G}^d_{m,l}$ or holds \emph{with overwhelming probability in} $\mathcal{G}^d_{m,l}$. 
This is the usual sense of genericity for the Gromov density model \cite[I.2.a]{ollivier-survey}.
 
We will also adopt the notation $\mathcal{R}^d_{m,l}$ which is the random variable of cyclically reduced words for the model parameters under the measure.
This defines a random group $G = \langle S | \mathcal{R}^d_{m,l} \rangle \sim \mathcal{G}^d_{m,l}.$

See \cite{ollivier-survey} and \cite{Kapovich-Schupp-group-random-mods} for background on random groups, Gromov's density model and other models.
The following foundational results are found in \cite[Sec.I.3.a.Thm.13]{ollivier-survey} and originate from Gromov \cite[Sec.9.B]{Gromov-invariants}.

\begin{theorem}(Gromov-Ollivier's genericity of hyperbolicity)\label{t : hyperbolic generic}
For all $m \geq 2$, $0 < d < 1/2$ the group $G \sim\mathcal{G}^d_{m,l}$ is infinite hyperbolic with hyperbolicity constant satisfying $\delta \leq 4l /(1-2d)$ in $\mathcal{G}^d_{m,l}$ with overwhelming probability.
\end{theorem}

One knows something a lot stronger.

\begin{theorem}(Generic linear isoperimetric inequality)\label{t : generic linear iso}
For all  $m\geq 2$, $0 < d < 1/2$ and $\epsilon > 0$, any reduced van-Kampen diagram $D$ for $G \sim \mathcal{G}^d_{m,l}$ satisfies

$$ \lvert \partial D \rvert \geq (1 - 2d - \epsilon) l \lvert D \rvert $$

\noindent
in $\mathcal{G}^d_{m,l}$ with overwhelming probability.
\end{theorem}

Small cancellation properties are generic at certain densities of the Gromov model \cite[Sec.I.2.Prop.10]{ollivier-survey}.
See \cite{Strebel-small-cancellation, McCammond-Wise-small-cancellation-fans-ladders} for information about the metric small cancellation property $C'(\cdot)$ and \cite{Bishop-Ferov-density-of-metric-small-cancellation-in-finitiely-presented-groups, Tsai-density-random-subsets-applications-to-group-theory} for a recent discussion in the context of density models.

\begin{theorem}(Generic small cancellation)\label{t : generic small cancelation}
For $\lambda >0$ we have

\begin{itemize}
    \item[(1)] for $d < \lambda/2$, $G \sim \mathcal{G}^{d}_{m,l}$ satisfies $C'(\lambda)$ in $  \mathcal{G}^{d}_{m,l}$ with overwhelming probability,
    \item[(2)] for $d > \lambda/2$, $G \sim \mathcal{G}^d_{m,l}$ does not satisfy $C'(\lambda)$ in $  \mathcal{G}^{d}_{m,l}$ with overwhelming probability.
\end{itemize}

\end{theorem}

\subsection{Diagrams at Low-density Random Groups}

Our main tool in studying the geometry of groups $G = \langle S | \mathcal{R}^d_{m,l} \rangle \sim \mathcal{G}^{d}_{m,l}$ will be the study of van-Kampen diagrams in the Cayley $2$-complex $X = \text{Cay}^{(2)}(G, S, \mathcal{R}^{d}_{m,l} )$ of a finitely presented group and other diagrammatic notions. 
See Definition \ref{d : abstract-vkmp} and Section \ref{s : prob diagram rule out} at large.
We use van-Kampen diagrams extensively in Section \ref{s : build round trees all densities} when we look at the metric geometry of $X$, given by the natural word metric, to show the round trees are undistorted.

As we wish to go from low-density groups to high-density we require special attention to diagrams at low-density models.
We now collect the properties that we will use later.

\begin{lemma}($2$-cells along geodesic, \cite[Lem.8.4]{Mackay-conf-rand-16})\label{l : cells intersecting along geodesics}
For $m \geq 2$, $0 < d < 1/4$ with overwhelming probability in $\mathcal{G}^{d}_{m,l}$ we have, for any geodesic $\gamma$ in the Cayley graph $X^{(1)}$, for every $2$-cell $R \subset X $, $R \cap \gamma$ is connected. 
\end{lemma}

There are certain diagrams that are ubiquitous at low-density.

\begin{definition}(Ladders, \cite[Def.8.6]{Mackay-conf-rand-16})\label{d : ladder definition}
A connected disk diagram $D$ is a \emph{ladder}, from $\beta_1 \subset \partial D$ to $\beta_2 \subset \partial D$, if $D$ is a union of a sequence of cells $R_1, R_2, \ldots, R_k$ for some $k \in \mathbb{N}_{\geq 1}$, where each $R_i$ is a closed $1$-cell or $2$-cell, and $R_i \cap R_j = \emptyset$ for $\lvert i - j \rvert > 1$.
Moreover, $\beta_1$ and $\beta_2$ are closed paths in $R_1 \setminus R_2$ and $R_k \setminus R_{k-1}$ respectively.
\end{definition}

In particular, for $0 < d < 1/6$ any reduced diagram for a geodesic bigon is a ladder.

\begin{lemma}(Finding ladders from bigons, \cite[Lem.8.7]{Mackay-conf-rand-16})\label{l : finding a ladder}
For $m \geq 2$, $0 < d < 1/6$ with overwhelming probability  in $\mathcal{G}^d_{m,l}$ we have for any reduced diagram $D \rightarrow X$ in $G \sim \mathcal{G}^d_{m,l}$ the following:
\begin{itemize}
    \item[(1)] Suppose $\partial D$ consists of, in order, a geodesic $\gamma_1$, a path $\beta_1 \subset R \subset D$, a geodesic $\gamma_2$, and a path $\beta_2 \subset R' \subset D$.
    Also suppose, that $\partial R$ and $\partial R'$ each contain edges from both $\gamma_1$ and $\gamma_2$.
    Then, $D$ is a ladder from $\beta_1$ to $\beta_2$ with $R = R_1$ and $R' = R_k$ for some $k \in \mathbb{N}_{k \geq 1}$.
    \item[(2)] Suppose $\partial D$ consists of a geodesic $\gamma_1$, a path $\beta_1 \subset R \subset D$, a geodesic $\gamma_2$, with $\gamma_1$, $\gamma_2$ sharing an endpoint $\beta_2$. 
    Also suppose, that $\beta_1 = D \cap R$ for some $2$-cell $R \subset X$ with $D \cup R \rightarrow X$ is also a reduced diagram.
    Then, $D$ is a ladder from $\beta_1$ to $\beta_2$.
    \item[(3)] Suppose $\partial D$ is a geodesic bigon $\gamma_1 \cup \gamma_2$ with endpoints $\beta_1$ and $\beta_2$.
    Then, $D$ is a ladder from $\beta_1$ to $\beta_2$. 
\end{itemize}
\end{lemma}

\subsection{Conformal Dimension}\label{ss : conf dim}

We will be looking at the conformal dimension invariant for hyperbolic groups.

Conformal dimension was introduced by Pansu \cite{Pansu-conformal-dimension} in connection with the study of rank one symmetric spaces.
We follow \cite{Mackay-Tyson-survey-conformal-dimension} for the discussion in this section and defer the reader to that survey for the definition of Hausdorff dimension and further statements about the invariant.

Infinite hyperbolic groups $G$ have a canonical family of metrics on $\partial_\infty G$ called a conformal gauge.
This includes all visual metrics on the boundary $\partial_\infty G $ whose existence is stated by, for example, \cite[Thm.3.2.9]{Mackay-Tyson-survey-conformal-dimension}.

\begin{definition}(Conformal dimension)\label{d : conf dim def}
The \emph{conformal dimension} of a conformal gauge is the infimum of the Hausdorff dimensions of all metric spaces in the gauge.

\end{definition}

Conformal dimension for the Gromov boundary $\partial_{\infty}G$ of a finitely generated hyperbolic group $G$ is a quasi-isometric invariant, see for example \cite[Thm. 3.2.13]{Mackay-Tyson-survey-conformal-dimension}.
It now makes sense to talk of the conformal dimension $\text{Confdim}(\partial_\infty G)$ of a finitely generated hyperbolic group $G$.
Note Theorem \ref{t : hyperbolic generic} tells us for groups in the model $\mathcal{G}^d_{m,l}$ where $m \geq 2$ and $0 < d < 1/2$ we have the groups are infinite and hyperbolic with overwhelming probability as $l \to \infty$.
We will be looking at $\text{Confdim}( \partial_\infty G)$ for $G \sim \mathcal{G}^d_{m,l}$ in the Gromov density model.
In particular, we wish to investigate $\text{Confdim}(\partial_\infty G)$ as a function of the defining algebraic parameters $m,l, d$ of the Gromov density model $\mathcal{G}^d_{m,l}$.

\section{Round Trees}\label{s : round trees}

The conformal dimension of a metric space can be bounded from below by finding within the space a product of a Cantor set and an interval \cite{Mackay-conf-rand-16}.
In this section we construct `combinatorial round trees' in the Cayley complex of low-density random groups that are additionally `thin'.
These round trees produce a Cantor set and an interval within the boundary of a hyperbolic group which results in lower bounds on conformal dimension by Theorem \ref{t : undistorted round trees lower bound}.

\subsection{Geometry of Round Trees}\label{ss : round trees geometry}

To find lower bounds of $\text{Confdim}( \partial_\infty G)$ for a hyperbolic finitely presented group $G = \langle S | R \rangle$ we will build so called combinatorial round trees in the Cayley $2$-complex $\text{Cay}^{(2)}(G,S, R)$.
The Cayley $2$-complex can be seen as the universal cover of the presentation complex or the Cayley graph $\text{Cay}^{(1)}(G,S)$ with relation loops filled in by $2$-cells using natural attaching maps. This approach is the method of finding lower bounds used in \cite{Mackay-conf-rand-16} inspired by  \cite[Sec.7.C3]{Gromov-invariants}.

Recall that a \emph{star} $St(x,\mathcal{U})$ of $x \in X$, $X$ a set, with respect to a family of subsets $\mathcal{U}$ of $X$ is the union of all $U \in \mathcal{U}$ containing $x$.
If $A \subset X$, then $St(A, \mathcal{U}) := \cup_{x \in A} St(x, \mathcal{U})$.

\begin{definition} (Combinatorial round tree, \cite[Def.7.1]{Mackay-conf-rand-16})\label{d : round tree}
A $2$-complex is a \emph{combinatorial round tree} with vertical branching $V \in \mathbb{N}$ and horizontal branching at most $H \in \mathbb{N}$ if, setting $T = \{1,2, \ldots, V \}$, we can write

$$ A = \bigcup_{\mathbf{a} \in T^{\mathbb{N}}} A_{\mathbf{a}},  $$

where
\begin{itemize}
    \item[(1)] $A$ has a base point $1$, contained in the boundary of a unique (initial) $2$-cell $A_{\emptyset} \subset A$.
    \item[(2)] Each $A_{\mathbf{a}}$ is an infinite planar $2$-complex, homeomorphic to a half-plane whose boundary is the union of two rays $L_{\mathbf{a}}$ and $R_{\mathbf{a}}$ with $L_{\mathbf{a}} \cap R_{\mathbf{a}} = \{1\}$.
    \item[(3)] Setting $A_{\emptyset} = A_0$, and for $n > 0$, let $A_n = A_{n-1} \cup St(A_{n-1}, \mathcal{U})$ where $\mathcal{U}$ is the collection of $2$-cells in $A$.
    Given $\mathbf{a} = (a_1, a_2, \ldots) \in T^{\mathbb{N}}$, let $\mathbf{a}_n = \mathbf{a} |_{[1, n]} \in T^n$.
    If $\mathbf{a}, \mathbf{b} \in T^{\mathbb{N}}$ with $\mathbf{a}_n = \mathbf{b}_n$, $\mathbf{a}_{n+1} \neq \mathbf{b}_{n+1}$, then 
    $$A_n \cap A_{\mathbf{a}} \subset A_{\mathbf{a}}\cap A_{\mathbf{b}} \subset A_{n+1} \cap A_{\mathbf{a}}. $$
    We require that each $2$-cell $R \subset A_n$ meets at most $V \cdot H$ $2$-cells in $A_{n+1} \setminus A_{n}$.
\end{itemize}

To emphasise the branching we will sometimes denote a combinatorial round tree as $A(V,H)$.
We will equip the $1$-skeleton $A^{(1)}$ with the unit combinatorial path metric to see it as a metric space.
\end{definition}

See Figure \ref{f : round tree} for a cartoon image of a combinatorial round tree.

For each $\mathbf{a}_n \in T^{n}$ we set $A_{\mathbf{a}_n}  := A_{\mathbf{a}} \cap A_n$ and have naturally defined combinatorial inclusion maps

$$ A_{\mathbf{a}_n} \rightarrow A_n \rightarrow A $$

\noindent
for a round tree $A$.
Taking $\mathbf{a} \in T^{\mathbb{N}}$ we restrict $L_{\mathbf{a}}$ and $R_{\mathbf{a}}$ to $A_{\mathbf{a}_n}$ defining $L_{\mathbf{a}_n}$ and $R_{\mathbf{a}_n}$.
This lets us partition the boundary of $A_{\mathbf{a}_n} \rightarrow \mathbb{R}^2$, considered as a planar complex, as $\partial A_{\mathbf{a_n}} = E_{\mathbf{a_n}} \cup L_{\mathbf{a_n}} \cup R_{\mathbf{a_n}}$  where $E_{\mathbf{a_n}} = \partial A_{\mathbf{a_n}} \setminus L_{\mathbf{a_n}} \cup R_{\mathbf{a_n}}$ is called the \emph{outer boundary} of $A_{\mathbf{a_n}}$.
We similarly define the outer boundary of $A_n$, denoted $E_n$, as the union of $E_{\mathbf{a}_n}$ over possible $\mathbf{a}_n$ coming from the round tree structure.

\begin{definition}(Extension path)\label{d : local v branching}
Consider a combinatorial round tree $A = A(V,H)$, $n, k \in \mathbb{N}_{\geq 0}$ and a vertex $v \in E_n^{(0)}$ of the boundary of $A_n$.
Embedded paths $P \rightarrow A_{n+1}^{(1)}$ of length $k$ with $P(0) = v$ and $P \setminus \{v \} \subset A \setminus \overline{A_{n}}$ are called $(n,k)$-\emph{extension paths} based at $v$.
\end{definition}

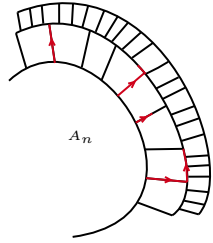
\begin{figure}[!ht]
    \centering

\tikzset{every picture/.style={line width=0.75pt}} 

\begin{tikzpicture}[x=0.75pt,y=0.75pt,yscale=-1,xscale=1]

\draw    (284.11,147.83) .. controls (328.8,105.8) and (396.8,215.8) .. (316.2,226.65) ;
\draw    (287.6,123.8) -- (292.8,141.4) ;
\draw    (304.4,119.4) -- (307.2,138.2) ;
\draw    (325,122.3) -- (320.8,141) ;
\draw    (338.2,129.7) -- (330.8,147.5) ;
\draw [color={rgb, 255:red, 208; green, 2; blue, 27 }  ,draw opacity=1 ]   (352.4,143.94) -- (338.8,155.54) ;
\draw    (363.4,160.5) -- (348,169) ;
\draw    (352.4,182.5) -- (372,182.2) ;
\draw    (353.2,197) -- (374,198.9) ;
\draw    (349.2,209.4) -- (366.4,215.8) ;
\draw    (287.6,123.8) .. controls (345.6,95.8) and (393.2,197.8) .. (366.4,215.8) ;
\draw    (288.4,113.3) -- (290,122.4) ;
\draw    (294.62,111.46) -- (295.8,120.2) ;
\draw    (309.38,108.85) -- (309.4,119) ;
\draw    (316.4,108.9) -- (315.8,119.8) ;
\draw    (324,110.1) -- (321,121.2) ;
\draw    (300.62,109.77) -- (301.4,119) ;
\draw    (377.85,160.54) -- (365.72,165.03) ;
\draw    (380.2,166.7) -- (368.4,170.8) ;
\draw    (370.8,178.3) -- (382.8,174.7) ;
\draw    (372.8,186.3) -- (384.8,185.5) ;
\draw    (288.4,113.3) .. controls (360.6,85.7) and (401.2,189.9) .. (379.8,215.3) ;
\draw    (369.8,212.5) -- (379.8,215.3) ;
\draw    (372.91,151.36) -- (361.64,157) ;
\draw    (369.64,146.09) -- (358.36,151.73) ;
\draw    (364.45,138.26) -- (359.33,141.27) -- (353,144.44) ;
\draw    (356.77,128.69) -- (346.18,136.27) ;
\draw    (350.62,123.46) -- (341.3,131.91) ;
\draw    (344,118.69) -- (336.07,128.01) ;
\draw    (339.08,115.15) -- (332.2,126.31) ;
\draw    (332.15,112.38) -- (327.82,123.18) ;
\draw    (385.09,190.45) -- (373.53,191.14) ;
\draw    (385.54,196.23) -- (374.55,196.27) ;
\draw    (384.77,204.08) -- (373.82,203) ;
\draw    (382.46,210.38) -- (372.43,208.38) ;
\draw [color={rgb, 255:red, 208; green, 2; blue, 27 }  ,draw opacity=1 ]   (349,140.06) .. controls (352.13,143.56) and (351.17,142.14) .. (353,144.44) ;
\draw [color={rgb, 255:red, 208; green, 2; blue, 27 }  ,draw opacity=1 ]   (348,169) -- (356.43,164.36) ;
\draw [color={rgb, 255:red, 208; green, 2; blue, 27 }  ,draw opacity=1 ]   (304.4,119.4) -- (307.22,138.91) ;
\draw  [color={rgb, 255:red, 208; green, 2; blue, 27 }  ,draw opacity=1 ][fill={rgb, 255:red, 208; green, 2; blue, 27 }  ,fill opacity=0 ] (304.42,130.23) -- (305.67,127.16) -- (307.23,130.08) ;
\draw  [color={rgb, 255:red, 208; green, 2; blue, 27 }  ,draw opacity=1 ][fill={rgb, 255:red, 208; green, 2; blue, 27 }  ,fill opacity=0 ] (344.2,149.26) -- (347.41,148.46) -- (345.98,151.44) ;
\draw  [color={rgb, 255:red, 208; green, 2; blue, 27 }  ,draw opacity=1 ][fill={rgb, 255:red, 208; green, 2; blue, 27 }  ,fill opacity=0 ] (349.93,166.24) -- (353.24,166.29) -- (351.09,168.81) ;
\draw [color={rgb, 255:red, 208; green, 2; blue, 27 }  ,draw opacity=1 ]   (374,198.9) -- (353.2,197) ;
\draw [color={rgb, 255:red, 208; green, 2; blue, 27 }  ,draw opacity=1 ]   (372,182.2) -- (373.44,189.44) -- (374,198.9) ;
\draw  [color={rgb, 255:red, 208; green, 2; blue, 27 }  ,draw opacity=1 ][fill={rgb, 255:red, 208; green, 2; blue, 27 }  ,fill opacity=0 ] (362.78,196.38) -- (365.69,197.97) -- (362.61,199.19) ;
\draw  [color={rgb, 255:red, 208; green, 2; blue, 27 }  ,draw opacity=1 ][fill={rgb, 255:red, 208; green, 2; blue, 27 }  ,fill opacity=0 ] (372.2,193.56) -- (373.45,190.49) -- (375.01,193.41) ;

\draw (312.11,171.84) node [anchor=north west][inner sep=0.75pt]  [font=\tiny]  {$A_{n}$};

\end{tikzpicture}

    \caption{Examples of extension paths.}
    \label{f : extension paths}
\end{figure}

\begin{definition}(Bracket)\label{d : bracket}
Consider a combinatorial round tree $A = A(V,H)$, $n,k \in  \mathbb{N}_{\geq 0}$ and
a $2$-cell $C \rightarrow A_{n+1} \setminus A_{n}$ necessarily with $\partial C \cap \partial A_{n} \neq \emptyset$. 

Denote the endpoints of $\partial C \cap \partial A_{n}$ as $p_1$, $p_2$ and let $v_1, v_2 \in \partial C \setminus \partial A_{n}$ be distance $k$ from $p_i$ in the path metric of $\partial C$.
A \emph{bracket path} $B(C)$ is the concatenation of the subpaths of $\partial C$ joining $v_1$ to $p_1$ to $p_2$ to $v_2$ or empty if no such $v_1, v_2$ exist. 
We call the collection of $B(C)$ as $C$ varies the set of $(n,k)$-brackets of $A$. 
\end{definition}

From this point forward all our brackets $B(C)$ will be non-empty.
Notice that a $(n,k)$-bracket contains two $(n,k)$-extension paths as subpaths.

We will say a combinatorial round tree $A = A(V,H)$ is \emph{built} in a finitely presented group $G = \langle S \rangle$ if we can find a topological immersion $\varphi : A \rightarrow \text{Cay}^2(G,S)$.
We say $A$ is \emph{partially built} if there exists $n$ with a topological immersion $A_n \rightarrow \text{Cay}^2(G,S)$.

\begin{definition}(Extension words)\label{d : local words}
Let $\mathbf{X}_k$ be a collection of reduced words in $S$ of length $k$.
$\mathbf{X}_k$ is the \emph{extension words} for a built round tree $A \rightarrow \text{Cay}^2(\langle S \rangle, S)$ if the set of labels of the images of $(n,k)$-extension paths is contained in $\mathbf{X}_k$ independent of $n$.
We will let $\mathbf{X}_k(v)$ be the subset of labels based at $v \in A^{(0)}$.
\end{definition}

The following theorem gives lower bounds of $\text{Confdim}(\partial_\infty G)$ by building certain combinatorial round trees for hyperbolic groups $G$.

\begin{theorem}(Lower bound by round trees, \cite[Thm.7.2]{Mackay-conf-rand-16})\label{t : undistorted round trees lower bound}
Let $A(V,H)$ be a combinatorial round tree with $V, H\geq 2$ built in a hyperbolic finitely generated group $G$.
If $\varphi : A \rightarrow \text{Cay}^2(G,S)$ restricts to a quasi-isometric embedding on $A^{(1)}$ then,

$$ \text{Confdim}( \partial_{\infty} G) \geq 1 + \frac{\log(V)}{\log(H)}.$$
\end{theorem}

We will call a combinatorial round tree with the quasi-isometric embedding criteria in the hypothesis of the above theorem \emph{undistorted}.

\subsection{Round Trees at Low Density}\label{ss : low density undistorted round trees}

We turn our attention to the generic infinite hyperbolic groups $G \sim \mathcal{G}^d_{m,l}$.

In \cite{Mackay-conf-rand-16} undistorted round trees were built in $G \sim \mathcal{G}^d_{m,l}$ for $0 < d < 1/8$ with overwhelming probability in $\mathcal{G}^d_{m,l}$ as $l \to \infty$.
The proof technique heavily relies on metric small cancellation properties of low-density Gromov density random groups to control the geometry of van-Kampen diagrams \cite[Prop. 8.3, Lem. 8.4, Thm. 8.5, Lem. 8.7]{Mackay-conf-rand-16}.
These small cancellation properties are no longer generic at higher density as Theorem \ref{t : generic small cancelation} indicates.

The existence of these round trees is important for our work (building trees at higher density) as we show under restrictions on the number of emanating words (Definition \ref{d : emanating words}) the round trees stay built after further quotienting, and more importantly stay undistorted.

This is inspired by the construction of $\pi_1$-injective surfaces seen in \cite{Calegari-Waker-surfsub-rand-15} with a restriction on the number of `emanating' words coming from a trivalent fatgraph structure.
In \cite{Calegari-Waker-surfsub-rand-15} quasi-convex surfaces were built in $\langle S | r \rangle $ where $r$ is a randomly chosen cyclically reduced word of length $l$ with probability going to $1$ as $l \to \infty$. One can think of $\mathcal{G}^0_{m,l}$ as this model of random groups.
These surfaces were then shown to stay quasi-convex under further quotienting to get surface subgroups in $\mathcal{G}^d_{m,l}$ generically for all $0 < d < 1/2$.

We adapt their methods to allow for an undistorted subcomplex in a strictly positive density Gromov model random group to stay undistorted under further quotients.
To get conformal dimension lower bounds we are interested in our subcomplexes being combinatorial round trees as Theorem \ref{t : undistorted round trees lower bound} requires.

To get our lower bound we find round trees $A(V,H)$ built undistorted at an arbitrarily small density with restricted emanating paths to later show it is built undistorted at a higher density.
See Theorem \ref{t : round trees everywhere}.

The following theorem is adapted from \cite[Sec. 8.3]{Mackay-conf-rand-16} to our setup.
For simplification we state the construction for $0 < d < 1/18$.
What is new here is the introduction of additional parameters that restrict the number of extension words (Definition \ref{d : local words}) giving a `thinner' round tree than that seen in \cite{Mackay-conf-rand-16}.

\begin{theorem}(Low density round trees)\label{t : build low d round trees}
For all $m \geq 2$, $0 < d < 1/18$, $0 < \beta \leq 1$, $\eta = d/40$ and $H > 2/d$ there exists combinatorial round trees $A = A((2m-1)^{\beta \eta l}, H) $ built undistorted with overwhelming probability in $\mathcal{G}^d_{m,l}$.

Moreover, one can arrange for a set of extension words $\mathbf{X}_{ 5 + \eta l }$ to exist for $A$ satisfying $\lvert \mathbf{X}_{5 + \eta l} \rvert \leq 2m(2m-1)^4(2m-1)^{\beta \eta l} $ and if bracket paths $B(C)$, $B(C')$ in $A$ have the same label, $C$ and $C'$ have the same boundary label.

\begin{proof}
We follow the proof in \cite[Sec.8.3]{Mackay-conf-rand-16}.
We wish to build a round tree $A = \cup_{\mathbf{a} \in T^{\mathbb{N}}} A_{\mathbf{a}}$ with index set $T = \{1, 2, \ldots, (2m-1)^{\beta \eta l} \}$.
The round tree is built inductively with $A_0 = \{A_\emptyset\}$ being a $2$-cell $C$ with boundary word a relator $r \in \mathcal{R}^d_{m,l}$ with $1 \rightarrow (\partial C)^0 \rightarrow C \rightarrow \text{Cay}^2(G,S,\mathcal{R}^{d}_{m,l})$ combinatorial maps.
Assume we have built $A_n = \cup_{\mathbf{a}_n \in T^n} A_{\mathbf{a_n}}$ with $A_{\mathbf{a_n}}$ as in Definition \ref{d : round tree}.

Let $\mathbf{a}_n \in T^n$.
For some $N_{\mathbf{a}_n} \in \mathbb{N} $, split the outer boundary $E_{\mathbf{a_n}}$ into subpaths $\{ p_i \}_{1 \leq i \leq N_{\mathbf{a}_n}}$, of length $l/H$ for $i \in \{2, \dots, N_{\mathbf{a}_n} - 1\} $ and possibly smaller for $i = 1$ and $i =  N_{\mathbf{a}_n}$, so that the endpoints $p_i^-, p_i^+$ are not local minima of the function $d( 1,\cdot)|{E_{\mathbf{a}_n}}$.
As we allow $p_1$ and $p_{N_{a_\mathbf{n}}}$ to have length $\leq l/H$ by the remainder theorem such a collection of subpaths $\{ p_i \}$ always exists.
We ensure $\{ p_i \}$ partition $E_{\mathbf{a}_n}$ with $p_1^- \in L_{\mathbf{a}_n}$ and $p_{N_{\mathbf{a}_n}}^+ \in R_{\mathbf{a}_n}$.
\newline
\begin{figure}[!ht]
    \centering
    
    \label{f : partition}

\tikzset{every picture/.style={line width=0.75pt}} 

\begin{tikzpicture}[x=0.75pt,y=0.75pt,yscale=-1,xscale=1]

\draw   (386.75,125.13) .. controls (430.57,90.18) and (503.53,64.21) .. (549.72,67.13) .. controls (595.9,70.05) and (597.82,100.76) .. (554,135.71) .. controls (510.18,170.67) and (437.21,196.64) .. (391.03,193.72) .. controls (344.84,190.8) and (342.93,160.09) .. (386.75,125.13) -- cycle ;
\draw  [fill={rgb, 255:red, 0; green, 0; blue, 0 }  ,fill opacity=1 ] (356.8,180.94) .. controls (356.8,180.06) and (357.52,179.34) .. (358.4,179.34) .. controls (359.28,179.34) and (360,180.06) .. (360,180.94) .. controls (360,181.83) and (359.28,182.54) .. (358.4,182.54) .. controls (357.52,182.54) and (356.8,181.83) .. (356.8,180.94) -- cycle ;
\draw  [fill={rgb, 255:red, 0; green, 0; blue, 0 }  ,fill opacity=1 ] (468.6,81.4) .. controls (468.6,80.96) and (468.96,80.6) .. (469.4,80.6) .. controls (469.84,80.6) and (470.2,80.96) .. (470.2,81.4) .. controls (470.2,81.84) and (469.84,82.2) .. (469.4,82.2) .. controls (468.96,82.2) and (468.6,81.84) .. (468.6,81.4) -- cycle ;
\draw  [fill={rgb, 255:red, 0; green, 0; blue, 0 }  ,fill opacity=1 ] (540.6,144.4) .. controls (540.6,143.96) and (540.96,143.6) .. (541.4,143.6) .. controls (541.84,143.6) and (542.2,143.96) .. (542.2,144.4) .. controls (542.2,144.84) and (541.84,145.2) .. (541.4,145.2) .. controls (540.96,145.2) and (540.6,144.84) .. (540.6,144.4) -- cycle ;
\draw    (497.02,72.8) -- (489.73,86.52) ;
\draw    (523,68) -- (515.71,81.71) ;
\draw    (553,136.71) -- (538.29,136.71) ;
\draw    (570.71,120.57) -- (553.43,121.43) ;
\draw    (553.14,67.29) -- (541.05,78.73) -- (539.86,79.86) ;
\draw    (581.29,106) -- (564.57,107.14) ;
\draw    (582,79.43) -- (566.71,90) ;
\draw  [fill={rgb, 255:red, 0; green, 0; blue, 0 }  ,fill opacity=1 ] (480.1,77.27) .. controls (480.1,76.83) and (480.46,76.47) .. (480.9,76.47) .. controls (481.34,76.47) and (481.7,76.83) .. (481.7,77.27) .. controls (481.7,77.72) and (481.34,78.07) .. (480.9,78.07) .. controls (480.46,78.07) and (480.1,77.72) .. (480.1,77.27) -- cycle ;
\draw  [fill={rgb, 255:red, 0; green, 0; blue, 0 }  ,fill opacity=1 ] (490.1,74.4) .. controls (490.1,73.96) and (490.46,73.6) .. (490.9,73.6) .. controls (491.34,73.6) and (491.7,73.96) .. (491.7,74.4) .. controls (491.7,74.84) and (491.34,75.2) .. (490.9,75.2) .. controls (490.46,75.2) and (490.1,74.84) .. (490.1,74.4) -- cycle ;
\draw  [fill={rgb, 255:red, 0; green, 0; blue, 0 }  ,fill opacity=1 ] (502.6,71.4) .. controls (502.6,70.96) and (502.96,70.6) .. (503.4,70.6) .. controls (503.84,70.6) and (504.2,70.96) .. (504.2,71.4) .. controls (504.2,71.84) and (503.84,72.2) .. (503.4,72.2) .. controls (502.96,72.2) and (502.6,71.84) .. (502.6,71.4) -- cycle ;
\draw  [fill={rgb, 255:red, 0; green, 0; blue, 0 }  ,fill opacity=1 ] (514.1,69.02) .. controls (514.1,68.58) and (514.46,68.22) .. (514.9,68.22) .. controls (515.34,68.22) and (515.7,68.58) .. (515.7,69.02) .. controls (515.7,69.47) and (515.34,69.82) .. (514.9,69.82) .. controls (514.46,69.82) and (514.1,69.47) .. (514.1,69.02) -- cycle ;
\draw  [fill={rgb, 255:red, 0; green, 0; blue, 0 }  ,fill opacity=1 ] (529.48,67.15) .. controls (529.48,66.71) and (529.83,66.35) .. (530.28,66.35) .. controls (530.72,66.35) and (531.08,66.71) .. (531.08,67.15) .. controls (531.08,67.59) and (530.72,67.95) .. (530.28,67.95) .. controls (529.83,67.95) and (529.48,67.59) .. (529.48,67.15) -- cycle ;
\draw  [fill={rgb, 255:red, 0; green, 0; blue, 0 }  ,fill opacity=1 ] (543.23,66.65) .. controls (543.23,66.21) and (543.58,65.85) .. (544.03,65.85) .. controls (544.47,65.85) and (544.83,66.21) .. (544.83,66.65) .. controls (544.83,67.09) and (544.47,67.45) .. (544.03,67.45) .. controls (543.58,67.45) and (543.23,67.09) .. (543.23,66.65) -- cycle ;
\draw  [fill={rgb, 255:red, 0; green, 0; blue, 0 }  ,fill opacity=1 ] (560.48,68.9) .. controls (560.48,68.46) and (560.83,68.1) .. (561.28,68.1) .. controls (561.72,68.1) and (562.08,68.46) .. (562.08,68.9) .. controls (562.08,69.34) and (561.72,69.7) .. (561.28,69.7) .. controls (560.83,69.7) and (560.48,69.34) .. (560.48,68.9) -- cycle ;
\draw  [fill={rgb, 255:red, 0; green, 0; blue, 0 }  ,fill opacity=1 ] (572.85,72.77) .. controls (572.85,72.33) and (573.21,71.97) .. (573.65,71.97) .. controls (574.09,71.97) and (574.45,72.33) .. (574.45,72.77) .. controls (574.45,73.22) and (574.09,73.57) .. (573.65,73.57) .. controls (573.21,73.57) and (572.85,73.22) .. (572.85,72.77) -- cycle ;
\draw  [fill={rgb, 255:red, 0; green, 0; blue, 0 }  ,fill opacity=1 ] (583.6,85.15) .. controls (583.6,84.71) and (583.96,84.35) .. (584.4,84.35) .. controls (584.84,84.35) and (585.2,84.71) .. (585.2,85.15) .. controls (585.2,85.59) and (584.84,85.95) .. (584.4,85.95) .. controls (583.96,85.95) and (583.6,85.59) .. (583.6,85.15) -- cycle ;
\draw  [fill={rgb, 255:red, 0; green, 0; blue, 0 }  ,fill opacity=1 ] (583.85,97.65) .. controls (583.85,97.21) and (584.21,96.85) .. (584.65,96.85) .. controls (585.09,96.85) and (585.45,97.21) .. (585.45,97.65) .. controls (585.45,98.09) and (585.09,98.45) .. (584.65,98.45) .. controls (584.21,98.45) and (583.85,98.09) .. (583.85,97.65) -- cycle ;
\draw  [fill={rgb, 255:red, 0; green, 0; blue, 0 }  ,fill opacity=1 ] (577.23,111.28) .. controls (577.23,110.83) and (577.58,110.48) .. (578.03,110.48) .. controls (578.47,110.48) and (578.83,110.83) .. (578.83,111.28) .. controls (578.83,111.72) and (578.47,112.07) .. (578.03,112.07) .. controls (577.58,112.07) and (577.23,111.72) .. (577.23,111.28) -- cycle ;
\draw  [fill={rgb, 255:red, 0; green, 0; blue, 0 }  ,fill opacity=1 ] (572.85,116.53) .. controls (572.85,116.08) and (573.21,115.73) .. (573.65,115.73) .. controls (574.09,115.73) and (574.45,116.08) .. (574.45,116.53) .. controls (574.45,116.97) and (574.09,117.32) .. (573.65,117.32) .. controls (573.21,117.32) and (572.85,116.97) .. (572.85,116.53) -- cycle ;
\draw  [fill={rgb, 255:red, 0; green, 0; blue, 0 }  ,fill opacity=1 ] (564.35,125.28) .. controls (564.35,124.83) and (564.71,124.48) .. (565.15,124.48) .. controls (565.59,124.48) and (565.95,124.83) .. (565.95,125.28) .. controls (565.95,125.72) and (565.59,126.07) .. (565.15,126.07) .. controls (564.71,126.07) and (564.35,125.72) .. (564.35,125.28) -- cycle ;
\draw  [fill={rgb, 255:red, 0; green, 0; blue, 0 }  ,fill opacity=1 ] (557.98,131.15) .. controls (557.98,130.71) and (558.33,130.35) .. (558.78,130.35) .. controls (559.22,130.35) and (559.58,130.71) .. (559.58,131.15) .. controls (559.58,131.59) and (559.22,131.95) .. (558.78,131.95) .. controls (558.33,131.95) and (557.98,131.59) .. (557.98,131.15) -- cycle ;

\draw (380,102.69) node [anchor=north west][inner sep=0.75pt]  [font=\tiny]  {$L_{\boldsymbol{a}_{n}}$};
\draw (479.86,179.4) node [anchor=north west][inner sep=0.75pt]  [font=\tiny]  {$R_{\boldsymbol{a}_{n}}$};
\draw (580.29,63.97) node [anchor=north west][inner sep=0.75pt]  [font=\tiny]  {$E_{\boldsymbol{a}_{n}}$};
\draw (351.43,183.83) node [anchor=north west][inner sep=0.75pt]  [font=\tiny]  {$1$};
\draw (451.33,68) node [anchor=north west][inner sep=0.75pt]  [font=\tiny]  {$p_{1}^{-}$};
\draw (465,57.73) node [anchor=north west][inner sep=0.75pt]  [font=\tiny]  {$p_{1}^{+} =p_{2}^{-} \ $};
\draw (541.6,150.47) node [anchor=north west][inner sep=0.75pt]  [font=\tiny]  {$p_{N_{\boldsymbol{a}_{n}}}^{+}$};

\end{tikzpicture}

\caption{Partitioning the Boundary.}
\end{figure}
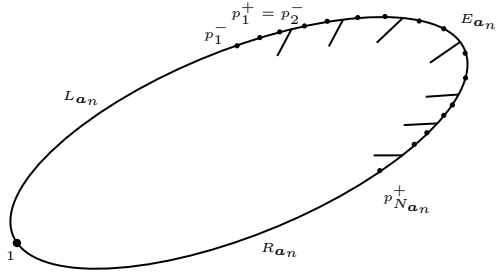

Let $u \in \{p_1^-, p_1^+ , p_2^+ \ldots, p_{N_{\mathbf{a}_n}}^+\}$, this is where we start extending.

Following directly \cite[Sec.8.3]{Mackay-conf-rand-16} by \cite[Prop.8.10]{Mackay-conf-rand-16} with $\lambda = 1/8$, as $ d < 1/16$, there exists a point $u' \in \text{Cay}^{2}(G,S, \mathcal{R}^d_{m,l}) \setminus A_n$ such $d(u,u') = 5$, $d(1, u') = d(1,u) + d(u,u')$ with the geodesic intersection $[u,u'] \cap E_{\mathbf{a_n}} = \{u\}$.
The Lemma \cite[Lem.8.11]{Mackay-conf-rand-16} with $\lambda = 1/8, N = 2 > \frac{\lambda}{\lambda - d}$ gives a word depending on $u$ of length $\lambda l$ which doesn't appear as a subword of any relator in $\mathcal{R}^{d}_{m,l}$ in more than one way.
Note that $'\eta'$ from \cite[Prop.8.10]{Mackay-conf-rand-16} may be taken to be $5$ since for \cite[Lem.8.11]{Mackay-conf-rand-16} $\lambda = 1/8, d < 1/18$ then $2^{'\eta' -2 } > 2 > \frac{\lambda}{\lambda - d}$ which contradicts the conclusion of \cite[Lem.8.11]{Mackay-conf-rand-16}.
We take the first subword of length $5$ giving the labelling of the path $[u,u']$ so the path label will only depend on the initial edge of $[u,1]$ implying no folding occurs.
We make this choice under the assumption $l > 5 \cdot 8$ is sufficiently large.
As there are at most $2m(2m-1)^4$ choice for this initial edge there are at most $2m(2m-1)^4$ words to label $[u,u']$ uniformly over $u$.

The Lemma \cite[Lem.8.8]{Mackay-conf-rand-16} with $\epsilon = \frac{1}{24}$, $'\eta' = \eta l < \frac{1}{16}l - 1 $ implies the existence of $2m-2 \geq 2$ ways to extend $u'$ 
so that all $V = (2m-1)^{\eta l}$ ways of extending (of length $\eta l$) exist and gives geodesics going away from $1$.

At this point we make a change from \cite[Sec.8.3]{Mackay-conf-rand-16} by restricting the possible extensions as follows.
Let $E_{n+1}(u) = E_{n+1}$ be a subset of the end points of these extensions of size $(2m-1)^{\beta \eta l}$ made uniform over $u$ so $\lvert \mathbf{X}_{5 + \eta l} \rvert \leq 2m(2m-1)^{\beta \eta l}  $. 
For each $a \in E_{n+1}(u)$ denote $\gamma_{u',a}$ as the extending path of length $\eta l$ starting at $u'$.
The set of concatenations of $[u,u'] \cdot \gamma_{u,a}$ as $u$ and $a$ vary in their domains determine the $(n,5 + \eta l)$-extension paths. 
Note $u'$ is uniquely determined by $u$ and hence there is one such labelling for the the path $[u,u']$ once $u$ is located.

\begin{figure}[!ht]

    \centering
    \label{f : A combinatorial round tree}

\tikzset{every picture/.style={line width=0.75pt}} 
\begin{tikzpicture}[x=1pt,y=1pt,yscale=-1,xscale=0.8]

\draw    (321.66,132.3) .. controls (333.72,132.25) and (339.62,134.72) .. (379.44,128.82) ;
\draw    (302.76,117.22) .. controls (316.59,130.22) and (328.19,135.27) .. (344.94,149.4) ;
\draw    (348.84,132.52) .. controls (356.97,138.27) and (359.33,139.4) .. (361.89,139.36) ;
\draw    (333.97,140.71) .. controls (342.09,146.45) and (351.05,144.37) .. (353.61,144.34) ;
\draw    (364.97,130.75) .. controls (363.99,131.6) and (367.18,133.48) .. (370.79,134.02) ;
\draw    (344.67,144.6) .. controls (343.68,145.45) and (345.55,146.63) .. (349.16,147.16) ;
\draw    (353.11,135.77) .. controls (345.29,137.62) and (353.57,141.51) .. (357.18,142.04) ;
\draw    (366.11,132.62) .. controls (362.9,134.7) and (364.89,135.91) .. (366.87,136.97) ;
\draw    (371.84,130.08) .. controls (370.86,130.93) and (371.39,131.1) .. (375,131.64) ;

\draw   (64.6,116.8) .. controls (56.98,103.56) and (109.27,83.2) .. (181.4,71.33) .. controls (253.53,59.45) and (318.18,60.56) .. (325.8,73.79) .. controls (333.42,87.03) and (281.13,107.39) .. (209,119.26) .. controls (136.87,131.14) and (72.22,130.04) .. (64.6,116.8) -- cycle ;
\draw    (286.73,49.9) -- (268.13,63.29) ;
\draw    (334.75,51.01) -- (305.87,65.32) ;
\draw    (367.71,66.79) -- (323.02,71.14) ;
\draw    (363.12,83.64) -- (321.86,83.86) ;
\draw    (357.28,100.7) -- (309.08,91.92) ;
\draw    (302.81,117.11) -- (286.69,101.31) ;
\draw    (380.35,103.53) .. controls (392.44,100.04) and (399.44,99.78) .. (436.82,84.91) ;
\draw    (354.72,100.21) .. controls (374.35,103.79) and (388.21,103.38) .. (411.26,106.77) ;
\draw    (407.74,95.85) .. controls (418.42,96.85) and (421.3,96.83) .. (423.85,96.08) ;
\draw    (396.41,104.88) .. controls (407.09,105.88) and (415.17,102.1) .. (417.73,101.35) ;
\draw    (423.15,90.19) .. controls (422.53,90.96) and (426.57,91.14) .. (430.43,90.41) ;
\draw    (408.87,104.07) .. controls (408.25,104.84) and (410.65,104.99) .. (414.51,104.26) ;
\draw    (413.45,96.51) .. controls (406.41,99.83) and (416.44,99.71) .. (420.3,98.98) ;
\draw    (425.11,90.95) .. controls (422.81,93.08) and (425.34,93.21) .. (427.79,93.26) ;
\draw    (429.75,87.82) .. controls (429.13,88.6) and (429.75,88.55) .. (433.61,87.82) ;

\draw    (241,51.28) -- (220.17,66.37) ;
\draw    (344.94,149.4) .. controls (381.53,200.49) and (532.02,155.67) .. (411.26,106.77) ;
\draw [color={rgb, 255:red, 0; green, 0; blue, 0 }  ,draw opacity=0.9 ]   (348.2,146.88) .. controls (402.56,182.7) and (549.29,138.21) .. (414.51,104.26) ;
\draw    (379.44,128.82) .. controls (426.92,110.89) and (490.31,58.05) .. (436.82,84.91) ;
\draw [color={rgb, 255:red, 0; green, 0; blue, 0 }  ,draw opacity=0.34 ] [dash pattern={on 0.84pt off 2.51pt}]  (359.94,141.02) .. controls (417.89,154.73) and (526.73,94.07) .. (421.9,97.73) ;
\draw    (363.12,83.64) -- (419.83,68.24) ;
\draw    (419.83,68.24) .. controls (478.22,43.36) and (488.95,62.63) .. (436.82,84.91) ;
\draw    (367.71,66.79) -- (386.68,58.02) ;
\draw    (334.75,51.01) -- (353.73,42.25) ;
\draw    (286.73,49.9) -- (299.73,39.57) ;
\draw    (386.68,58.02) .. controls (445.06,33.14) and (446.32,52.16) .. (419.83,68.24) ;
\draw    (353.73,42.25) .. controls (412.11,17.37) and (413.16,41.94) .. (386.68,58.02) ;
\draw    (299.73,39.57) .. controls (362.63,3.74) and (363.74,33.68) .. (337.26,49.76) ;
\draw    (258.03,40.28) .. controls (302.79,12.31) and (324.68,21.43) .. (299.73,39.57) ;
\draw    (375,131.64) .. controls (429.77,123.98) and (510.05,68.01) .. (433.61,87.82) ;
\draw    (241,51.28) -- (258.03,40.28) ;
\draw [color={rgb, 255:red, 0; green, 0; blue, 0 }  ,draw opacity=0.3 ]   (353.61,144.34) .. controls (422.24,173.8) and (544.82,112.01) .. (417.73,101.35) ;
\draw [color={rgb, 255:red, 0; green, 0; blue, 0 }  ,draw opacity=0.21 ]   (370.79,134.02) .. controls (439.02,134.76) and (521.73,70.63) .. (430.43,90.41) ;
\draw [color={rgb, 255:red, 0; green, 0; blue, 0 }  ,draw opacity=0.22 ]   (247.32,47.02) -- (275.95,44.39) ;
\draw [color={rgb, 255:red, 0; green, 0; blue, 0 }  ,draw opacity=0.24 ]   (289.82,47.42) -- (314.42,48.43) ;
\draw [color={rgb, 255:red, 0; green, 0; blue, 0 }  ,draw opacity=0.22 ]   (327.85,54.3) -- (356.03,55.31) ;
\draw [color={rgb, 255:red, 0; green, 0; blue, 0 }  ,draw opacity=0.22 ] [dash pattern={on 0.84pt off 2.51pt}]  (367.71,66.79) -- (393.65,69.01) ;
\draw [color={rgb, 255:red, 0; green, 0; blue, 0 }  ,draw opacity=0.21 ] [dash pattern={on 0.84pt off 2.51pt}]  (363.12,83.64) -- (386.9,87.26) ;
\draw [color={rgb, 255:red, 0; green, 0; blue, 0 }  ,draw opacity=0.21 ]   (275.95,44.39) .. controls (335.45,40.35) and (338.14,49.24) .. (314.42,48.43) ;
\draw [color={rgb, 255:red, 0; green, 0; blue, 0 }  ,draw opacity=0.16 ]   (314.42,48.43) .. controls (356.48,48.43) and (373.03,58.95) .. (356.03,55.31) ;
\draw [color={rgb, 255:red, 0; green, 0; blue, 0 }  ,draw opacity=0.25 ]   (356.03,55.31) .. controls (381.08,61.78) and (418.21,72.7) .. (393.65,69.01) ;
\draw [color={rgb, 255:red, 0; green, 0; blue, 0 }  ,draw opacity=0.28 ] [dash pattern={on 0.84pt off 2.51pt}]  (247.32,47.02) -- (262.08,42.37) ;
\draw [color={rgb, 255:red, 0; green, 0; blue, 0 }  ,draw opacity=0.28 ] [dash pattern={on 0.84pt off 2.51pt}]  (289.82,47.42) -- (304.58,42.77) ;

\draw [color={rgb, 255:red, 0; green, 0; blue, 0 }  ,draw opacity=0.28 ] [dash pattern={on 0.84pt off 2.51pt}]  (327.85,54.7) -- (342.61,51.67) ;

\draw  [fill={rgb, 255:red, 0; green, 0; blue, 0 }  ,fill opacity=1 ] (62.38,115.73) .. controls (64.13,115.35) and (66.32,115.31) .. (67.27,115.65) .. controls (68.23,115.99) and (67.59,116.57) .. (65.84,116.95) .. controls (64.1,117.33) and (61.91,117.36) .. (60.96,117.03) .. controls (60,116.69) and (60.64,116.11) .. (62.38,115.73) -- cycle ;
\draw  [fill={rgb, 255:red, 0; green, 0; blue, 0 }  ,fill opacity=1 ] (285.6,100.29) .. controls (286.08,100.11) and (286.73,100.21) .. (287.06,100.52) .. controls (287.4,100.83) and (287.29,101.23) .. (286.81,101.42) .. controls (286.34,101.6) and (285.69,101.49) .. (285.35,101.18) .. controls (285.02,100.87) and (285.13,100.47) .. (285.6,100.29) -- cycle ;
\draw  [fill={rgb, 255:red, 0; green, 0; blue, 0 }  ,fill opacity=1 ] (302.16,116.53) .. controls (302.63,116.35) and (303.29,116.45) .. (303.62,116.76) .. controls (303.96,117.07) and (303.84,117.48) .. (303.37,117.66) .. controls (302.9,117.84) and (302.24,117.74) .. (301.91,117.42) .. controls (301.58,117.11) and (301.69,116.71) .. (302.16,116.53) -- cycle ;
\draw  [fill={rgb, 255:red, 0; green, 0; blue, 0 }  ,fill opacity=1 ] (378.72,128.38) .. controls (379.14,128.11) and (379.8,128.09) .. (380.2,128.33) .. controls (380.6,128.57) and (380.59,128.98) .. (380.17,129.25) .. controls (379.75,129.52) and (379.09,129.54) .. (378.69,129.3) .. controls (378.29,129.07) and (378.3,128.65) .. (378.72,128.38) -- cycle ;
\draw  [fill={rgb, 255:red, 0; green, 0; blue, 0 }  ,fill opacity=1 ] (375.25,130.74) .. controls (375.73,130.56) and (376.38,130.66) .. (376.71,130.97) .. controls (377.05,131.29) and (376.94,131.69) .. (376.46,131.87) .. controls (375.99,132.05) and (375.34,131.95) .. (375,131.64) .. controls (374.67,131.32) and (374.78,130.92) .. (375.25,130.74) -- cycle ;
\draw  [fill={rgb, 255:red, 0; green, 0; blue, 0 }  ,fill opacity=1 ] (370.18,133.45) .. controls (370.66,133.27) and (371.31,133.37) .. (371.65,133.69) .. controls (371.98,134) and (371.87,134.4) .. (371.4,134.58) .. controls (370.92,134.76) and (370.27,134.66) .. (369.93,134.35) .. controls (369.6,134.04) and (369.71,133.64) .. (370.18,133.45) -- cycle ;
\draw  [fill={rgb, 255:red, 0; green, 0; blue, 0 }  ,fill opacity=1 ] (344.34,148.83) .. controls (344.81,148.65) and (345.47,148.75) .. (345.8,149.06) .. controls (346.13,149.38) and (346.02,149.78) .. (345.55,149.96) .. controls (345.08,150.14) and (344.42,150.04) .. (344.09,149.73) .. controls (343.75,149.41) and (343.87,149.01) .. (344.34,148.83) -- cycle ;
\draw  [fill={rgb, 255:red, 0; green, 0; blue, 0 }  ,fill opacity=1 ] (348.2,146.88) .. controls (348.67,146.7) and (349.32,146.81) .. (349.66,147.12) .. controls (349.99,147.43) and (349.88,147.83) .. (349.41,148.01) .. controls (348.93,148.19) and (348.28,148.09) .. (347.94,147.78) .. controls (347.61,147.47) and (347.72,147.07) .. (348.2,146.88) -- cycle ;
\draw  [fill={rgb, 255:red, 0; green, 0; blue, 0 }  ,fill opacity=1 ] (353,143.78) .. controls (353.47,143.6) and (354.13,143.7) .. (354.46,144.01) .. controls (354.8,144.32) and (354.68,144.72) .. (354.21,144.9) .. controls (353.74,145.09) and (353.08,144.98) .. (352.75,144.67) .. controls (352.42,144.36) and (352.53,143.96) .. (353,143.78) -- cycle ;
\draw  [fill={rgb, 255:red, 0; green, 0; blue, 0 }  ,fill opacity=1 ] (361.28,138.8) .. controls (361.76,138.62) and (362.41,138.72) .. (362.75,139.03) .. controls (363.08,139.34) and (362.97,139.75) .. (362.49,139.93) .. controls (362.02,140.11) and (361.37,140.01) .. (361.03,139.69) .. controls (360.7,139.38) and (360.81,138.98) .. (361.28,138.8) -- cycle ;

\draw (101.92,103.5) node [anchor=north west][inner sep=0.75pt]  [font=\tiny,rotate=-359.84]  {${\displaystyle \mathnormal{A_{\boldsymbol{a}_{n}}}}$};
\draw (50.98,109.86) node [anchor=north west][inner sep=0.75pt]  [font=\fontsize{0.29em}{0.35em}\selectfont]  {$1$};
\draw (278.07,104.86) node [anchor=north west][inner sep=0.75pt]  [font=\fontsize{0.29em}{0.35em}\selectfont]  {$u$};
\draw (294.12,120.63) node [anchor=north west][inner sep=0.75pt]  [font=\fontsize{0.29em}{0.35em}\selectfont]  {$u'$};
\draw (305.88,155.08) node [anchor=north west][inner sep=0.75pt]  [font=\tiny]  {$E_{n+1}( u)$};
\end{tikzpicture}
\caption{$A_{\mathbf{a}_{n+1}}$ from $A_{\mathbf{a}_n}$. }
\label{f : round tree}
\end{figure}
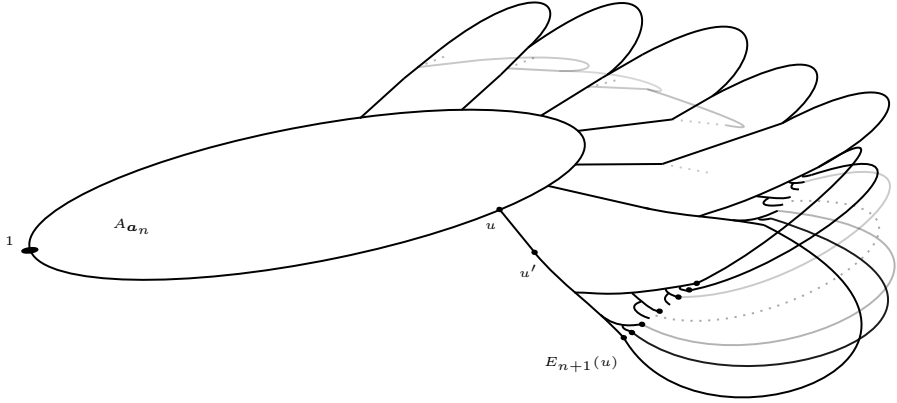

Observe $\lvert \mathbf{X}_{5 + \eta l}(u) \rvert \leq (2m-1)^{\beta \eta l} $ for each $u \in A_n^{(1)}$ by construction as the number of labels for $[u,u']$ is one and $E_{n+1}(u)$'s definition.
Extension paths only exist from the points $u$ we start extending from.
For each possible $2m(2m-1)^4$ choice of labels for $[u,u']$, pick $(2m-1)^{\beta \eta l}$ extension words once and for all. 
This gives us the first part of the second claim.

Set $u_1 = p_i^-$, $u_2 = p_i^+$ for some $i$.
To define $A_{n+1}$ along with a topological embedding we need to identify $2$-cells $C_{a,b} \rightarrow \text{Cay}^2(G,S)$ for $a \in E_{n+1}(u_1)$, $b \in E_{n+1}(u_2)$ where $\gamma^{-1}_{u_1, a} \cdot p_i \cdot \gamma_{u_2, b}$ is the bracket path $B(C_{a,b})$.
$B(C_{a,b})$ will be a $(n, 5 + \eta l)$-bracket of $A$ as Definition \ref{d : bracket}.

By \cite[Prop. 2.7]{Mackay-conf-rand-12}, since $ \lvert \gamma^{-1}_{u_1, a} \cdot p_i \cdot \gamma_{u_2, b} \rvert_S \leq 10 + 2\eta l + l/H < \frac{19}{30}  d l < d l < l/4$, we can find such a $2$-cell $C_{a,b}$.

We can do this procedure for all such points $u_1, u_2$ as $p_i$ varies amongst the partition since all brackets satisfy the above inequality.
We also have the choice, which we make, to force the boundary words $ \partial C_{a,b} = \partial C_{a',b'}$ if the labels, in $S$, of $\gamma^{-1}_{u_1, a} \cdot p_i \cdot \gamma_{u_2, b}$ and $\gamma^{-1}_{u_1', a'} \cdot p_j \cdot \gamma_{u_2', b'}$ are equal for some previously labelled $C_{a',b'}$.
This gives us the second part of the second claim.

This defines $A_{n+1}$ with $\mathbf{a_{n+1}} = (\mathbf{a_n}, a_{n+1})$ for $a_{n+1} \in \{1,\ldots, V\}$.
Inductively, we have constructed a topological immersion $A \rightarrow \text{Cay}^2(G,S)$.

\begin{figure}[!ht]

    \centering
    \label{f : filling in bracket}

\tikzset{every picture/.style={line width=0.75pt}} 

\begin{tikzpicture}[x=0.75pt,y=0.75pt,yscale=-1,xscale=1]

\draw    (259.65,192.94) .. controls (286.45,185.34) and (338.05,121.74) .. (321.25,100.94) ;
\draw    (280.18,179.6) .. controls (305.98,189.6) and (338.38,189.2) .. (380.18,179.6) ;
\draw    (321.45,125.14) .. controls (350.45,120.94) and (385.65,120.14) .. (421.45,125.14) ;
\draw    (380.18,179.6) .. controls (392.03,177.51) and (403.74,173.07) .. (413.73,167.57) .. controls (423.73,162.06) and (456.92,129.13) .. (422.48,125.14) ;
\draw    (421.45,125.14) .. controls (445.65,110.94) and (460.45,70.14) .. (402.45,89.34) ;
\draw    (386.38,219.6) .. controls (421.58,210.4) and (437.58,166) .. (380.18,179.6) ;
\draw   (304.08,150.85) -- (309.35,148.72) -- (308.24,154.29) ;
\draw   (340.53,188.86) -- (335.27,186.71) -- (339.95,183.49) ;
\draw   (367.36,119.45) -- (372.64,121.55) -- (367.98,124.81) ;
\draw  [fill={rgb, 255:red, 0; green, 0; blue, 0 }  ,fill opacity=1 ] (390.47,176.64) .. controls (391.48,175.99) and (392.98,175.65) .. (393.83,175.87) .. controls (394.68,176.09) and (394.55,176.78) .. (393.53,177.43) .. controls (392.52,178.07) and (391.02,178.42) .. (390.17,178.2) .. controls (389.32,177.98) and (389.45,177.28) .. (390.47,176.64) -- cycle ;
\draw  [fill={rgb, 255:red, 0; green, 0; blue, 0 }  ,fill opacity=1 ] (320.45,125.14) .. controls (320.45,124.32) and (321.12,123.66) .. (321.94,123.66) .. controls (322.76,123.66) and (323.42,124.32) .. (323.42,125.14) .. controls (323.42,125.96) and (322.76,126.63) .. (321.94,126.63) .. controls (321.12,126.63) and (320.45,125.96) .. (320.45,125.14) -- cycle ;
\draw  [fill={rgb, 255:red, 0; green, 0; blue, 0 }  ,fill opacity=1 ] (420.68,125.14) .. controls (421.2,124.15) and (422.42,123.34) .. (423.41,123.34) .. controls (424.41,123.34) and (424.8,124.15) .. (424.28,125.14) .. controls (423.77,126.14) and (422.55,126.94) .. (421.56,126.94) .. controls (420.56,126.94) and (420.17,126.14) .. (420.68,125.14) -- cycle ;
\draw  [fill={rgb, 255:red, 0; green, 0; blue, 0 }  ,fill opacity=1 ] (282.18,178.6) .. controls (282.92,179) and (283.2,179.92) .. (282.8,180.67) .. controls (282.4,181.41) and (281.47,181.69) .. (280.73,181.29) .. controls (279.99,180.89) and (279.71,179.96) .. (280.11,179.22) .. controls (280.51,178.48) and (281.43,178.2) .. (282.18,178.6) -- cycle ;

\draw (274.00,169.08) node [anchor=north west][inner sep=0.75pt]  [font=\tiny]  {$u_{1}$};
\draw (328.22,108.93) node [anchor=north west][inner sep=0.75pt]  [font=\tiny]  {$u_{2}$};
\draw (294.25,142.74) node [anchor=north west][inner sep=0.75pt]  [font=\tiny]  {$p_{i}$};
\draw (387.85,181.21) node [anchor=north west][inner sep=0.75pt]  [font=\tiny]  {$a$};
\draw (419.76,127.87) node [anchor=north west][inner sep=0.75pt]  [font=\tiny]  {$b$};
\draw (257.05,135.54) node [anchor=north west][inner sep=0.75pt]  [font=\scriptsize]  {$A_{\boldsymbol{a}_{n}}$};
\draw (357.85,146.74) node [anchor=north west][inner sep=0.75pt]  [font=\scriptsize]  {$C_{a,b}$};

\end{tikzpicture}
\caption{Filling in a bracket path at step $n$.}
\end{figure}
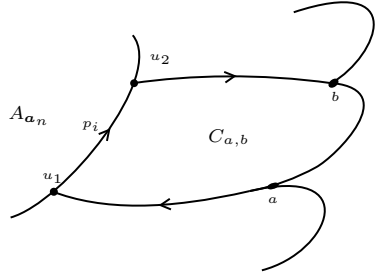

The vertical branching is $V = (2m-1)^{\beta \eta l}$, every $2$-cell $R$ meets at most $l / (\frac{l}{H}) \cdot V = H V$ $2$-cells in $A_{n+1} \setminus A_n$ and so the horizontal branching is at most $H$.
The complex constructed is undistorted and embedded with the proof identical to that found in \cite[Lem. 8.13, Lem. 8.15]{Mackay-conf-rand-16}.
\end{proof}

\end{theorem}

\subsection{Restrictions on the Number of Emanating Words}\label{ss : emanating restrictions}

Our main goal is for built undistorted round trees to stay undistorted under further quotients.
We will later apply a probabilistic diagram rule out statement to `small' van-Kampen diagrams abstracted from \cite[Thm.6.4.1]{Calegari-Waker-surfsub-rand-15}.
These diagrams are controlled in size and shape by the hyperbolic geometry that is generic in Gromov's model $\mathcal{G}^d_{m,l}$. 

In Section \ref{s : build round trees all densities} we will see that if $A$ is distorted at a higher density these diagrams can appear in the Cayley $2$-complex with the possibility of some boundary word labels coming from emanating words (Definition \ref{d : emanating words}) of $A$.
We need to control the number of emanating words in our round tree, which Lemma \ref{l : counting emanating words intrinsic} provides.

\begin{definition}(Emanating paths)\label{d : emanating paths}
Embedded paths $\gamma$ of length $k$ that fit into combinatorial maps $\gamma  \rightarrow \gamma' \rightarrow A^{(1)} \rightarrow \text{Cay}^{2}(\langle S \rangle ,S)$ with $\gamma'(0) = 1$ and $\gamma'$ geodesic in $A^{(1)}$ are called \emph{emanating paths} of length $k$. Denote the set of all such paths by $\Gamma(k)$.
\end{definition}

\begin{definition}(Emanating words)\label{d : emanating words}
 The set $\mathbf{E}_k$ of \emph{emanating words} of length $k$ for a built round tree $A \rightarrow \text{Cay}^2(\langle S \rangle, S )$ is the set of reduced labels of images of emanating paths $\gamma \in \Gamma(k)$.
\end{definition}

For later sections it is useful to denote the set of emanating words of length at most $k$ as $\mathbf{E}_{\leq k}$.

We now compute an upper bound on the number of emanating words in a round tree $A$ given by the parameter set up in Theorem \ref{t : build low d round trees}.
The restrictions on extension and bracket paths from that theorem help strengthen the upper bound.

\begin{lemma}(Counting emanating words)\label{l : counting emanating words intrinsic}
For the hypothesis in Theorem \ref{t : build low d round trees} and for all $\epsilon > 0 $ one can arrange so for each $k \geq 1$ the set of emanating words $\mathbf{E}_k$ satisfy

$$ \lvert \mathbf{E}_k \rvert \leq \frac{l^2k^{\frac{40k}{d l}}}{2} (2m-1)^{(2\beta + \frac{40}{dH} + \epsilon)k }  (2m-1)^{4dl} $$

\noindent
for built round trees $A((2m-1)^{\beta \eta l}, H)$ with overwhelming probability in $\mathcal{G}^d_{m,l}$.

\begin{proof}
Let $w \in \mathbf{E}_k$ with $\gamma =_S w$ the path, and $\gamma'$ the geodesic $\gamma \rightarrow \gamma' \rightarrow A^{(1)}$ with $\gamma'(0) = 1$ and $\rho$ be the path metric on $A^{(1)}$.
By \cite[Lem. 8.13]{Mackay-conf-rand-16} or Theorem \ref{t : build low d round trees}
the geodesics $\gamma'$ follow along the boundary path of $A_\emptyset$, a defining relation, till it hits some point $u_1 \in \partial A_\emptyset$ and branches along an extension path, follows another relation, filled in by some neighbouring bracket path, then hits $u_2 \in \partial A_1$ and branches along an extension path and so on. 
Suppose $\gamma$ passes through $\gamma(0)$ to $u_j, u_{j+1}, \ldots, u_{j'}$ in $\gamma'$ and $u_{j'}$ to $\gamma(1)$.

We begin by counting the number of labels for the subpath of $\gamma$ joining $\gamma(0)$ to $u_j$ and $u_{j'}$ to $\gamma(1)$.
See Figure \ref{f : initial and final path}.
Theorem \ref{t : generic small cancelation} implies the $C'(2d)$ small cancellation hypothesis holds generically in $\mathcal{G}^d_{m,l}$.
This implies that as $\gamma$ follows a defining relation from $\gamma(0)$ once we pass $2dl$ along the boundary path of that relation we have uniquely defined the possible labels for the boundary path of the relation.
Likewise for the path of $\gamma$ from $u_{j'}$ to $\gamma(1)$.

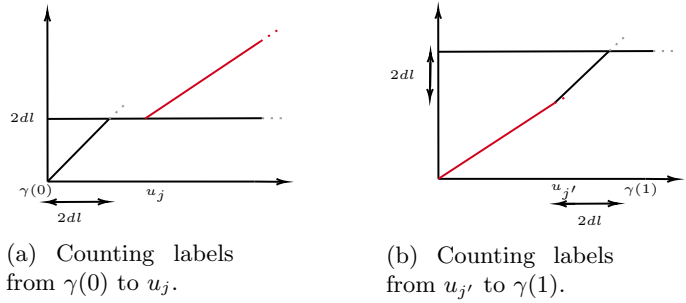
\begin{figure}[!ht]
    \centering
\begin{subfigure}{0.25\textwidth}

    \tikzset{every picture/.style={line width=0.75pt}} 

    \begin{tikzpicture}[x=0.75pt,y=0.75pt,yscale=-0.45,xscale=0.45]

    \draw    (135.43,215.3) -- (367.42,215.36) ;
    \draw    (135.43,215.3) -- (204.4,145.1) ;
    \draw    (134.8,145.1) -- (375.6,144.3) ;
    \draw    (367.42,215.36) -- (400.86,215.66) ;
    \draw [shift={(402.86,215.67)}, rotate = 180.51] [color={rgb, 255:red, 0; green, 0; blue, 0 }  ][line width=0.75]    (10.93,-3.29) .. controls (6.95,-1.4) and (3.31,-0.3) .. (0,0) .. controls (3.31,0.3) and (6.95,1.4) .. (10.93,3.29)   ;
    \draw    (135.43,215.3) -- (135.2,26.7) ;
    \draw [shift={(135.2,24.7)}, rotate = 89.93] [color={rgb, 255:red, 0; green, 0; blue, 0 }  ][line width=0.75]    (10.93,-3.29) .. controls (6.95,-1.4) and (3.31,-0.3) .. (0,0) .. controls (3.31,0.3) and (6.95,1.4) .. (10.93,3.29)   ;
    \draw [color={rgb, 255:red, 208; green, 2; blue, 27 }  ,draw opacity=1 ]   (244.6,144.8) -- (374.2,58) ;
    \draw [color={rgb, 255:red, 155; green, 155; blue, 155 }  ,draw opacity=1 ] [dash pattern={on 0.84pt off 2.51pt}]  (204.4,145.1) -- (226,125.1) ;
    \draw [color={rgb, 255:red, 155; green, 155; blue, 155 }  ,draw opacity=1 ] [dash pattern={on 0.84pt off 2.51pt}]  (375.6,144.3) -- (402.8,144.3) ;
    \draw [color={rgb, 255:red, 208; green, 2; blue, 27 }  ,draw opacity=1 ] [dash pattern={on 0.84pt off 2.51pt}]  (374.2,58) -- (396.8,43.1) ;
    \draw    (138.67,238.33) -- (200.67,238.01) ;
    \draw [shift={(202.67,238)}, rotate = 179.7] [color={rgb, 255:red, 0; green, 0; blue, 0 }  ][line width=0.75]    (10.93,-3.29) .. controls (6.95,-1.4) and (3.31,-0.3) .. (0,0) .. controls (3.31,0.3) and (6.95,1.4) .. (10.93,3.29)   ;
    \draw    (155,238.33) -- (136.67,238.33) ;
    \draw [shift={(134.67,238.33)}, rotate = 360] [color={rgb, 255:red, 0; green, 0; blue, 0 }  ][line width=0.75]    (10.93,-3.29) .. controls (6.95,-1.4) and (3.31,-0.3) .. (0,0) .. controls (3.31,0.3) and (6.95,1.4) .. (10.93,3.29)   ;

    \draw (241.91,220.73) node [anchor=north west][inner sep=0.75pt]  [font=\tiny]  {$u_{j}$};
    \draw (90,135) node [anchor=north west][inner sep=0.75pt]  [font=\tiny]  {$2dl$};
    \draw (100,215) node [anchor=north west][inner sep=0.75pt]  [font=\tiny]  {$ 
    \gamma ( 0)$};
    \draw (143,247.63) node [anchor=north west][inner sep=0.75pt]  [font=\tiny]  {$2dl$};
    \end{tikzpicture}
    \caption{Counting labels from $\gamma(0)$ to $u_j$.}
    \label{fig:1}
\end{subfigure}\hfil 
\begin{subfigure}{0.25\textwidth}

    \tikzset{every picture/.style={line width=0.75pt}} 

    \begin{tikzpicture}[x=0.75pt,y=0.75pt,yscale=-0.45,xscale=0.45]

    \draw    (107.53,215.3) -- (339.52,215.36) ;
    \draw    (237.7,130.3) -- (299,71.83) ;
    \draw    (107.23,72.77) -- (348.03,71.97) ;
    \draw    (339.52,215.36) -- (372.96,215.66) ;
    \draw [shift={(374.96,215.67)}, rotate = 180.51] [color={rgb, 255:red, 0; green, 0; blue, 0 }  ][line width=0.75]    (10.93,-3.29) .. controls (6.95,-1.4) and (3.31,-0.3) .. (0,0) .. controls (3.31,0.3) and (6.95,1.4) .. (10.93,3.29)   ;
    \draw    (107.53,215.3) -- (107.3,26.7) ;
    \draw [shift={(107.3,24.7)}, rotate = 89.93] [color={rgb, 255:red, 0; green, 0; blue, 0 }  ][line width=0.75]    (10.93,-3.29) .. controls (6.95,-1.4) and (3.31,-0.3) .. (0,0) .. controls (3.31,0.3) and (6.95,1.4) .. (10.93,3.29)   ;
    \draw [color={rgb, 255:red, 208; green, 2; blue, 27 }  ,draw opacity=1 ]   (107.53,215.3) -- (237.7,130.3) ;
    \draw [color={rgb, 255:red, 155; green, 155; blue, 155 }  ,draw opacity=1 ] [dash pattern={on 0.84pt off 2.51pt}]  (299,71.83) -- (320.6,51.83) ;
    \draw [color={rgb, 255:red, 155; green, 155; blue, 155 }  ,draw opacity=1 ] [dash pattern={on 0.84pt off 2.51pt}]  (348.03,71.97) -- (375.23,71.97) ;
    \draw [color={rgb, 255:red, 208; green, 2; blue, 27 }  ,draw opacity=1 ] [dash pattern={on 0.84pt off 2.51pt}]  (237.7,130.3) -- (250.1,123.5) ;
    \draw    (98.5,84.7) -- (97.93,123.5) ;
    \draw [shift={(97.9,125.5)}, rotate = 270.84] [color={rgb, 255:red, 0; green, 0; blue, 0 }  ][line width=0.75]    (10.93,-3.29) .. controls (6.95,-1.4) and (3.31,-0.3) .. (0,0) .. controls (3.31,0.3) and (6.95,1.4) .. (10.93,3.29)   ;
    \draw    (98.3,112.5) -- (98.65,75.5) ;
    \draw [shift={(98.67,73.5)}, rotate = 90.54] [color={rgb, 255:red, 0; green, 0; blue, 0 }  ][line width=0.75]    (10.93,-3.29) .. controls (6.95,-1.4) and (3.31,-0.3) .. (0,0) .. controls (3.31,0.3) and (6.95,1.4) .. (10.93,3.29)   ;
    \draw    (244.33,243.33) -- (306.33,243.01) ;
    \draw [shift={(308.33,243)}, rotate = 179.7] [color={rgb, 255:red, 0; green, 0; blue, 0 }  ][line width=0.75]    (10.93,-3.29) .. controls (6.95,-1.4) and (3.31,-0.3) .. (0,0) .. controls (3.31,0.3) and (6.95,1.4) .. (10.93,3.29)   ;
    \draw    (260.67,243.33) -- (242.33,243.33) ;
    \draw [shift={(240.33,243.33)}, rotate = 360] [color={rgb, 255:red, 0; green, 0; blue, 0 }  ][line width=0.75]    (10.93,-3.29) .. controls (6.95,-1.4) and (3.31,-0.3) .. (0,0) .. controls (3.31,0.3) and (6.95,1.4) .. (10.93,3.29)   ;

    \draw (50,90) node [anchor=north west][inner sep=0.75pt]  [font=\tiny]  {$2dl$};
    \draw (230,218) node [anchor=north west][inner sep=0.75pt]  [font=\tiny]  {$u_{j'}$};
    \draw (260,253.63) node [anchor=north west][inner sep=0.75pt]  [font=\tiny]  {$2dl$};
    \draw (312,217.23) node [anchor=north west][inner sep=0.75pt]  [font=\tiny]  {$\gamma ( 1)$};
    \end{tikzpicture}
    \caption{Counting labels from $u_{j'}$ to $\gamma(1)$.}
    \label{fig:2}
\end{subfigure}\hfil 
   \caption{Along the $x$-axis is $\gamma$.
   The $y$-axis counts $\log_{2m-1}$ of the number of labels possible for such a path. }
   \label{f : initial and final path}
\end{figure}

We will bound the rest of the number of labels in two distinct ways which will combine and give our actual bound.

Suppose $\gamma$ passes through $u_i$ and $u_{i+1}$.
It contains an extension word $w_i \in \mathbf{X}_{5 + \eta l}$ from $u_i$ to $e_i$ which by Theorem \ref{t : build low d round trees} satisfies $\lvert \mathbf{X}_{5 + \eta l} \rvert \leq 2m(2m-1)^4(2m-1)^{\beta \eta l}$.
This implies there are at most $2m(2m-1)^4(2m-1)^{\beta \eta l}$ labels for $w_i$ along $\gamma$ then at most $2m-1$ choices for each subsequent letter.

Alternatively, moving forward from $e_i$ the next letter in $w$ comes from the bounding relator of the $2$-cell $C = C^{i}_{e_i, b}$ where $b$ is some endpoint of a bracket path $B(C)$ sharing the emanating word $w_i$ as a label of a subpath.
Each candidate $C$ along with its boundary path, by Theorem \ref{t : build low d round trees}, is determined by the bracket path $B = B(C)$ uniquely.
There are at most $(2m)^2(2m-1)^8(2m-1)^{2\beta \eta l + l/H}$ bracket paths giving at most $(2m)^2(2m-1)^8(2m-1)^{2\beta \eta l + l/H}$ labels for $B(C)$ and hence labels of $\gamma$ from $u_i$ to $u_{i+1}$.

To combine these counting methods and bounds we define $t_i$ to be the point at which they cross.

\begin{figure}[!ht]
    \centering

\tikzset{every picture/.style={line width=0.75pt}} 

\begin{tikzpicture}[x=0.75pt,y=0.75pt,yscale=-1,xscale=1]

\draw    (102.91,246.32) -- (91.73,246.5) ;
\draw    (88.71,248.37) -- (320.71,248.43) ;
\draw    (194.16,201.76) -- (194.16,248.48) ;
\draw    (274.76,119.26) -- (274.53,248.02) ;
\draw    (320.86,119.26) -- (320.71,248.43) ;
\draw    (103,242.5) -- (194.16,201.76) ;
\draw    (194.16,201.76) -- (274.76,119.26) ;
\draw    (274.76,119.26) -- (320.86,119.26) ;
\draw    (89.38,119.82) -- (274.76,119.26) ;
\draw    (320.71,248.43) -- (354.14,248.73) ;
\draw [shift={(356.14,248.75)}, rotate = 180.51] [color={rgb, 255:red, 0; green, 0; blue, 0 }  ][line width=0.75]    (10.93,-3.29) .. controls (6.95,-1.4) and (3.31,-0.3) .. (0,0) .. controls (3.31,0.3) and (6.95,1.4) .. (10.93,3.29)   ;
\draw    (88.71,248.37) -- (89.26,88.14) ;
\draw [shift={(89.27,86.14)}, rotate = 90.19] [color={rgb, 255:red, 0; green, 0; blue, 0 }  ][line width=0.75]    (10.93,-3.29) .. controls (6.95,-1.4) and (3.31,-0.3) .. (0,0) .. controls (3.31,0.3) and (6.95,1.4) .. (10.93,3.29)   ;
\draw [color={rgb, 255:red, 155; green, 155; blue, 155 }  ,draw opacity=1 ] [dash pattern={on 0.84pt off 2.51pt}]  (274.76,119.26) -- (286.72,107.74) ;
\draw [color={rgb, 255:red, 208; green, 2; blue, 27 }  ,draw opacity=1 ]   (88.71,248.37) -- (322.2,85.8) ;
\draw [color={rgb, 255:red, 155; green, 155; blue, 155 }  ,draw opacity=1 ] [dash pattern={on 0.84pt off 2.51pt}]  (320.86,119.26) -- (340,119.25) ;
\draw [color={rgb, 255:red, 208; green, 2; blue, 27 }  ,draw opacity=1 ] [dash pattern={on 0.84pt off 2.51pt}]  (322.2,85.8) -- (342.6,72.2) ;
\draw    (105.49,269.83) -- (191.33,269.67) ;
\draw [shift={(193.33,269.67)}, rotate = 179.89] [color={rgb, 255:red, 0; green, 0; blue, 0 }  ][line width=0.75]    (10.93,-3.29) .. controls (6.95,-1.4) and (3.31,-0.3) .. (0,0) .. controls (3.31,0.3) and (6.95,1.4) .. (10.93,3.29)   ;
\draw    (127.91,269.83) -- (90.29,269.92) ;
\draw [shift={(88.29,269.93)}, rotate = 359.86] [color={rgb, 255:red, 0; green, 0; blue, 0 }  ][line width=0.75]    (10.93,-3.29) .. controls (6.95,-1.4) and (3.31,-0.3) .. (0,0) .. controls (3.31,0.3) and (6.95,1.4) .. (10.93,3.29)   ;
\draw    (103,241.81) -- (102.91,246.32) ;

\draw (0,110) node [anchor=north west][inner sep=0.75pt]  [font=\tiny]  {$ \begin{array}{l}
2\log_{2m-1}( 2m) + 8\\
\ \ \ \ 2\beta \eta l\ +\ l/H
\end{array}$};
\draw (84.57,256.51) node [anchor=north west][inner sep=0.75pt]  [font=\tiny]  {$u_{i}$};
\draw (98,250.4) node [anchor=north west][inner sep=0.75pt]  [font=\tiny]  {$u_{i}^{'}$};
\draw (191.43,255.78) node [anchor=north west][inner sep=0.75pt]  [font=\tiny]  {$e_{i}$};
\draw (311.43,255.4) node [anchor=north west][inner sep=0.75pt]  [font=\tiny]  {$u_{i+1}$};
\draw (270.86,254.97) node [anchor=north west][inner sep=0.75pt]  [font=\tiny]  {$t_{i}$};
\draw (129.1,277.8) node [anchor=north west][inner sep=0.75pt]  [font=\tiny]  {$5 + \eta l$};
\draw (15.89,192.8) node [anchor=north west][inner sep=0.75pt]  [font=\tiny]  {$ \begin{array}{l}
\log_{2m-1} (2m)\ \\
\ \ \ \ \ 4 + \beta \eta l\ 
\end{array}$};

\end{tikzpicture}

\caption{Finding the slope for the number of labels between $u_i$ and $u_{i+1}$. 
Along the $x$-axis is $\gamma$.
The $y$-axis counts $\log_{2m-1}$ of the number of labels possible for such a path.}
\label{f : labels of u to next}

\end{figure}
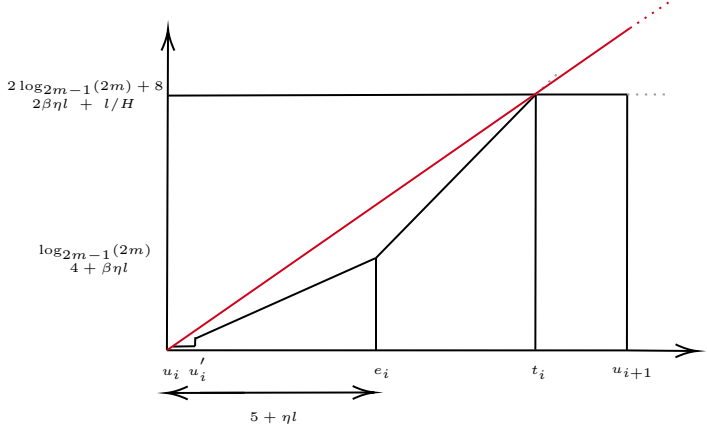

Using elementary Euclidean geometry to find the slope of the red line in Figure \ref{f : labels of u to next} we conclude $ \rho(u_i,t_i) \leq  \log_{2m-1}(2m) + 9 + (\beta + 1)\eta l + l/H $ and that there are at most

$$ (2m-1)^{ \frac{2 M l^{-1} + 2 \beta \eta + 1/H}{(M + 9)l^{-1} + (\beta + 1) \eta + 1/H} \rho (u_{i}, u_{i+1}) } $$
choices of labels of words joining $u_i$ to $u_{i+1}$ along $\gamma'$ for all such $j \leq i < j'$ where $M = \log_{2m-1}(2m)$.

To obtain the precise exponents recall $\eta = d/40$, $H > 2/d$ and take $l$ sufficiently large so that
$$ \frac{2 M l^{-1}+ 2d \beta + 40/H}{\beta d + d + 40/H} < \frac{2Ml^{-1} + 2d \beta + 40/H}{d} \leq 2\beta + \frac{40}{dH} + \epsilon $$
which results in at most $(2m-1)^{(2 \beta + \frac{40}{d H} + \epsilon) \rho(u_i, u_{i+1})}$ choices of labels of words joining $u_i$ to $u_{i+1}$.
We multiply these separate counts across $j < i < j'$ to obtain an upper bound on the number of emanating words of length $k$ passing through a fixed sequence of $ \gamma(0), u_j, u_{j+1}\ldots, u_{j'}, \gamma(1)$:

$$ (2m-1)^{ (2\beta + \frac{40}{dH} + \epsilon) \rho (u_j, u_{j'}) } (2m-1)^{4dl}. $$

Since $ \rho (u_i, u_{i+1}) \geq 5 + \eta l \geq \eta l$ there are at most $k^{\frac{k}{ \eta l}} + 1 \leq 2k^{\frac{k}{ \eta l}} $ positions along $\gamma$ for the points $ u_j, u_{j+1} \ldots, u_{j'}$ to be placed.
As we are geodesic we have $\max\{\rho(\gamma(0), u_j), \rho(\gamma(1), u_{j'})\} \leq l/2$ and so there are at most $l^2/4$ positions for $\gamma(0)$ and $\gamma(1)$ to be placed in $A^{(1)}$ once the $\{u_i\}$ have been placed along $\gamma$.

These three bounds multiplied together conclude the proof.
\end{proof}
\end{lemma}

\section{Probabilistic Diagram Rule Out}\label{s : prob diagram rule out}

In this section we will detail probabilistic methods of ruling out certain diagrams using the language of Ollivier \cite{ollivier-survey} and adapting methods inspired by those of Calegari and Walker \cite{Calegari-Waker-surfsub-rand-15}.

For the remainder of this section we fix parameters $m \geq 2$, $l \geq 2$ and $0 < d < 1/2$ for the Gromov-density model $\mathcal{G}^d_{m,l}$.
Also $S$ will be a finite alphabet of size $m$ used as our generators of the groups defined by presentations.

\subsection{Restricted Abstract Diagrams}

The following definition is similar to that of Ollivier \cite[Sec.V.b.Def.57]{ollivier-survey} but differs in the sense that we allow for boundary edges to be `restricted' in an alphabet. 
For a $2$-complex $X$ we denote $\lvert X \rvert$ as the number of faces or $2$-cells.

\begin{definition}(Restricted abstract Diagrams)\label{d : abstract-vkmp}
A planar $2$-complex $X$ will be called a \emph{restricted abstract (van-Kampen) diagram in $S$} if we have the following additional data:

\begin{itemize}
    \item[(1)] an integer $1 \leq n \leq \lvert X \rvert$ called the \emph{number of distinct relators} in $X$, sometimes denoted $n(X)$,
    \item[(2)] a surjective map $b_X$ from the faces of $X$ to the set $\{1, \ldots, n\}$; a face $f$ with image $i$ is said to \emph{bear relator $i$},
    \item[(3)] for each face $f$, a distinguished edge $e(f)$ on on $\partial f$ and an orientation of the plane $+1 $ or $ -1$. Along the boundary path $p$ starting at $e(f)$ following the orientation, we call the $k$th edge of $p$ the $k$th edge of $f$,  
    \item[(4)] a map $r_X$ from the boundary edges $e \in \partial X$ to $\{0,1\}$ that can only take the value $1$ if there exists a face $f$ such that $e \in \partial f$.
    If $r(e) = 1$ we say $e$ is \emph{restricted} and \emph{not-restricted} when $r(e) = 0$,
    \item[(5)] an alphabet $S$ such that for each $e \in r_X^{-1}(1)$ we associate a label $\text{lab}(e) \in S \cup S^{-1}$.
    The function $\text{lab} : r_X^{-1}(1) \rightarrow S \cup S^{-1} $ is called \emph{the boundary restriction function}.
\end{itemize}
A restricted abstract diagram in $S$ will be referred to by the triple $(X, r_X, \text{lab})$ where $X$ is an abstract diagram, like Ollivier, satisfying (1)-(3).
\end{definition}

Note that a van-Kampen diagram $D \rightarrow \text{Cay}^{(2)}(G,S,R)$ defines many distinct restricted abstract diagrams  in $S$ with $w$ a word in $S \cup S^{-1}$ of length at most $\lvert \partial D \rvert$ by forgetting which relators label each face and some subset of the boundary edges that are not-restricted defining $r_X$. 

A planar $2$-complex satisfying (1)-(3) of Definition \ref{d : abstract-vkmp} is called an \emph{abstract diagram}.
A restricted abstract diagram $(X, r_X, \text{lab})$ in $S$ gives an abstract diagram $X$ by forgetting $r_X$ and $\text{lab}$.
An abstract diagram is said to be \emph{reduced} if no edge is adjacent to two faces bearing the same relator with opposite orientations such that the edge is the $k$th edge of both faces and there exists no degree $1$ vertices in $\partial X$.
The restricted diagram $(X, r_X, \text{lab})$ is \emph{reduced} if $X$ is a reduced abstract diagram.

We want to count the number of abstract diagrams for a fixed diagram of a particularly important combinatorial type for our context.
Since we are in the Gromov model $\mathcal{G}^d_{m,l}$ we restrict to diagrams where each face is an $l$-gon.
An $l$-gon is a $2$-cell with boundary length $l$.

\begin{lemma}(Counting abstract diagrams)\label{l : counting the diagrams}
Let $C > 0$ and $l \geq 3$ an integer.
The number of contractible, connected reduced abstract diagrams $X$ with $l$-gon faces and every edge being in the boundary of some face satisfying $\lvert X \rvert < C$ is at most

$$N(C,l) < P(C)l^{4C}
,$$

\noindent
with $P$ independent of $l$.

\begin{proof}
The one skeleton $X^{(1)}$ of $X$ can be seen as a planar graph with vertices being elements of $X^{(0)}$ of degree at least $3$ in $X^{(1)}$ and edges following paths in $X^{(1)}$ passing through elements of $X^{(0)}$ of exactly degree $2$ in $X^{(1)}$ connecting them.
As $X$ is reduced there are no degree $1$ vertices in $X^{(1)}$ meaning this graph covers $X^{(1)}$.

In such a graph $\Gamma = (V,E)$ by the Handshake Lemma, $2\lvert E \rvert = \sum_{v\in V} \text{deg}(v) $ $\geq 3\lvert V \rvert$, applying Euler's formula to the finite planar graph $\Gamma$ we obtain $\lvert E \rvert < \lvert V \rvert + \lvert F \rvert \leq \frac{2}{3} \lvert E \rvert + \lvert F \rvert$, hence $\lvert E \rvert < 3 \lvert F \rvert < 3C$.

Note, $\lvert V \rvert < 2C$.
Let $P(C)$ be the number of finite planar graphs $\Gamma$ with at most $2C$ vertices and at most $C$ faces and vertex degree everywhere at least $3$. 

As we assume each edge is in the boundary of a face, an $l$-gon, there are at most $l^{3C}$ choices for edge lengths determining the one skeleton $X^{(1)}$ with the vertices of degree 2.

Since $n \leq \lvert X \rvert < C$, for each face $f \in X^{(2)}$ there are $2lC$ choices for the data of $b_X(f)$, $e(f)$ and its orientation. Hence, $(2Cl)^C$ choices overall.
We absorb the constants not depending on $l$ into the function $P$.
\end{proof}
\end{lemma}

Going from an abstract diagram to a van Kampen diagram we need the notion of a (partial) filling by a group.

\begin{definition}(Filling a diagram)\label{d : filling-diagram-words}
An $m$-tuple $(w_1, \ldots, w_m)$ of cyclically reduced words in $S$ is said to $m$-\emph{partially fill} an abstract diagram $(X, r_X, \text{lab})$ restricted in $S$ with $m \leq n(X) = n$ if the following holds:

\begin{itemize}
    \item[(1)] for each two faces $f_1$ and $f_2$ bearing relators $i_1 \leq m$ and $i_2 \leq m$, such that the $k_1$th edge of $f_1$ is equal to the $k_2$th edge of $f_2$, then the $k_1$th letter of $w_{i_1}$ and the $k_2$th letter of $w_{i_2}$ are inverses (when the orientations of $f_1$ and $f_2$ agree) or equal (when the orientations disagree), and
    \item[(2)]for each boundary face $f$ bearing relator $i \leq m$ if the $k$th edge $e$ of $f$ is restricted, $r_X(e) = 1$, then the $k$th letter of $w_i$ is $lab(e)$.
\end{itemize}
    
An $n(X)$-tuple of cyclically reduced words that partially fills an abstract diagram $X$ is said to \emph{fill} the diagram.
\end{definition}

Given a group presentation $G = \langle S | R \rangle$ where $R$ is a set of cyclically reduced words we say an abstract diagram $X$ is (partially) fillable in $G$ if it is (partially) filled by a tuple of distinct words from $R$.

To talk about the interaction of faces along internal edges in an abstract diagram we define the following.
Supose we have an abstract diagram $X$ with all faces $l$-gons and for $1 \leq i \leq n$ let $m_i$ denote the number of times relator $i$ appears in $D$.
Up to reordering suppose $m_1 \geq \ldots \geq m_n$.
Set $m_{n+1} = 0$.

For $1 \leq i_1, i_2 \leq n$ and $1 \leq k_1, k_2 \leq l$ say that $(i_1,k_1) > (i_2, k_2)$ if $i_1 > i_2$, or if $i_1 = i_2$ and $k_1 > k_2$.
Let $e$ be an edge of $X$ adjacent to faces $f_1$ and $f_2$ bearing relators $i_1$ and $i_2$, which is the $k_1$th edge of $f_1$ and the $k_2$th edge of $f_2$.
Say $e$ \emph{belongs} to $f_1$ if $(i_1, k_1) > (i_2, k_2)$, and belongs to $f_2$ if $(i_2, k_2) > (i_1, k_1)$, so that an edge belongs to the second face it meets.

So far belonging is not defined for boundary edges. We say a boundary edge $e \in \partial X$ belongs to face $f \in X^{(2)}$ if $e \in \partial f \subset \partial X$ and $r_X(e) = 1$.
Note by Definition \ref{d : abstract-vkmp} if $r_X(e) = 1$ for some edge $e$ then such a face $f$ must exist for the edge.
In other words, every restricted edge of $(X, r_X, \text{lab})$ belongs to the boundary of some face of $X$.

If $X$ is reduced and fillable, each internal edge belongs to some face.
When $(i_1, k_1) = (i_2, k_2) $ then either the two faces have opposite orientations and then $X$ is not reduced, or they have the same orientation and the diagram is never fillable since a letter would have to be its own inverse.

Let $E(f)$ be the number of edges belonging to the face $f$.
For a given abstract diagram $X$ we define the degree of constraint $d_c(X)$ as:

$$ d_c(X) = \sum_{f \text{ face of } X}E(f) .$$

The degree of constraint contains information on how different we can treat filling $X$ to independently labelling the edges and faces of $X$: the larger $d_c(X)$ is, the more `constrained' the diagram is and the less likely the abstract diagram $X$ can be filled by a typical group in the model $\mathcal{G}^d_{m,l}$.
See Theorem \ref{t : restricted abstract diagram rule out} to see a precise mathematical statement.

For the Gromov density model as already mentioned we will use $\mathcal{R}^d_{m,l}$ to denote the random variable of uniformily chosen cyclically reduced words of length $l$ chosen in the appropriate alphabet of size $m$ defining the group.

The following lemma follows Ollivier's Lemma \cite[Sec.V.b.Lem.59]{ollivier-survey} with the caveat that we also deal with restricted boundary edges and adds details regarding cyclically reduced words.

We count the number of cyclically reduced words in a free group by a result of Rivin \cite{Rivin-growth-free-groups}.

\begin{theorem}(\cite[Thm. 1.1]{Rivin-growth-free-groups}, Rivin's count)\label{t : counting cyclically reduced words}
The number $N_l$ of cyclically reduced words of length $l$ in $F_m$ satisfies
$$ N_l = (2m-1)^l + 1 + (m-1)(1 + (-1)^l).$$

\end{theorem}

\begin{lemma}(Inductive filling lemma)\label{l : inductive filling lemma}
Let $(X, r_X, \text{lab})$ be a reduced abstract diagram restricted in $S$.
Set $P_i$ as the probability that there exists an $i$-tuple of words partially filling $X$ by $\mathcal{R}^d_{m,l}$ and $p_i$ as the probability that $i$ uniformly random cyclically reduced words in $S$ partially fill the diagram $X$. 
Then,

\begin{alignat*}{3}
p_i \leq 2m (2m-1)^{ - E_i}p_{i-1} \text{ and }
P_i \leq (2m-1)^{idl}p_i
\end{alignat*}

\noindent
where $E_i = \max\{E(f) : f \text{ is a face bearing } i \}$ and $p_0 = 1$.

\begin{proof}
Suppose $i-1$ cyclically reduced words $(w_j)$ in $S$ are given that partially fill the diagram $X$.

Let $e$ be an edge on the boundary of a face $f$ that bears relator $i$ 
in the restricted abstract diagram $X$.
These spaces fit in the following sequence of combinatorial maps between complexes

$$ e \rightarrow \partial f \rightarrow f \rightarrow X$$
\noindent
considering $\partial f$ as the oriented edge path and $f$ as an abstract subdiagram of $X$.

Suppose $f$ attains the maximum of $E_i = E(f)$.
$\partial f $ is oriented so we can order the edges as $e_1 \leq e_2 \leq \cdots \leq e_l$ where $e_1 = e(f)$. 
Let $1 \leq k \leq l$, and we compute the number $n_{i,k}$ which denotes the number of possible symbols in $S \cup S^{-1}$ that we can associated to the $kth $ edge of $f$ so that the new `partially filled' diagram is not reduced given the assignment from belonging under the inductive hypothesis and the previous $1 \leq k' < k$ symbols associated to $e_{k'}$. 

Depending on which of $e_l, e_1$ and $e_2$ belong, $n_{i,1} \in \{1,2m-2, 2m-1, 2m \}$ which is $\leq 2m$.
For $2 \leq k \leq l$, depending if $e_k, e_{k+1}$ belong ($e_{l+1} = e_1$) we have $n_{i,k} \in \{1,2m-2, 2m-1\} $ which is $\leq 2m-1$.
This gives: 
\begin{alignat*}{3}
p_i &\leq \;\mathbb{P}(f \text{ is partially fillable with } w_i \;|\; w_i \text{ is a randomly chosen } \\  & \; \; \; \; \; \; \text{cyclically reduced word of length  } l)p_{i-1}, \\
 &\leq \frac{\prod_{k=1}^{l}n_{i,k}}{N_l}p_{i-1}  \leq \frac{2m(2m-1)^{l - E_i}}{(2m-1)^l}p_{i-1} = 2m(2m-1)^{-E_i}p_{i-1}.
\end{alignat*}

The second to last inequality uses the observation $N_l \geq (2m-1)^l$ from Theorem \ref{t : counting cyclically reduced words}.

To conclude the lemma we observe $P_i \leq \lvert \mathcal{R}^d_{m,l } \lvert^i p_i = (2m-1)^{idl}p_i$.
\end{proof}
\end{lemma}

\begin{theorem}(Restricted boundary diagram rule out)\label{t : restricted abstract diagram rule out}
Let $m \geq 2$, $0 < d < 1/2$, for $l \geq 2$ and $(X, r_X, \text{lab})$ be a reduced abstract diagram restricted in $S$ with each face an $l$-gon.
If $2\lvert r_X^{-1}(1) \rvert \geq \lvert \partial X \rvert$ then,
$$\mathbb{P}((X, r_X, \text{lab}) \text{ is fillable by } \mathcal{R}^{d}_{m,l} ) \leq  2m(2m-1)^{(d - 1/2)l }.$$
 
\begin{proof}
We follow the framework of Ollivier \cite[Prop. 58]{ollivier-survey}. 
Denote $I$ as the set of internal edges of $X$ i.e. those not on the boundary $\partial X$.
We start with the verifiable statements: 
$$
l \lvert X \rvert \leq \lvert \partial X \rvert + 2\lvert I \rvert \text{ and } d_c(X) = \lvert I \rvert + \lvert r_X^{-1}(1) \rvert
$$
as every edge $e \in r_X^{-1}(1) \cup I$ belongs to a face.
Since by hypothesis, $\lvert r_X^{-1}(1) \rvert \geq \lvert \partial X \rvert/2$, we have:
\begin{align*}
    0 & \geq l \lvert X \rvert - \lvert \partial X \rvert - 2 \lvert I \rvert \\
      & = l \lvert X \rvert - \lvert \partial X \rvert + 2 \lvert r_X^{-1}(1) \rvert - 2d_c(X)\\
      & \geq l \lvert X \rvert - 2d_c(X).
\end{align*}

\noindent
We expand:
\begin{alignat*}{3}
  0 &\geq l\lvert X \rvert -2\sum_{f \text{ face of } X}E(f) \\
   &\geq    l \lvert X \rvert - 2\sum_{i = 1}^{n}m_iE_i
\end{alignat*}
\noindent
where $E_i = \max\{E(f) : f \text{ bears }i\}$ as above.
Following the notation of Lemma \ref{l : inductive filling lemma} and its conclusion of $ - E_i \geq \log_{2m-1}p_i - \log_{2m-1}2mp_{i-1}$ we obtain:
$$ 0 \geq l \lvert X \rvert + 2 \sum_{i=1}^n m_i (\log_{2m-1}p_i - \log_{2m-1}2mp_{i-1}).$$
\noindent
We have, recalling $m_{n+1} = 0$,
$$ 0 \geq l \lvert X \rvert + 2\sum_{i=1}^{n}(m_i - m_{i+1})(\log_{2m-1}P_i - idl) - \; 2\Big(\sum_{i = 1}^n m_i \Big) \log_{2m-1}2m. $$
\noindent
Keep in mind we are taking logarithms of probabilities.
Set $P = \min P_i$, so
$$ 0 \geq l \lvert X \rvert (1 - 2d) + 2 m_1\log_{2m-1}P - 2 \lvert X \rvert\log_{2m-1}2m ,$$

\noindent
observing $m_1 \leq \lvert X \rvert$,
\begin{alignat*}{3}
  0 &\geq l \lvert X \rvert (1 - 2d) + 2 \lvert X \rvert \log_{2m-1}(\frac{P}{2m}).
\end{alignat*}

As $(X, r_X, \text{lab})$ being filled by $\mathcal{R}^{d}_{m,l}$ happens if and only if $(X, r_X, \text{lab})$ is partially filled by $\mathcal{R}^{d}_{m,l}$ for all $i \leq n$, so
\begin{equation*}
   \mathbb{P}((X, r_X, \text{lab}) \text{ is fillable by } \mathcal{R}^{d}_{m,l} ) \leq P \leq 2m(2m-1)^{(d - 1/2)l} .
\end{equation*}
\end{proof}
\end{theorem}

\section{Building Round Trees at High Densities}\label{s : build round trees all densities}

In this section we show the round trees constructed in Section \ref{s : round trees} for particular parameters are undistorted under the natural quotient map by using the hyperbolic geometry of random groups in $\mathcal{G}^{d}_{m,l}$ for $0 < d < 1/2$ and the probabilistic methods to rule out certain diagrams developed in Section \ref{s : prob diagram rule out}.

Suppose a combinatorial round tree $A$ has been built in a finitely presented group $G = \langle S | R \rangle$ with combinatorial map $\varphi: A \rightarrow \text{Cay}^{(2)}(G,S,R)$.
Let $G' = \langle S | R \cup R' \rangle$ be another finitely presented group with a quotient map $\pi : G \rightarrow G'$ which gives a natural extension to Cayley $2$-complexes, hence there is a combinatorial map $\pi \circ \varphi : A \rightarrow \text{Cay}^{(2)}(G', S, R \cup R')$.

For the density of the respective models we use the subscript $t$ to stand for `target' and $s$ for `source'.
We will use the notation $\pi^{d_t}_{d_s}$ to denote the quotient map $ \pi^{d_t}_{d_s} : \langle S |  \mathcal{R}^{d_s}_{m,l} \rangle \rightarrow \langle S | \mathcal{R}^{d_t}_{m,l} \rangle$, considering $ \mathcal{R}^{d_s}_{m,l} \subset \mathcal{R}^{d_t}_{m,l}$, along with its natural extension to Cayley-$2$-complexes.
For $d \in \{d_s, d_t \}$ let $K_d$ denote the Cayley $2$-complex $\text{Cay}^{(2)}(\langle S | \mathcal{R}^{d}_{m,l} \rangle , S)$.

Parameters $m, l \geq 2$ integers and density $1/8 \leq d_t < 1/2$ set our hypothesis for the model $\mathcal{G}^{d_t}_{m,l}$.
The following choices are made to be concrete but are in no way canonical.
Let $\epsilon = 1/2 - d_t < 1$, we will choose $0 < d_s < 1/18$ defining a lower density model and parameters $0 < \beta < 1$,  $\eta = d_s/40$ and $H > 2/d_s$ defining a combinatorial round tree built in $\mathcal{G}^{d_s}_{m,l}$.
Our choices are:
\begin{equation*}
\epsilon = 1/2 - d_t, d_s = 10^{-7} \epsilon^3, \beta = 10^{-7}\epsilon^2, \eta = 4^{-1} \cdot 10^{-8} \epsilon^3 \text{ and } H = 40 \cdot 10^{14} \epsilon^{-5}.   
\end{equation*}

These parameters give a lower density model of groups for which we can find a combinatorial round tree $A$ depending on the density $d_t$ parameter defining $\mathcal{G}^{d_t}_{m,l}$,
as by design they satisfy the hypothesis of Theorem \ref{t : build low d round trees} and Lemma \ref{l : counting emanating words intrinsic}.

Note our choices have put a restriction on $d_t$ as implicitly we require $d_s \leq d_t$.
We are still in the framework of assuming $d_t \geq 1/8$ and in principle we can find `thin' round trees for $0 < d_t < 1/8$ but one should not expect an improvement of the lower bounds on $\text{Confdim}( \partial_\infty G)$ achieved in \cite{Mackay-conf-rand-16}.

Fix $A = A((2m-1)^{\beta \eta l}, H) $ to be the round tree given by Theorem \ref{t : build low d round trees}.
Let $C > 0$ and $l \geq 2$ by Lemma \ref{l : counting emanating words intrinsic} the set of emanating words $\mathbf{E}_{C l}$ of $A$ satisfies
\begin{equation}\label{eq : count bound}
     \lvert \mathbf{E}_{C l} \rvert \leq \frac{l^2(Cl)^{\frac{40C}{d_s}}}{2} (2m-1)^{4 \cdot 10^{-7}\epsilon^2(C+ \epsilon)l} 
     \end{equation}
\noindent
with overwhelming probability in $\mathcal{G}^{d_s}_{m,l}$.
For the result in the above line we take the epsilon in Lemma \ref{l : counting emanating words intrinsic} to be $10^{-7} \epsilon^2$.

We aim to prove the following result.

\begin{theorem}(High density round trees)\label{t : round trees everywhere}
For all $1/8 \leq d_t < 1/2$, $m \geq 2$ and $d_s, \beta, \eta, H$ as above.
The round tree $A \rightarrow K_{d_s} \rightarrow K_{d_t}$ is an undistorted round tree with overwhelming probability in $\mathcal{G}^{d_s}_{m,l}$ and $\mathcal{G}^{d_t}_{m,l}$.
\end{theorem}

We split the proof into a collection of lemmas themed by subsections. 
For the remainder of this section we assume the hypothesis of Theorem \ref{t : round trees everywhere} and model parameters in the preceding paragraphs of this section.

\subsection{Fillability of Diagrams Along Low-density Paths}

Indispensable to our proofs is a way of estimating the probability of filling diagrams of a particular combinatorial type along subpaths from the lower density complex $K_{d_s}$.
This will allow us to rule out the existence of certain diagrams we come across.

We can think of the choice of the random list of relators $\mathcal{R}^{d_t}_{m,l}$ as happening in two steps: first $\mathcal{R}^{d_s}_{m,l}$ is chosen, then independently $\mathcal{R}^{d_t}_{m,l} \setminus \mathcal{R}^{d_s}_{m,l}$ is chosen.
The following lemma shows certain diagrams, if not fillable by $\mathcal{R}^{d_s}_{m,l}$, are unlikely to be fillable by $\mathcal{R}^{d_t}_{m,l}$.

\begin{lemma}(Fillability of diagrams along paths)\label{l : filling diagrams along paths}
Let $m \geq 2$, $0 < d_s < d_t < 1/2$, $C, C' > 0$ be constants.
For $l \geq 2$ suppose the random presentation $\mathcal{R}^{d_s}_{m,l}$ is chosen and a collection $\mathbf{W}_{d_s,l}$ of path labels in $K_{d_s}$ of length $\leq C' l$, $\lvert \mathbf{W}_{d_s,l} \rvert \leq P(l) (2m-1)^{gl}$
for some $g \geq 0$ constant and $P(l)$ polynomial in $l$.
Also suppose $d = 1/2 - d_s C - g - d_t$ satisfies $d > 0$.

Then, there exists a polynomial $Q(l)$ such that the following event for $\mathcal{R}^{d_t}_{m,l}$ given $\mathcal{R}^{d_s}_{m,l}$ holds with probability $\leq Q(l)(2m-1)^{-d l}$:

There exists $(X, r_X, \text{lab})$ a reduced, connected, contractible abstract diagram restricted in $S$ with,
\begin{itemize}
    \item[(1)] $r_X^{-1}(1) \subset \partial X$ a subpath labelled by a word in $\mathbf{W}_{d_s,l}$,
    \item[(2)] each edge of $X$ is in the boundary of some face,
    \item[(3)] $2 \lvert r_X^{-1}(1) \rvert \geq \lvert \partial X \rvert$, $\lvert X \rvert \leq C$, $\lvert \partial X \rvert \leq C' l$,
    \item[(4)] $(X, r_X, \text{lab})$ is fillable by $\mathcal{R}^{d_t}_{m,l}$ such that $\partial X \setminus r_X^{-1}(1)$ is a geodesic in $K_{d_t}$ and at least one relator of $X$ has label in $\mathcal{R}^{d_t}_{m,l} \setminus \mathcal{R}^{d_s}_{m,l}$.
\end{itemize}

\begin{remark}
We can take $Q(l) = 2^{2C}N(C,l) P(l)$, with $N(C,l)$ given by Lemma \ref{l : counting the diagrams}.
\end{remark}

\begin{proof}
Suppose that a particular $(X, r_X, \text{lab})$ is fillable by $\mathcal{R}^{d_t}_{m,l}$ satisfying $(1)-(4)$.
By Lemma \ref{l : counting the diagrams} there are $N(C,l)$ choices of $X$, $(C' l)^2$ choices for $r_X$ and $P(l)(2m-1)^{gl}$ choices of $\text{lab}$.

Consider a maximal connected subdiagram $D$ of $X \rightarrow K_{d_t}$ with boundary labels of faces contained in $\mathcal{R}^{d_t}_{m,l} \setminus \mathcal{R}^{d_s}_{m,l}$.
As $\lvert X \rvert \leq C$ there are at most $2^C$ choices of faces for $X$, at most $\lvert \mathcal{R}^{d_s}_{m,l} \rvert^C = (2m-1)^{d_s C l }$ choices of boundary labels for these faces and $2^C$ ways to orient the boundary labels around the distinguished faces.
There are also at most $2^C$ positions of where $D$ is as a subdiagram of $X$.
This implies there are at most $2^{2C} (2m-1)^{d_s C l}$ possible diagrams $D$.

\begin{figure}[!ht]
    \centering
    \includegraphics[scale = 0.25]{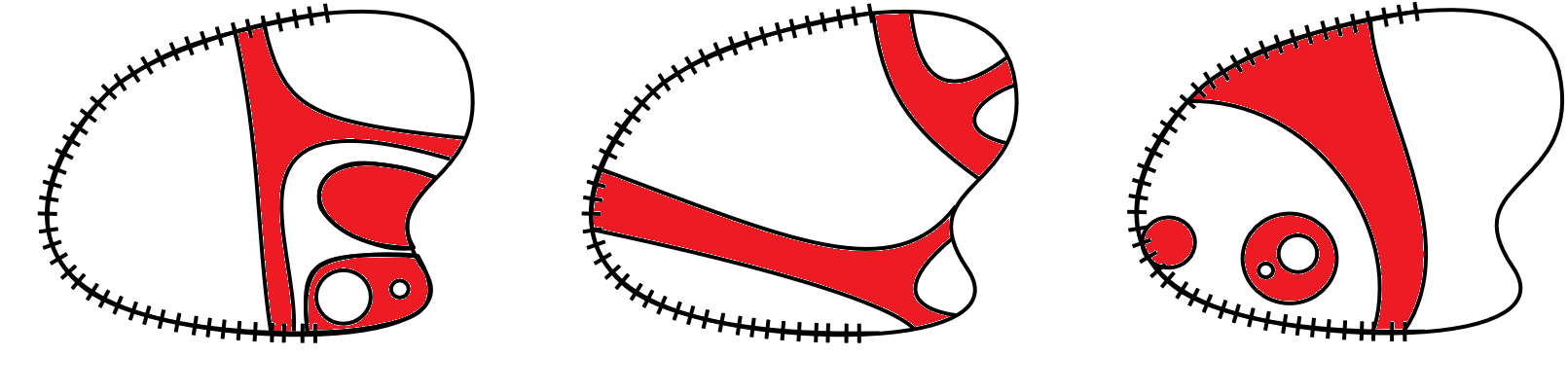}
    \caption{Some possibilities for the location of the diagram $D \subset X$ with the hashed lines representing $r_X^{-1}(1)$.}
    \label{f : cases for diagram}
\end{figure}

All neighbouring faces of $D$ in $X$ are labelled by elements of $\mathcal{R}^{d_s}_{m,l}$, set $\alpha = \partial X \setminus r_X^{-1}(1)$ the geodesic in $K_{d_t}$ and partition $\partial D $ into a collection of maximal, non-trivial subpaths $ \partial D = P_1 \cup \cdots \cup P_n$ where $P_i$ is either a subpath of $\alpha$, $r_X^{-1}(1)$ or has a neighbourhood in $\mathcal{R}^{d_s}_{m,l}$.
We say a path $P \rightarrow X^{(1)}$ has a \emph{neighbourhood in } $\mathcal{R}^{d_s}_{m,l}$ if there exists a reduced connected van-Kampen diagram $D \rightarrow K_{d_s}$ with combinatorial maps $P \rightarrow \partial D \rightarrow D \rightarrow X \rightarrow K_{d_t}$ and $D = \overline{D^{(2)}}$.
Note $X \neq D$ if and only if there exists $i$ such that $P_i$ has a neighbourhood in $\mathcal{R}^{d_s}_{m,l}$, and note $n \leq Cl$.
Let $I_D \subset I$ be the indices such that $P_i \nsubseteq \alpha$.

Let $(Y,r_Y, \text{lab}')$ be a reduced connected abstract diagram restricted in $S$ with $\lvert Y \rvert = \lvert D \rvert$, $\lvert \partial Y \rvert = \lvert \partial D \rvert$ and each edge in the boundary of some face with all faces $l$-gons.
The abstract diagram structure is defined according to $D$.
Identify the paths $i : \partial Y \rightarrow \partial D$ as $1$-complexes and define $r_Y^{-1}(1) = \{ e \in \partial Y : i(e) \in P_j \text{ where } j \in I_D\}$.
This identification also defines the boundary restriction function $\text{lab}'$, by the mapping $D \rightarrow X \rightarrow K_{d_t}$, which is a disjoint collection of concatenations of alternating subpaths of $r_X^{-1}(1)$ and paths in $\partial D$ with a neighbourhood in $\mathcal{R}^{d_s}_{m,l}$.

There are at most
\begin{equation}\label{eq : counting new restricted}
2^{2C}(C' l)^2 P(l) (2m-1)^{(d_s C + g)l }
\end{equation}

\noindent
possible diagrams $(Y,r_Y, \text{lab}')$ by counting over all possible diagrams $D$ and using the assumption on the size of $\lvert \mathbf{W}_{d_s,l} \rvert$ summing over possible $(r_X, \text{lab})$.

\begin{claim}\label{cl : larger labelled boundary}
$\sum_{i \in I_D} \lvert P_i \rvert \geq \sum_{i \in I \setminus I_D} \lvert P_i \rvert $ as $\alpha$ is geodesic in $K_{d_t}$.
\begin{proof}
Since $r_X^{-1}(1) \subset \partial X$ is a subpath we can partition $I = I_1 \sqcup I_2$ such that all $j \in I$ with $P_j \subset r_X^{-1}(1)$ are in $I_1$ and $I_1$ is maximal with respect to containing $P_j$ with a neighbourhood in $\mathcal{R}^{d_s}_{m,l}$ and $\cup_{i \in I_1} P_i \subset \partial D$ is a connected subpath.
Define $P_D \subset \alpha$ as the connected subpath joining the end points of $\cup_{i \in I_1} P_i$ together.
Since $\alpha$ is a geodesic in $K_{d_t}$ we have $\sum_{i \in I_1} \lvert P_i \rvert \geq \lvert P_D \rvert$.

Hence,

$$ \sum_{ i \in I_D} \lvert P_i \rvert \geq  \sum_{i \in I_1} \lvert P_i \rvert \geq \lvert P_D \rvert \geq \sum_{i \in I \setminus I_D} \lvert P_i \rvert.$$
\noindent
as for all $i \in I \setminus I_D$ $P_i \subset \alpha$ and $I_1 \subset I_D$.
\end{proof}
\end{claim}

Claim \ref{cl : larger labelled boundary} implies the hypothesis of Theorem \ref{t : restricted abstract diagram rule out} is satisfied for a given $(Y,r_Y, \text{lab}')$, it is equivalent to $2 \lvert r_Y^{-1}(1) \rvert \geq \lvert \partial Y \rvert $, so summing over possible $(Y,r_Y, \text{lab}')$ using (\ref{eq : counting new restricted})
\begin{align*}
    \mathbb{P}(D \rightarrow K_{d_t} \text{ exists}) &\leq \mathbb{P}((Y,r_Y, \text{lab}') \text{ is fillable by } \mathcal{R}^{d_t}_{m,l} ) \\
    & \leq 2^{2C} P(l) (2m-1)^{-dl}.
\end{align*}

In the above probability estimate we have summed over pairs $(r_X, \text{lab})$ implicitly for a fixed $X$ so to get the main probability estimate we sum over possible $X$ given by Lemma \ref{l : counting the diagrams}.

We have shown the probability of the event occurring is bounded above by $2^{2C}N(C,l) P(l) (2m-1)^{-dl}$.

\end{proof}
\end{lemma}

\subsection{Emanating Paths are Quasi-geodesics}

In this subsection we fix $ \varphi : A \rightarrow K_{d_s}$ as our low-density undistorted round tree.
Let $k \geq 1$ and $\gamma \in \Gamma(k)$ be an emanating path for such a round tree.
Recall $\epsilon = 1/2 - d_t$ and consider the path metric $\rho_A$ of $A^{(1)}$.

\begin{lemma}(Local-geodesics)\label{l : emanating rays are local-geodesics}
All paths $\pi^{d_t}_{d_s} \circ \gamma$ are $18 \epsilon^{-1}l$-local-geodesics in $K_{d_t}$ with overwhelming probability in $\mathcal{G}^{d_t}_{m,l}$.

\begin{proof}
Suppose not.
Denote $\pi^{d_t}_{d_s} \circ \gamma$ by $\gamma'$.
Let $P$ be a subpath of $\gamma'$ of minimal length with a geodesic $\alpha$ in $K_{d_t}$ joining the endpoints of $P$, $P^-$ and $P^+$, together of lengths $\lvert \alpha \rvert < \lvert P \rvert < 18 \epsilon^{-1} l$.
We can assume $\alpha \cap \gamma^- = \{P^-, P^+\}$ by minimality of $P$. 
Set $\mathbf{W}_{d_s,l} = \mathbf{E}_{ \leq 36 \epsilon^{-1} l } $ as the set of emanating words, Definition \ref{d : emanating words}, of length at most $36 \epsilon^{-1} l$ for the undistorted round tree $A \rightarrow K_{d_s}$.

Let $D$ be a connected reduced van-Kampen diagram in $K_{d_t}$ bounded by $P \cdot \alpha^{-1}$.
By minimality of $P$ we can assume $D$ is homeomorphic to a disc, hence $D = \overline{D^{(2)}}$ and so each edge is in the boundary of some face.
By Theorem \ref{t : generic linear iso} we may assume we have $\lvert D \rvert \leq 36 \epsilon^{-2}$.
The diagram $D$ gives a reduced connected abstract diagram $(X, r_X, \text{lab})$ restricted in $S$ with $r_X^{-1}(1) = P$ and the label of the restricted boundary path in $\mathbf{W}_{d_s,l}$.

Setting $C = 36 \epsilon^{-2}$, $C' = 36 \epsilon^{-1}$, $g = 4 \cdot 10^{-7} \epsilon(36 + \epsilon^2) < \epsilon/4$, motivated by (\ref{eq : count bound}), and computing $ d_s C + g + d_t < 1/2 - (\epsilon - \epsilon / 4 - 36 \cdot 10^{-7} \epsilon) $ the hypothesis of Lemma \ref{l : filling diagrams along paths} is satisfied with $d > \epsilon - \epsilon / 4 - 36 \cdot 10^{-7} \epsilon > 0$.

We conclude for some polynomial $Q(l)$ the probability that $(X, r_X, \text{lab})$ is fillable by 
$\mathcal{R}^{d_t}_{m,l}$ with $\partial X \setminus r_X^{-1}(1)$ a geodesic in $K_{d_t}$ and at least one face $f$ of $X$ bearing a relator $r \in \mathcal{R}^{d_t}_{m,l} \setminus \mathcal{R}^{d_s}_{m,l}$ is less than $Q(l)(2m-1)^{-dl}$. 

As by assumption $(X,r_X, \text{lab})$ is fillable by $\mathcal{R}^{d_t}_{m,l}$ with overwhelming probability in $\mathcal{G}^{d_t}_{m,l}$ and this implies that $D$ is a van-Kampen diagram in $K_{d_s}$ with overwhelming probability.
Since $\lvert \alpha \rvert < \lvert P \rvert$ we contradict the fact emanating paths, in particular $P$ and $\gamma$, are geodesics in $K_{d_s}$ \cite[Lem.8.13]{Mackay-conf-rand-16}.
We contradict any such $\gamma'$ not being a local-geodesic with probability going to $1$ as $l \to \infty$.
\end{proof}
\end{lemma}

From this and Theorem \ref{t : hyperbolic generic} we have the following consequence by \cite[III.H.Thm.1.13]{Bridson-Haefliger-non-positive-curv-geom}.

\begin{corollary}(Quasi-geodesics)\label{c : emanating rays are quasi-geodesics}
All paths $\pi^{d_t}_{d_s} \circ \gamma$ are embedded $(13/5, 4\epsilon^{-1}l)$-quasi-geodesics in $K_{d_t}$ with overwhelming probability in $\mathcal{G}^{d_t}_{m,l}$.
\begin{proof}
They are $(13/5, 4\epsilon^{-1}l)$-quasi-geodesics by \cite[III.H.Thm.1.13]{Bridson-Haefliger-non-positive-curv-geom}.
Denote $\pi^{d_t}_{d_s} \circ \gamma$ by $\gamma'$.
Suppose $\gamma'$ self-intersects itself at point $Q$ and let $P$ be the subpath of $\gamma'$ that forms the loop starting and ending at $Q$.
Let $Q_1$ and $Q_2$ denote the endpoints, in $A^{(1)}$, of the path $P'$ satisfying $\pi^{d_t}_{d_s} \circ \varphi (P') = P$.
$P'$ is an emanating path, Definition \ref{d : emanating paths}, and so is geodesic.
If $\rho_A( Q_1, Q_2) = \lvert P' \rvert \leq 18 \epsilon^{-1} l $ by Lemma \ref{l : emanating rays are local-geodesics} we have contradicted $\gamma'$ being an $18\epsilon^{-1}l$-local geodesic.
Otherwise $\rho_A( Q_1, Q_2) = \lvert P' \rvert > 18 \epsilon^{-1} l $ but as $\gamma'$ is a $(13/5, 4 \epsilon^{-1} l)$-quasi-geodesic we have $\rho_t(Q,Q) > \frac{5}{13} \cdot 18 \epsilon^{-1} l - 4\epsilon^{-1} l =  \frac{38}{13} \epsilon^{-1} l > 0$ a contradiction.

\end{proof}
\end{corollary}

\subsection{Quasi-isometric Embedding}

Up to this point we have shown all emanating paths are quasi-geodesics.
We will use this fact in showing the combinatorial round tree $A$ is undistorted.

\begin{lemma}(Restricts to quasi-isometry)\label{l : quasi-isometric-embedding}
The compositions of mappings 
$$(A^{(1)}, \rho_A) \rightarrow (K_{d_s}, \rho_s) \rightarrow (K_{d_t}, \rho_t)$$ are quasi-isometric embeddings with overwhelming probability in $\mathcal{G}^{d_s}_{m,l}$ and $\mathcal{G}^{d_t}_{m,l}$.

\begin{proof}
By Theorem \ref{t : build low d round trees} the first mapping $\varphi : (A^{(1)}, \rho_A) \rightarrow (K_{d_s}, \rho_s)$ is a $(\lambda, k)$-quasi-isometric embedding for some $\lambda \geq 1$, $k \geq 0$.

Take $p, q \in A$.
Let $\gamma_{1,p} \in \Gamma(\rho_A(1,p))$ and $ \gamma_{1,q} \in \Gamma(\rho_A(1,q))$ be emanating paths of $A$ joining $1$ to $p$ and $q$ respectively.
We have that $ \varphi \circ \gamma_{1, \cdot}$ is a geodesic in $K_{d_s}$ by \cite[Lem. 8.13]{Mackay-conf-rand-16}.

Consider the geodesic triangle $\Delta$ in $K_{d_s}$ with vertices $1,\varphi(p), \varphi(q)$ and sides $\varphi \circ \gamma_{1,p}$, $\varphi \circ \gamma_{1,q}$ and a geodesic $[\varphi(p), \varphi(q)]$ for the other.
Let $L$ be a reduced diagram for $\Delta$.
Let $R_{p,q}$ be the $2$-cell or $1$-cell that meets both $\varphi \circ \gamma_{1,p}$ and $\varphi \circ \gamma_{1,q}$ and is furthest from $1$ in $L$.
In a similar way we define the cells $R_{1,p}$ and $R_{1,q}$ and consider the subdiagram $L' \subset L$ containing, $R_{p,q}, R_{1,p}$ and $R_{1,q}$ on the periphery.
Apart from at most three $2$-cells of $L \rightarrow K_{d_s}$ all intersect along $\lvert \partial L'  \rvert$ for at most $l/2$.
Since, for a $2$-cell $R \subset K_{d_s}$ and geodesic $\gamma \subset K_{d_s}$, when $d_s = 10^{-7} \epsilon^{3} < 1/4$, $R \cap \gamma$ is connected (Lemma \ref{l : cells intersecting along geodesics}) and if not we will contradict $\Delta$ being a geodesic triangle.
Hence by Theorem \ref{t : generic linear iso} for suitable $\epsilon' > 0$,
$$
    (1- \frac{52}{256}\cdot 10^{-7} - \epsilon')\lvert L' \rvert l < (1-2d_s - \epsilon')\lvert L' \rvert l
    \leq \lvert \partial L' \rvert \leq 3l + (\lvert L' \rvert - 3) \cdot \frac{1}{2}l = \frac{1}{2} \lvert L' \rvert l + \frac{3}{2} l,
$$
and so $\lvert L' \rvert < 4$, thus $\lvert L' \rvert \leq 3$.

Set $x$ as the point furthest along $\varphi \circ \gamma_{1,p}$ contained in $\partial R_{p,q}$, denote by $y$ similarly for $q$.
In the proof of \cite[Lem.8.15]{Mackay-conf-rand-16} it was shown $\rho_s(x,y) \leq 3l $, actually $\leq l$ as it is in the boundary of a face, thus $\rho_t(\pi(x), \pi (y )) \leq l$.
Consider the subpath $C'$ of $\partial R_{p,q}$ joining $x$ to $y$ and denote its image under $\pi$ in $K_{d_t}$ as $C$.
There are at most $2l(2m-1)^{d_s l}$ labels for $C$ since $R_{p,q}$ bounds a ladder from $1$ to $C'$ in $K_{d_s}$ by Lemma \ref{l : finding a ladder} as $d_s < 1/6$.

Denote by $\alpha_{1,p}$ and $\alpha_{1,q}$ the images of $\varphi \circ \gamma_{1,p} $ and $ \varphi \circ \gamma_{1,q} $ in $K_{d_t}$ which are $(13/5, 4 \epsilon^{-1} l)$-quasi-geodesics by Corollary \ref{c : emanating rays are quasi-geodesics}.
Set $u = \pi \circ \varphi (p)$, $v = \pi \circ \varphi (q)$ for notation reasons and $[u,v]$ a geodesic in $K_{d_t}$.
Denote this triangle of paths $\Delta_1$ with vertices $1, u, v$ and sides quasi-geodesics $\alpha_{1,p}, \alpha_{1,q}$ and geodesic $[u,v]$.

By the stability of quasi-geodesics, \cite[III.H.Thm.1.7]{Bridson-Haefliger-non-positive-curv-geom}, there exists $R = R(d_t, l) \geq 0$ such that a geodesic triangle $\Delta_2$ with one geodesic edge the previous geodesic $[u,v]$ and others $e_{1,u}, e_{1,v}$ satisfies $\Delta_2 \subset \mathcal{N}_{R}(\Delta_1)$ and $ \Delta_1 \subset \mathcal{N}_{R}(\Delta_2)$.
Using \cite[Thm.1.1]{Gouezel-Shchur-corrected-quantitative-version-of-the-morse-lemma} we can take $R = 4 \cdot 10^3 \epsilon^{-1} l $ which we will do later.
Let $\chi :\Delta_2 \rightarrow T$ be a tripod tree approximation for $\Delta_2$ with center $O \in T$.

Label $x' \in e_{1,u}$ a point such that $\rho_t(x', \pi(x) ) \leq R$ and $y' \in e_{1,v}$ similarly for $v$. 
Set $z_1' \in e_{1,u}$ and $z_2' \in e_{1,v}$ as the points in $ \chi^{-1}(O) $ satisfying $ \rho_t(z_1', z_2') \leq 4\delta_t \leq 8\epsilon^{-1}l $.
Recall $\delta_t \leq 2\epsilon^{-1} l$ by Theorem \ref{t : hyperbolic generic}.
We denote $z_1 \in \alpha_{1,p}$ and $z_2 \in \alpha_2$ as points $R$-close to $z_1'$ and $ z_2'$ respectively.

See Figure \ref{f : final configuration} to see the configuration clearly.

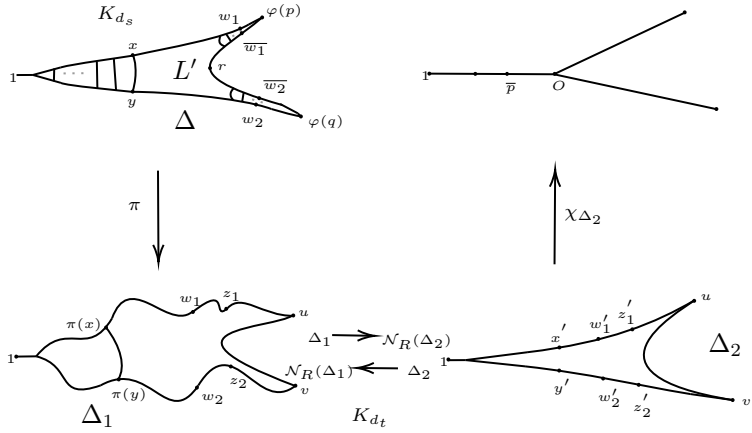
\begin{figure}[!ht]
\centering
\tikzset{every picture/.style={line width=0.75pt}} 

\begin{tikzpicture}[x=0.6pt,y=0.6pt,yscale=-1,xscale=1]

\draw    (120.22,121) -- (120.51,176.77) ;
\draw [shift={(120.52,178.77)}, rotate = 269.7] [color={rgb, 255:red, 0; green, 0; blue, 0 }  ][line width=0.75]    (10.93,-3.29) .. controls (6.95,-1.4) and (3.31,-0.3) .. (0,0) .. controls (3.31,0.3) and (6.95,1.4) .. (10.93,3.29)   ;
\draw    (271.5,245.75) -- (247.33,245.52) ;
\draw [shift={(245.33,245.5)}, rotate = 0.55] [color={rgb, 255:red, 0; green, 0; blue, 0 }  ][line width=0.75]    (10.93,-3.29) .. controls (6.95,-1.4) and (3.31,-0.3) .. (0,0) .. controls (3.31,0.3) and (6.95,1.4) .. (10.93,3.29)   ;
\draw    (370.2,180.1) -- (370.59,122.2) ;
\draw [shift={(370.6,120.2)}, rotate = 90.38] [color={rgb, 255:red, 0; green, 0; blue, 0 }  ][line width=0.75]    (10.93,-3.29) .. controls (6.95,-1.4) and (3.31,-0.3) .. (0,0) .. controls (3.31,0.3) and (6.95,1.4) .. (10.93,3.29)   ;
\draw    (291.2,59.8) -- (369.8,60.2) ;
\draw    (369.8,60.2) -- (452.2,21.4) ;
\draw    (369.8,60.2) -- (472.2,82.2) ;
\draw    (303.14,240.59) -- (314.15,240.54) ;
\draw    (314.15,240.54) -- (352.9,235.61) ;
\draw    (314.15,240.54) -- (351.64,244.49) ;
\draw    (455.76,261.95) -- (482.5,265.75) ;
\draw    (352.9,235.61) .. controls (404.89,229.36) and (434,217) .. (458.53,202.75) ;
\draw    (351.64,244.49) .. controls (392.45,248.98) and (420.39,256.12) .. (455.76,261.95) ;
\draw    (31.04,60.69) -- (41.38,60.64) ;
\draw    (41.38,60.64) -- (58.92,55.88) ;
\draw    (41.38,60.64) -- (59,65.63) ;
\draw    (197.22,78.94) -- (208.4,84.85) ;
\draw    (177.5,28.75) -- (185.9,24.35) ;
\draw    (185.9,24.35) .. controls (157.5,48.75) and (123.25,58.38) .. (198.56,79.61) ;
\draw    (208.4,84.85) -- (210.33,86.83) ;
\draw    (104.43,48.43) .. controls (125.86,43.57) and (162.5,37.75) .. (177.5,28.75) ;
\draw    (103.86,71.29) .. controls (123.33,71.21) and (171.94,74.82) .. (191.75,82.13) ;
\draw    (58.92,55.88) .. controls (67.67,53.5) and (95.25,49.96) .. (104.43,48.43) ;
\draw    (59,65.63) .. controls (62.27,67.44) and (96.43,70.43) .. (103.86,71.29) ;
\draw    (104.43,48.43) .. controls (105.55,49.91) and (107.67,59.83) .. (104.25,71.04) ;
\draw    (191.75,82.13) .. controls (197.08,83.54) and (202.25,85.38) .. (210.33,86.83) ;
\draw    (30.84,238.33) -- (43.46,238.16) ;
\draw  [fill={rgb, 255:red, 0; green, 0; blue, 0 }  ,fill opacity=1 ] (30.26,238.33) .. controls (30.26,237.99) and (30.52,237.71) .. (30.84,237.71) .. controls (31.16,237.71) and (31.42,237.99) .. (31.42,238.33) .. controls (31.42,238.67) and (31.16,238.95) .. (30.84,238.95) .. controls (30.52,238.95) and (30.26,238.67) .. (30.26,238.33) -- cycle ;
\draw    (43.46,238.16) .. controls (62.91,214.22) and (75.53,240.42) .. (91.23,214.22) ;
\draw    (164.16,207.67) .. controls (173.4,197.85) and (186.32,212.91) .. (205.4,212.58) ;
\draw    (43.46,238.16) .. controls (67.53,244.35) and (54.91,266.62) .. (84.15,256.46) ;
\draw    (84.15,256.46) .. controls (115.23,239.11) and (126,285.61) .. (145.7,256.46) ;
\draw    (145.7,256.46) .. controls (168.47,219.79) and (182.63,276.11) .. (207.25,256.46) ;
\draw    (142,210.29) .. controls (166.62,190.64) and (153.7,217.17) .. (164.16,207.67) ;
\draw    (87.22,219.63) .. controls (97.07,232.23) and (99.84,245) .. (95.23,252.54) ;
\draw  [fill={rgb, 255:red, 0; green, 0; blue, 0 }  ,fill opacity=1 ] (204.94,212.39) .. controls (204.94,212.05) and (205.2,211.77) .. (205.52,211.77) .. controls (205.84,211.77) and (206.1,212.05) .. (206.1,212.39) .. controls (206.1,212.73) and (205.84,213.01) .. (205.52,213.01) .. controls (205.2,213.01) and (204.94,212.73) .. (204.94,212.39) -- cycle ;
\draw  [fill={rgb, 255:red, 0; green, 0; blue, 0 }  ,fill opacity=1 ] (94.88,252.74) .. controls (94.88,252.39) and (95.14,252.12) .. (95.47,252.12) .. controls (95.79,252.12) and (96.05,252.39) .. (96.05,252.74) .. controls (96.05,253.08) and (95.79,253.35) .. (95.47,253.35) .. controls (95.14,253.35) and (94.88,253.08) .. (94.88,252.74) -- cycle ;
\draw  [fill={rgb, 255:red, 0; green, 0; blue, 0 }  ,fill opacity=1 ] (86.88,219.99) .. controls (86.88,219.65) and (87.14,219.37) .. (87.46,219.37) .. controls (87.79,219.37) and (88.05,219.65) .. (88.05,219.99) .. controls (88.05,220.33) and (87.79,220.61) .. (87.46,220.61) .. controls (87.14,220.61) and (86.88,220.33) .. (86.88,219.99) -- cycle ;
\draw  [fill={rgb, 255:red, 0; green, 0; blue, 0 }  ,fill opacity=1 ] (165.51,244.57) .. controls (165.51,244.22) and (165.77,243.95) .. (166.09,243.95) .. controls (166.41,243.95) and (166.67,244.22) .. (166.67,244.57) .. controls (166.67,244.91) and (166.41,245.18) .. (166.09,245.18) .. controls (165.77,245.18) and (165.51,244.91) .. (165.51,244.57) -- cycle ;
\draw  [fill={rgb, 255:red, 0; green, 0; blue, 0 }  ,fill opacity=1 ] (162.66,208.17) .. controls (162.66,207.82) and (162.92,207.55) .. (163.24,207.55) .. controls (163.57,207.55) and (163.83,207.82) .. (163.83,208.17) .. controls (163.83,208.51) and (163.57,208.78) .. (163.24,208.78) .. controls (162.92,208.78) and (162.66,208.51) .. (162.66,208.17) -- cycle ;
\draw  [fill={rgb, 255:red, 0; green, 0; blue, 0 }  ,fill opacity=1 ] (206.29,256.66) .. controls (206.29,256.32) and (206.55,256.05) .. (206.87,256.05) .. controls (207.19,256.05) and (207.45,256.32) .. (207.45,256.66) .. controls (207.45,257.01) and (207.19,257.28) .. (206.87,257.28) .. controls (206.55,257.28) and (206.29,257.01) .. (206.29,256.66) -- cycle ;
\draw    (204.94,212.39) .. controls (193.22,220.82) and (110,220.25) .. (206.29,256.66) ;
\draw    (458.53,202.75) .. controls (446.81,211.18) and (380.5,249.25) .. (482.5,265.75) ;
\draw    (92.4,50.6) -- (93.8,70.6) ;
\draw    (81.67,52.17) -- (82,69.17) ;
\draw    (54.8,56.5) -- (54.6,64.7) ;
\draw [color={rgb, 255:red, 155; green, 155; blue, 155 }  ,draw opacity=1 ] [dash pattern={on 0.84pt off 2.51pt}]  (60.67,59.67) -- (75.67,59.5) ;
\draw  [fill={rgb, 255:red, 0; green, 0; blue, 0 }  ,fill opacity=1 ] (471.69,82.14) .. controls (471.69,81.8) and (471.95,81.52) .. (472.27,81.52) .. controls (472.59,81.52) and (472.85,81.8) .. (472.85,82.14) .. controls (472.85,82.48) and (472.59,82.76) .. (472.27,82.76) .. controls (471.95,82.76) and (471.69,82.48) .. (471.69,82.14) -- cycle ;
\draw  [fill={rgb, 255:red, 0; green, 0; blue, 0 }  ,fill opacity=1 ] (302.69,240.14) .. controls (302.69,239.8) and (302.95,239.52) .. (303.27,239.52) .. controls (303.59,239.52) and (303.85,239.8) .. (303.85,240.14) .. controls (303.85,240.48) and (303.59,240.76) .. (303.27,240.76) .. controls (302.95,240.76) and (302.69,240.48) .. (302.69,240.14) -- cycle ;
\draw  [fill={rgb, 255:red, 0; green, 0; blue, 0 }  ,fill opacity=1 ] (290.69,59.81) .. controls (290.69,59.47) and (290.95,59.19) .. (291.27,59.19) .. controls (291.59,59.19) and (291.85,59.47) .. (291.85,59.81) .. controls (291.85,60.15) and (291.59,60.43) .. (291.27,60.43) .. controls (290.95,60.43) and (290.69,60.15) .. (290.69,59.81) -- cycle ;
\draw  [fill={rgb, 255:red, 0; green, 0; blue, 0 }  ,fill opacity=1 ] (369.69,59.95) .. controls (369.69,59.61) and (369.95,59.33) .. (370.27,59.33) .. controls (370.59,59.33) and (370.85,59.61) .. (370.85,59.95) .. controls (370.85,60.29) and (370.59,60.57) .. (370.27,60.57) .. controls (369.95,60.57) and (369.69,60.29) .. (369.69,59.95) -- cycle ;
\draw  [fill={rgb, 255:red, 0; green, 0; blue, 0 }  ,fill opacity=1 ] (339.69,59.95) .. controls (339.69,59.61) and (339.95,59.33) .. (340.27,59.33) .. controls (340.59,59.33) and (340.85,59.61) .. (340.85,59.95) .. controls (340.85,60.29) and (340.59,60.57) .. (340.27,60.57) .. controls (339.95,60.57) and (339.69,60.29) .. (339.69,59.95) -- cycle ;
\draw  [fill={rgb, 255:red, 0; green, 0; blue, 0 }  ,fill opacity=1 ] (319.69,59.95) .. controls (319.69,59.61) and (319.95,59.33) .. (320.27,59.33) .. controls (320.59,59.33) and (320.85,59.61) .. (320.85,59.95) .. controls (320.85,60.29) and (320.59,60.57) .. (320.27,60.57) .. controls (319.95,60.57) and (319.69,60.29) .. (319.69,59.95) -- cycle ;
\draw  [fill={rgb, 255:red, 0; green, 0; blue, 0 }  ,fill opacity=1 ] (418.69,221.14) .. controls (418.69,220.8) and (418.95,220.52) .. (419.27,220.52) .. controls (419.59,220.52) and (419.85,220.8) .. (419.85,221.14) .. controls (419.85,221.48) and (419.59,221.76) .. (419.27,221.76) .. controls (418.95,221.76) and (418.69,221.48) .. (418.69,221.14) -- cycle ;
\draw  [fill={rgb, 255:red, 0; green, 0; blue, 0 }  ,fill opacity=1 ] (422.69,256.14) .. controls (422.69,255.8) and (422.95,255.52) .. (423.27,255.52) .. controls (423.59,255.52) and (423.85,255.8) .. (423.85,256.14) .. controls (423.85,256.48) and (423.59,256.76) .. (423.27,256.76) .. controls (422.95,256.76) and (422.69,256.48) .. (422.69,256.14) -- cycle ;
\draw  [fill={rgb, 255:red, 0; green, 0; blue, 0 }  ,fill opacity=1 ] (372.69,232.57) .. controls (372.69,232.23) and (372.95,231.95) .. (373.27,231.95) .. controls (373.59,231.95) and (373.85,232.23) .. (373.85,232.57) .. controls (373.85,232.91) and (373.59,233.19) .. (373.27,233.19) .. controls (372.95,233.19) and (372.69,232.91) .. (372.69,232.57) -- cycle ;
\draw  [fill={rgb, 255:red, 0; green, 0; blue, 0 }  ,fill opacity=1 ] (372.57,247.14) .. controls (372.57,246.8) and (372.83,246.52) .. (373.15,246.52) .. controls (373.47,246.52) and (373.73,246.8) .. (373.73,247.14) .. controls (373.73,247.48) and (373.47,247.76) .. (373.15,247.76) .. controls (372.83,247.76) and (372.57,247.48) .. (372.57,247.14) -- cycle ;
\draw  [fill={rgb, 255:red, 0; green, 0; blue, 0 }  ,fill opacity=1 ] (103.69,71.14) .. controls (103.69,70.8) and (103.95,70.52) .. (104.27,70.52) .. controls (104.59,70.52) and (104.85,70.8) .. (104.85,71.14) .. controls (104.85,71.48) and (104.59,71.76) .. (104.27,71.76) .. controls (103.95,71.76) and (103.69,71.48) .. (103.69,71.14) -- cycle ;
\draw  [fill={rgb, 255:red, 0; green, 0; blue, 0 }  ,fill opacity=1 ] (103.69,48.14) .. controls (103.69,47.8) and (103.95,47.52) .. (104.27,47.52) .. controls (104.59,47.52) and (104.85,47.8) .. (104.85,48.14) .. controls (104.85,48.48) and (104.59,48.76) .. (104.27,48.76) .. controls (103.95,48.76) and (103.69,48.48) .. (103.69,48.14) -- cycle ;
\draw  [fill={rgb, 255:red, 0; green, 0; blue, 0 }  ,fill opacity=1 ] (209.75,86.83) .. controls (209.75,86.49) and (210.01,86.21) .. (210.33,86.21) .. controls (210.65,86.21) and (210.91,86.49) .. (210.91,86.83) .. controls (210.91,87.17) and (210.65,87.45) .. (210.33,87.45) .. controls (210.01,87.45) and (209.75,87.17) .. (209.75,86.83) -- cycle ;
\draw  [fill={rgb, 255:red, 0; green, 0; blue, 0 }  ,fill opacity=1 ] (185.32,24.35) .. controls (185.32,24.01) and (185.58,23.73) .. (185.9,23.73) .. controls (186.22,23.73) and (186.48,24.01) .. (186.48,24.35) .. controls (186.48,24.69) and (186.22,24.97) .. (185.9,24.97) .. controls (185.58,24.97) and (185.32,24.69) .. (185.32,24.35) -- cycle ;
\draw  [fill={rgb, 255:red, 0; green, 0; blue, 0 }  ,fill opacity=1 ] (457.69,203.14) .. controls (457.69,202.8) and (457.95,202.52) .. (458.27,202.52) .. controls (458.59,202.52) and (458.85,202.8) .. (458.85,203.14) .. controls (458.85,203.48) and (458.59,203.76) .. (458.27,203.76) .. controls (457.95,203.76) and (457.69,203.48) .. (457.69,203.14) -- cycle ;
\draw  [fill={rgb, 255:red, 0; green, 0; blue, 0 }  ,fill opacity=1 ] (481.92,265.75) .. controls (481.92,265.41) and (482.18,265.13) .. (482.5,265.13) .. controls (482.82,265.13) and (483.08,265.41) .. (483.08,265.75) .. controls (483.08,266.09) and (482.82,266.37) .. (482.5,266.37) .. controls (482.18,266.37) and (481.92,266.09) .. (481.92,265.75) -- cycle ;
\draw  [fill={rgb, 255:red, 0; green, 0; blue, 0 }  ,fill opacity=1 ] (451.69,21.14) .. controls (451.69,20.8) and (451.95,20.52) .. (452.27,20.52) .. controls (452.59,20.52) and (452.85,20.8) .. (452.85,21.14) .. controls (452.85,21.48) and (452.59,21.76) .. (452.27,21.76) .. controls (451.95,21.76) and (451.69,21.48) .. (451.69,21.14) -- cycle ;
\draw    (230,224.75) -- (255,225.21) ;
\draw [shift={(257,225.25)}, rotate = 181.06] [color={rgb, 255:red, 0; green, 0; blue, 0 }  ][line width=0.75]    (10.93,-3.29) .. controls (6.95,-1.4) and (3.31,-0.3) .. (0,0) .. controls (3.31,0.3) and (6.95,1.4) .. (10.93,3.29)   ;
\draw    (91.23,214.22) .. controls (96.54,199.69) and (104.57,200.61) .. (113.2,204.29) .. controls (121.82,207.98) and (133.48,216.78) .. (141.66,210.17) ;
\draw    (168.25,68.88) .. controls (165.58,70.93) and (165,74.13) .. (167.89,76.61) ;
\draw    (159.44,36.04) .. controls (158.89,38.83) and (157.7,41.03) .. (162.5,42.63) ;
\draw    (174.25,71.63) -- (173.33,77.28) ;
\draw    (163.18,34.77) -- (166.5,39.88) ;
\draw [color={rgb, 255:red, 155; green, 155; blue, 155 }  ,draw opacity=1 ] [dash pattern={on 0.84pt off 2.51pt}]  (167.11,36.33) -- (175.5,31.13) ;
\draw [color={rgb, 255:red, 155; green, 155; blue, 155 }  ,draw opacity=1 ] [dash pattern={on 0.84pt off 2.51pt}]  (179.56,76.06) -- (189.22,78.94) ;
\draw  [fill={rgb, 255:red, 0; green, 0; blue, 0 }  ,fill opacity=1 ] (152.44,56.31) .. controls (152.44,55.97) and (152.7,55.69) .. (153.02,55.69) .. controls (153.34,55.69) and (153.6,55.97) .. (153.6,56.31) .. controls (153.6,56.65) and (153.34,56.93) .. (153.02,56.93) .. controls (152.7,56.93) and (152.44,56.65) .. (152.44,56.31) -- cycle ;
\draw  [fill={rgb, 255:red, 0; green, 0; blue, 0 }  ,fill opacity=1 ] (183.44,75.31) .. controls (183.44,74.97) and (183.7,74.69) .. (184.02,74.69) .. controls (184.34,74.69) and (184.6,74.97) .. (184.6,75.31) .. controls (184.6,75.65) and (184.34,75.93) .. (184.02,75.93) .. controls (183.7,75.93) and (183.44,75.65) .. (183.44,75.31) -- cycle ;
\draw  [fill={rgb, 255:red, 0; green, 0; blue, 0 }  ,fill opacity=1 ] (181.44,79.31) .. controls (181.44,78.97) and (181.7,78.69) .. (182.02,78.69) .. controls (182.34,78.69) and (182.6,78.97) .. (182.6,79.31) .. controls (182.6,79.65) and (182.34,79.93) .. (182.02,79.93) .. controls (181.7,79.93) and (181.44,79.65) .. (181.44,79.31) -- cycle ;
\draw  [fill={rgb, 255:red, 0; green, 0; blue, 0 }  ,fill opacity=1 ] (172.44,34.31) .. controls (172.44,33.97) and (172.7,33.69) .. (173.02,33.69) .. controls (173.34,33.69) and (173.6,33.97) .. (173.6,34.31) .. controls (173.6,34.65) and (173.34,34.93) .. (173.02,34.93) .. controls (172.7,34.93) and (172.44,34.65) .. (172.44,34.31) -- cycle ;
\draw  [fill={rgb, 255:red, 0; green, 0; blue, 0 }  ,fill opacity=1 ] (171.44,31.31) .. controls (171.44,30.97) and (171.7,30.69) .. (172.02,30.69) .. controls (172.34,30.69) and (172.6,30.97) .. (172.6,31.31) .. controls (172.6,31.65) and (172.34,31.93) .. (172.02,31.93) .. controls (171.7,31.93) and (171.44,31.65) .. (171.44,31.31) -- cycle ;
\draw  [fill={rgb, 255:red, 0; green, 0; blue, 0 }  ,fill opacity=1 ] (141.66,210.17) .. controls (141.66,209.82) and (141.92,209.55) .. (142.24,209.55) .. controls (142.57,209.55) and (142.83,209.82) .. (142.83,210.17) .. controls (142.83,210.51) and (142.57,210.78) .. (142.24,210.78) .. controls (141.92,210.78) and (141.66,210.51) .. (141.66,210.17) -- cycle ;
\draw  [fill={rgb, 255:red, 0; green, 0; blue, 0 }  ,fill opacity=1 ] (144.16,257.92) .. controls (144.16,257.57) and (144.42,257.3) .. (144.74,257.3) .. controls (145.07,257.3) and (145.33,257.57) .. (145.33,257.92) .. controls (145.33,258.26) and (145.07,258.53) .. (144.74,258.53) .. controls (144.42,258.53) and (144.16,258.26) .. (144.16,257.92) -- cycle ;
\draw  [fill={rgb, 255:red, 0; green, 0; blue, 0 }  ,fill opacity=1 ] (397.25,227.17) .. controls (397.25,226.82) and (397.51,226.55) .. (397.83,226.55) .. controls (398.15,226.55) and (398.41,226.82) .. (398.41,227.17) .. controls (398.41,227.51) and (398.15,227.78) .. (397.83,227.78) .. controls (397.51,227.78) and (397.25,227.51) .. (397.25,227.17) -- cycle ;
\draw  [fill={rgb, 255:red, 0; green, 0; blue, 0 }  ,fill opacity=1 ] (400.25,252.17) .. controls (400.25,251.82) and (400.51,251.55) .. (400.83,251.55) .. controls (401.15,251.55) and (401.41,251.82) .. (401.41,252.17) .. controls (401.41,252.51) and (401.15,252.78) .. (400.83,252.78) .. controls (400.51,252.78) and (400.25,252.51) .. (400.25,252.17) -- cycle ;

\draw (99.6,139.25) node [anchor=north west][inner sep=0.75pt]  [font=\scriptsize]  {$\pi $};
\draw (374.6,143.4) node [anchor=north west][inner sep=0.75pt]  [font=\scriptsize]  {$\chi _{\Delta _{2}}$};
\draw (22.9,235.04) node [anchor=north west][inner sep=0.75pt]  [font=\tiny]  {$1$};
\draw (158.74,195.4) node [anchor=north west][inner sep=0.75pt]  [font=\tiny]  {$z_{1}$};
\draw (163.17,250.06) node [anchor=north west][inner sep=0.75pt]  [font=\tiny]  {$z_{2}$};
\draw (60,212.12) node [anchor=north west][inner sep=0.75pt]  [font=\tiny]  {$\pi(x)$};
\draw (88.55,256.31) node [anchor=north west][inner sep=0.75pt]  [font=\tiny]  {$\pi(y)$};
\draw (206.8,206.99) node [anchor=north west][inner sep=0.75pt]  [font=\tiny]  {$u$};
\draw (486.23,263.42) node [anchor=north west][inner sep=0.75pt]  [font=\tiny]  {$v$};
\draw (460.4,197.32) node [anchor=north west][inner sep=0.75pt]  [font=\tiny]  {$u$};
\draw (207.8,256.4) node [anchor=north west][inner sep=0.75pt]  [font=\tiny]  {$v$};
\draw (295.9,237.04) node [anchor=north west][inner sep=0.75pt]  [font=\tiny]  {$1$};
\draw (284.3,56.84) node [anchor=north west][inner sep=0.75pt]  [font=\tiny]  {$1$};
\draw (24.3,57.84) node [anchor=north west][inner sep=0.75pt]  [font=\tiny]  {$1$};
\draw (80,15.73) node [anchor=north west][inner sep=0.75pt]  [font=\scriptsize]  {$K_{d_{s}}$};
\draw (240.33,269.07) node [anchor=north west][inner sep=0.75pt]  [font=\scriptsize]  {$K_{d_{t}}{}$};
\draw (275.33,242.4) node [anchor=north west][inner sep=0.75pt]  [font=\tiny]  {$\Delta _{2}$};
\draw (128.53,81.73) node [anchor=north west][inner sep=0.75pt]    {$\Delta $};
\draw (69.73,266.67) node [anchor=north west][inner sep=0.75pt]    {$\Delta _{1}$};
\draw (465.33,222.07) node [anchor=north west][inner sep=0.75pt]    {$\Delta _{2}$};
\draw (198.46,241.52) node [anchor=north west][inner sep=0.75pt]  [font=\tiny]  {$\mathcal{N}_{R}( \Delta _{1})$};
\draw (366.67,62.4) node [anchor=north west][inner sep=0.75pt]  [font=\tiny]  {$O$};
\draw (407.74,200.65) node [anchor=north west][inner sep=0.75pt]  [font=\tiny]  {$z_{1}^{'}$};
\draw (417.74,259.65) node [anchor=north west][inner sep=0.75pt]  [font=\tiny]  {$z_{2}^{'}$};
\draw (365.74,215.65) node [anchor=north west][inner sep=0.75pt]  [font=\tiny]  {$x^{'}$};
\draw (367.74,250.65) node [anchor=north west][inner sep=0.75pt]  [font=\tiny]  {$y'$};
\draw (337.64,63.45) node [anchor=north west][inner sep=0.75pt]  [font=\tiny]  {$\overline{p}$};
\draw (99.24,38.65) node [anchor=north west][inner sep=0.75pt]  [font=\tiny]  {$x$};
\draw (98.74,74.15) node [anchor=north west][inner sep=0.75pt]  [font=\tiny]  {$y$};
\draw (186.37,13.88) node [anchor=north west][inner sep=0.75pt]  [font=\tiny]  {$\varphi ( p)$};
\draw (212.73,221.3) node [anchor=north west][inner sep=0.75pt]  [font=\tiny]  {$\Delta _{1}$};
\draw (259.96,222.52) node [anchor=north west][inner sep=0.75pt]  [font=\tiny]  {$\mathcal{N}_{R}( \Delta _{2})$};
\draw (212.87,84.4) node [anchor=north west][inner sep=0.75pt]  [font=\tiny]  {$\varphi ( q)$};
\draw (155.75,51.76) node [anchor=north west][inner sep=0.75pt]  [font=\tiny]  {$r$};
\draw (127.45,48.67) node [anchor=north west][inner sep=0.75pt]    {$L'$};
\draw (156,20.8) node [anchor=north west][inner sep=0.75pt]  [font=\tiny]  {$w_{1}$};
\draw (131.6,198.2) node [anchor=north west][inner sep=0.75pt]  [font=\tiny]  {$w_{1}$};
\draw (390,210.2) node [anchor=north west][inner sep=0.75pt]  [font=\tiny]  {$w'_{1}$};
\draw (170.8,85.8) node [anchor=north west][inner sep=0.75pt]  [font=\tiny]  {$w_{2}$};
\draw (145.6,260.2) node [anchor=north west][inner sep=0.75pt]  [font=\tiny]  {$w_{2}$};
\draw (395.6,256.6) node [anchor=north west][inner sep=0.75pt]  [font=\tiny]  {$w'_{2}$};
\draw (172.22,38.29) node [anchor=north west][inner sep=0.75pt]  [font=\tiny]  {$\overline{w_{1}}$};
\draw (184.8,61.8) node [anchor=north west][inner sep=0.75pt]  [font=\tiny]  {$\overline{w_{2}}$};
\end{tikzpicture}
\caption{Schematic for quasi-isometric embedding computation.}
\label{f : final configuration}
\end{figure} 

We set constants $A = 25 \cdot 10^3 $, $A' = 8 \cdot 10^3$ and $A'' = 25 \cdot 10^3 $ whose role will be clear in the proof of the following claim.

\begin{claim}\label{cl : gromov product bound}

$\min\{\rho_t(\pi(x),z_1), \rho_t(\pi(y),z_2) \} \leq (A + A') \epsilon^{-1} l$ so $ (u,v)_{\Bar{p}} \leq (A + A') \epsilon^{-1} l + 2R $ where $ \Bar{p}$ is the point $x'$ or $y'$ with smallest $\rho_t(x',z_1')$ or $ \rho_t(y',z_2')$.

\begin{proof}
Suppose $\rho_t(\pi(x),z_1) = \min\{\rho_t(\pi(x),z_1), \rho_t(\pi(y),z_2) \} > A \epsilon^{-1} l$ without loss of generality.
Let $P_1 \subset \alpha_{1,p}$, $P_2 \subset \alpha_{1,q}$ be subpaths joining $\pi(x) $ to $w_1$ and $\pi(y) $ to $w_2$ where the point $w_1$ is distance $ A \epsilon^{-1} l$ from $x$ along $\alpha_{1,p}$.
We define $w_2$ similarly.

If $\rho_t(w_1, z_1) \leq A' \epsilon^{-1} l $ then $\rho_t(\pi(x), z_1) \leq (A + A') \epsilon^{-1}l$ and $\rho_t(\pi(y), z_2) \leq l + (A + A') \epsilon^{-1} l + 4 \delta_t$ so $\rho_t(w_2, z_2) \leq l + (2A + A') \epsilon^{-1} l + 4 \delta_t $.

Otherwise $\rho_t(w_1, z_1 ) > A' \epsilon^{-1} l $ and by the stability of geodesics there exists $w_1' \in e_{1,u}$ with $\rho_t(w_1',w_1) \leq R$ and $w_2' \in e_{1,v}$ with $\rho_t(w_2', w_2) \leq R$.
Since $\rho_t(w_1, z_1) > 2R$, $w_1'$ lies on $e_{1,u}$ between $x'$ and $z_1'$.
Hence, by the tripod definition of hyperbolicity there exists $\overline{w} \in e_{1,v}$ with $\rho_t(w_1', \overline{w}) \leq 4 \delta_t$ so $\rho_t(w_1, \overline{w}) \leq R + 4 \delta_t$.

We aim to estimate $\rho_t(w_2', \overline{w})$ to bound $\rho_t(w_1,w_2)$.
Since $y', w_2', \overline{w} \in e_{1,v}$ with $e_{1,v}$ a geodesic in $K_{d_s}$,
we can bound $\lvert \rho_t(y', \overline{w}) - A \epsilon^{-1} l \rvert \leq \rho_t(y',\pi(x)) + \rho_t(w_1,\overline{w}) \leq 2R + l + 4\delta_t$ and $\lvert \rho_t(y', w_2') - A \epsilon^{-1} l \rvert \leq \rho_t(y', \pi(y)) + \rho_t(w_2,w_2') \leq 2R $ so $\rho_t(w_2', \overline{w}) \leq 4R + l + 4\delta_t $ and thus $\rho_t(w_1, w_2) \leq 6R + l + 8 \delta_t \leq A'' \epsilon^{-1} l$.

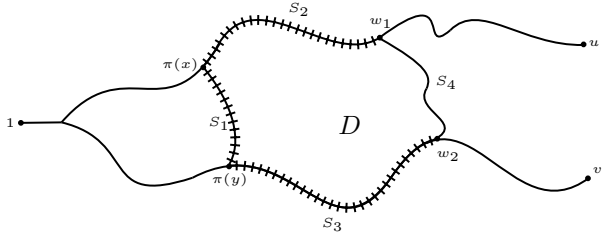
\begin{figure}[!ht]
\centering

\tikzset{every picture/.style={line width=0.75pt}} 

\begin{tikzpicture}[x=0.75pt,y=0.75pt,yscale=-1,xscale=1]

\draw    (140.39,150.06) -- (160.89,149.81) ;
\draw  [fill={rgb, 255:red, 0; green, 0; blue, 0 }  ,fill opacity=1 ] (139.44,150.06) .. controls (139.44,149.53) and (139.87,149.11) .. (140.39,149.11) .. controls (140.91,149.11) and (141.33,149.53) .. (141.33,150.06) .. controls (141.33,150.58) and (140.91,151) .. (140.39,151) .. controls (139.87,151) and (139.44,150.58) .. (139.44,150.06) -- cycle ;
\draw    (160.89,149.81) .. controls (192.5,113.25) and (213,153.25) .. (238.5,113.25) ;
\draw    (238.5,113.25) .. controls (257,75.25) and (295.5,122.75) .. (321,107.25)(238.39,108.26) -- (242.64,110.88)(241.05,104.5) -- (244.92,107.66)(243.96,101.44) -- (247.29,105.17)(247.49,98.82) -- (250.07,103.11)(251.69,96.88) -- (253.3,101.61)(256.07,95.86) -- (256.74,100.82)(260.1,95.63) -- (260.04,100.63)(264.63,95.99) -- (263.94,100.94)(268.74,96.75) -- (267.63,101.62)(272.87,97.83) -- (271.47,102.63)(276.55,98.98) -- (274.98,103.73)(280.7,100.42) -- (279.01,105.13)(284.38,101.76) -- (282.66,106.45)(288.04,103.1) -- (286.34,107.8)(291.67,104.37) -- (290.06,109.11)(295.69,105.67) -- (294.26,110.46)(299.2,106.64) -- (298,111.49)(303.05,107.45) -- (302.22,112.38)(306.77,107.9) -- (306.42,112.89)(310.32,107.93) -- (310.61,112.92)(313.71,107.48) -- (314.76,112.37)(317.29,106.35) -- (319.26,110.95) ;
\draw    (357,103.25) .. controls (372,88.25) and (393,111.25) .. (424,110.75) ;
\draw    (160.89,149.81) .. controls (200,159.25) and (179.5,193.25) .. (227,177.75) ;
\draw    (227,177.75) .. controls (277.5,151.25) and (295,222.25) .. (327,177.75) ;
\draw    (327,177.75) .. controls (364,121.75) and (387,207.75) .. (427,177.75) ;
\draw    (321,107.25) .. controls (361,77.25) and (340,117.75) .. (357,103.25) ;
\draw    (232,121.5) .. controls (248,140.75) and (252.5,160.25) .. (245,171.75)(236.52,123.16) -- (232.55,126.2)(238.91,126.42) -- (234.81,129.29)(241.27,129.96) -- (237.05,132.63)(243.39,133.49) -- (239.03,135.95)(245.39,137.29) -- (240.9,139.49)(247.09,141.04) -- (242.48,142.96)(248.5,144.74) -- (243.76,146.35)(249.74,148.92) -- (244.89,150.14)(250.57,153) -- (245.63,153.77)(251,157.22) -- (246,157.44)(250.92,161.52) -- (245.93,161.09)(250.21,165.83) -- (245.36,164.64)(248.76,170.06) -- (244.21,168) ;
\draw    (321,107.25) .. controls (346,121.25) and (347.5,126.75) .. (345,132.75) ;
\draw    (345,132.75) .. controls (337.5,144.75) and (369,150.75) .. (345,160.25) ;
\draw    (231.71,122.57) .. controls (237.86,114.64) and (237.14,115.71) .. (240.86,109.71)(232.21,117.86) -- (236.16,120.92)(234.62,114.71) -- (238.63,117.69)(236.75,111.59) -- (240.97,114.26) ;
\draw    (245,171.75) .. controls (280.71,167.86) and (301,216.71) .. (327,177.75)(249.22,169.03) -- (249.21,174.03)(253.56,169.27) -- (253.01,174.24)(257.74,169.94) -- (256.72,174.83)(261.76,170.96) -- (260.33,175.75)(265.64,172.28) -- (263.86,176.96)(269.37,173.84) -- (267.31,178.4)(273.41,175.8) -- (271.12,180.24)(276.87,177.66) -- (274.44,182.03)(280.23,179.57) -- (277.71,183.89)(283.88,181.72) -- (281.33,186.03)(287.41,183.79) -- (284.91,188.12)(290.82,185.7) -- (288.46,190.1)(294.13,187.37) -- (292,191.9)(297.68,188.87) -- (295.96,193.56)(301.09,189.9) -- (299.96,194.77)(304.68,190.4) -- (304.46,195.4)(308.17,190.18) -- (309.05,195.1)(311.34,189.25) -- (313.23,193.88)(314.6,187.51) -- (317.37,191.67)(317.32,185.41) -- (320.63,189.16)(320.14,182.6) -- (323.85,185.95)(322.68,179.51) -- (326.67,182.54)(324.92,176.36) -- (329.08,179.14) ;
\draw    (327,177.75) .. controls (334.71,166.71) and (340.74,160.9) .. (347.6,158.9)(327.37,172.95) -- (331.38,175.93)(329.93,169.64) -- (333.82,172.78)(332.52,166.58) -- (336.24,169.92)(335.5,163.5) -- (338.94,167.12)(338.69,160.77) -- (341.72,164.75)(342.38,158.37) -- (344.75,162.77)(346.44,156.64) -- (347.94,161.41) ;
\draw  [fill={rgb, 255:red, 0; green, 0; blue, 0 }  ,fill opacity=1 ] (423.24,110.46) .. controls (423.24,109.93) and (423.67,109.51) .. (424.19,109.51) .. controls (424.71,109.51) and (425.13,109.93) .. (425.13,110.46) .. controls (425.13,110.98) and (424.71,111.4) .. (424.19,111.4) .. controls (423.67,111.4) and (423.24,110.98) .. (423.24,110.46) -- cycle ;
\draw  [fill={rgb, 255:red, 0; green, 0; blue, 0 }  ,fill opacity=1 ] (244.44,172.06) .. controls (244.44,171.53) and (244.87,171.11) .. (245.39,171.11) .. controls (245.91,171.11) and (246.33,171.53) .. (246.33,172.06) .. controls (246.33,172.58) and (245.91,173) .. (245.39,173) .. controls (244.87,173) and (244.44,172.58) .. (244.44,172.06) -- cycle ;
\draw  [fill={rgb, 255:red, 0; green, 0; blue, 0 }  ,fill opacity=1 ] (231.44,122.06) .. controls (231.44,121.53) and (231.87,121.11) .. (232.39,121.11) .. controls (232.91,121.11) and (233.33,121.53) .. (233.33,122.06) .. controls (233.33,122.58) and (232.91,123) .. (232.39,123) .. controls (231.87,123) and (231.44,122.58) .. (231.44,122.06) -- cycle ;
\draw  [fill={rgb, 255:red, 0; green, 0; blue, 0 }  ,fill opacity=1 ] (349.44,158.06) .. controls (349.44,157.53) and (349.87,157.11) .. (350.39,157.11) .. controls (350.91,157.11) and (351.33,157.53) .. (351.33,158.06) .. controls (351.33,158.58) and (350.91,159) .. (350.39,159) .. controls (349.87,159) and (349.44,158.58) .. (349.44,158.06) -- cycle ;
\draw  [fill={rgb, 255:red, 0; green, 0; blue, 0 }  ,fill opacity=1 ] (320.44,107.06) .. controls (320.44,106.53) and (320.87,106.11) .. (321.39,106.11) .. controls (321.91,106.11) and (322.33,106.53) .. (322.33,107.06) .. controls (322.33,107.58) and (321.91,108) .. (321.39,108) .. controls (320.87,108) and (320.44,107.58) .. (320.44,107.06) -- cycle ;
\draw  [fill={rgb, 255:red, 0; green, 0; blue, 0 }  ,fill opacity=1 ] (425.44,178.06) .. controls (425.44,177.53) and (425.87,177.11) .. (426.39,177.11) .. controls (426.91,177.11) and (427.33,177.53) .. (427.33,178.06) .. controls (427.33,178.58) and (426.91,179) .. (426.39,179) .. controls (425.87,179) and (425.44,178.58) .. (425.44,178.06) -- cycle ;

\draw (131.24,146.4) node [anchor=north west][inner sep=0.75pt]  [font=\tiny]  {$1$};
\draw (299.14,145.11) node [anchor=north west][inner sep=0.75pt]    {$D$};
\draw (314.57,95.69) node [anchor=north west][inner sep=0.75pt]  [font=\tiny]  {$w_{1}$};
\draw (348.57,163.11) node [anchor=north west][inner sep=0.75pt]  [font=\tiny]  {$w_{2}$};
\draw (210.14,115.69) node [anchor=north west][inner sep=0.75pt]  [font=\tiny]  {$\pi ( x)$};
\draw (235.14,173.97) node [anchor=north west][inner sep=0.75pt]  [font=\tiny]  {$\pi ( y)$};
\draw (425.14,106.83) node [anchor=north west][inner sep=0.75pt]  [font=\tiny]  {$u$};
\draw (427.14,171.11) node [anchor=north west][inner sep=0.75pt]  [font=\tiny]  {$v$};
\draw (233.5,145.65) node [anchor=north west][inner sep=0.75pt]  [font=\tiny]  {$S_{1}$};
\draw (273.75,88.15) node [anchor=north west][inner sep=0.75pt]  [font=\tiny]  {$S_{2}$};
\draw (290.75,196.65) node [anchor=north west][inner sep=0.75pt]  [font=\tiny]  {$S_{3}$};
\draw (348,124.9) node [anchor=north west][inner sep=0.75pt]  [font=\tiny]  {$S_{4}$};

\end{tikzpicture}

\caption{A diagram $D \rightarrow K_{d_t}$ found to be in $K_{d_s}$ with overwhelming probability.
For the restricted abstract diagram $X$ the dashed line represents $r_X^{-1}(1)$.}
\label{f : possible diagram}
\end{figure}

Consider a connected reduced diagram $D \rightarrow K_{d_t}$ homeomorphic to a disk for the quadrilateral as seen in Figure \ref{f : possible diagram} with $\lvert \partial D \rvert \leq(A'' + 2A)\epsilon^{-1} l+ 3l$ and $\lvert D \rvert \leq (A'' + 2A)\epsilon^{-2} + 3 \epsilon^{-1} $ by the linear isoperimetric inequality, Theorem \ref{t : generic linear iso}, with a geodesic $\alpha$ in $K_{d_t}$ the boundary path between $w_1$ and $w_2$.
By construction we have the sides $S_1$, $S_2$, $S_3$ and $S_4$ satisfying $\lvert S_2 \rvert = \lvert S_3 \rvert = A \epsilon^{-1} l $  so $\lvert S_4 \rvert \leq A''\epsilon^{-1} l + l < 2A \epsilon^{-1} l \leq \lvert S_1 \rvert + \lvert S_2 \rvert + \lvert S_3 \rvert $.
The diagram $D$ gives a reduced connected abstract diagram $(X,r_X, \text{lab})$ restricted in $S$ with $r_{X}^{-1}(1) = P_1^{-1} \cdot C \cdot P_2$ that is fillable by $\mathcal{R}^{d_t}_{m,l}$.
The above observation about side lengths implies $2 \lvert r_X^{-1}(1) \rvert \geq \lvert \partial X \rvert$.
Define $\mathbf{W}_{d_s,l} $ as the possible labels of $P_1^{-1} \cdot C \cdot P_2$ and
note by using (\ref{eq : count bound})
\begin{align*}
    \lvert \mathbf{W}_{d_s,l} \lvert \; &\leq 2l(2m-1)^{d_s l } \lvert \mathbf{E}_{A \epsilon^{-1} l} \rvert^2 \\
    &\leq l^5(A\epsilon^{-1}l)^{\frac{80A}{d_s \epsilon}} (2m-1)^{8 \cdot 10^{-7}\epsilon^2(A \epsilon^{-1}+ 2\epsilon)l} \\
    & \leq l^5(A\epsilon^{-1}l)^{\frac{2 \cdot 10^6}{d_s \epsilon}} (2m-1)^{3 \cdot 10^{-2}\epsilon l}
\end{align*}
\noindent
by our choice of $A$, $S_2$ and $S_3$ being emanating paths using Lemma \ref{l : counting emanating words intrinsic} and that we have already estimated the number of labels for $S_1$.

Setting $C = (A'' + 2A)\epsilon^{-2} + 3 \epsilon^{-1}$, $C' = (A'' + 2A)\epsilon^{-1} + 3$, $g = 3 \cdot 10^{-2}$ and computing $ d_s C + g + d_t < 1/2 - (\epsilon - (75 \cdot 10^{-4} + 3 \cdot 10^{-7} + 3 \cdot 10^{-2}) \epsilon ) $ the hypothesis of Lemma \ref{l : filling diagrams along paths} is satisfied with $d > \epsilon - (75 \cdot 10^{-4} + 3 \cdot 10^{-7} + 3 \cdot 10^{-2}) \epsilon  > \frac{19}{20} \epsilon >  0$.

We conclude for some polynomial $Q(l)$ the probability that $(X, r_X, \text{lab})$ is fillable by 
$\mathcal{R}^{d_t}_{m,l}$ with boundary label of the restricted path in $\mathbf{W}_{d_s,l}$, $\partial X \setminus r_X^{-1}(1)$ a geodesic in $K_{d_t}$ and at least one face $f$ of $X$ bearing a relator $r \in \mathcal{R}^{d_t}_{m,l} \setminus \mathcal{R}^{d_s}_{m,l}$ is less than $Q(l)(2m-1)^{-dl}$. 
As by assumption $(X,r_X, \text{lab})$ is fillable by $\mathcal{R}^{d_t}_{m,l}$ with overwhelming probability in $\mathcal{G}^{d_t}_{m,l}$, this implies that $D$ is a van-Kampen diagram in $K_{d_s}$ and hence $\rho_s(w_1, w_2) = \rho_t(w_1, w_2) \leq A'' \epsilon^{-1} l$.

To find a contradiction consider points $\overline{w_1}, \overline{w_2}, r \in [\varphi(p), \varphi(q) ] $ such that $\rho_s(w_1,\overline{w_1}) \leq l$, $\rho_s(w_2,\overline{w_2}) \leq l$ and $r \in \partial L'$ as seen in Figure \ref{f : final configuration}.
We can find points $\overline{w_1}, \overline{w_2}$ as $L \setminus L'$ is a disconnected set of three ladders and both $\rho_s(x,w_1)$ and $\rho_s(y,w_2) $ are  $ \geq A \epsilon^{-1}l$.
\begin{align*}
    \rho_s(w_1, w_2) & \geq \rho_s(\overline{w_1},r) + \rho_s(\overline{w_2},r) - 2l \tag{Geodesic sides of $\Delta$, triangle inequality and $\rho_s(w_i, \overline{w_i}) \leq l$} \\
    & \geq \rho_s(w_1,x) - l - \rho_s(r,x) + \rho_s(w_2,y) - l - \rho_s(r,y) - 2l  \tag{Triangle inequality and $\rho_s(w_i, \overline{w_i}) \leq l$} \\
    & \geq 2A\epsilon^{-1}l - 2\lvert L' \rvert l - 4l \tag{$ \rho_s(x,w_1) = \rho_s(y,w_2) = A\epsilon^{-1}l$ and $x,y,r \in \partial L'$}  \\
    & \geq 2A\epsilon^{-1}l - 10l \tag{$ \lvert L' \rvert \leq 3$} \\
\end{align*}
\noindent
giving us a contradiction as $ A'' \epsilon^{-1} l < 2A \epsilon^{-1} l - 10 l $.

To observe $(u,v)_{\Bar{p}} \leq (A + A') \epsilon^{-1} l + 2R $ we compute $(u,v)_{\Bar{p}} -2R = \min\{\rho_t(x', \newline z_1'),  \rho_t(y',z_2')\} - 2R \leq (A + A') \epsilon^{-1} l$ due to an application of the triangle inequality, the claim's first assertion and that $\chi$ is isometric restricted on the geodesic sides.

\end{proof}
\end{claim}
Define $ \Bar{x} = \varphi^{-1}(x), \Bar{y} =  \varphi^{-1}(y)$ and without loss of generality assume $\Bar{p} = x'$, we compute:
\begin{alignat*}{3}
  \rho_{t}(u,v) &= \rho_t(u,\Bar{p}) + \rho_t(v,\Bar{p}) - 2(u,v)_{\Bar{p}} \tag{Gromov product}\\
  & \geq \rho_t(u,x') + \rho_t(v,y') - \rho_t(x',y') - 2(u,v)_{\Bar{p}} \tag{Triangle inequality}\\
  & \geq \rho_t(u,\pi(x)) + \rho_t(v,\pi(y)) - (1 + 2(A + A') \epsilon^{-1}) l - 4R \tag{Claim. \ref{cl : gromov product bound} and $\rho_t(\pi(x),\pi(y)) \leq l$}\\
  & \geq \frac{5}{13} \rho_s(p,x) + \frac{5}{13}\rho_s(q,y) - (1 + (2(A + A') + 8) \epsilon^{-1}) l - 4R \tag{Corollary \ref{c : emanating rays are quasi-geodesics}}\\
  & \geq \frac{5}{13} \rho_A(p,q) - (1 + \lambda + (2(A + A') + 8) \epsilon^{-1}) l - \lambda k - 4R\tag{$\varphi \circ \gamma_{1, \cdot} $ geodesics in $K_{d_s}$, triangle inequality and $\rho_A(\Bar{x},\Bar{y}) \leq \lambda l + k$}\\
\end{alignat*}
Since $ (A^{(1)}, \rho_A) \rightarrow (K_{d_s}, \rho_s) \rightarrow (K_{d_t}, \rho_t) $ maps cells to cells as a combinatorial map we have $\rho_t(u,v) \leq \rho_A(p,q)$. 
This proves the lemma.

\end{proof}
\end{lemma}

\begin{remark}
From \cite[Lem.8.15]{Mackay-conf-rand-16} we can take $\lambda = 1 $, $k = 8l$ and by \cite[Thm.1.1]{Gouezel-Shchur-corrected-quantitative-version-of-the-morse-lemma} we can take $R = 4 \cdot 10^3 \epsilon^{-1} l$ so one can see the composition in Lemma \ref{l : quasi-isometric-embedding} recalling $A = 25 \cdot 10^5$ and $A' = 8 \cdot 10^3 $ as, say, a $(13/5, 10^7(1-2d)^{-1} l)$-quasi-isometric embedding in terms of the model parameters $d$ and $l$.
\end{remark}

We have undistorted combinatorial round trees at all densities so Theorem \ref{t : round trees everywhere} is proved. 
Additionally, every geodesic in the $1$-skeleton of the combinatorial round tree $A$ is an embedded quasi-geodesic in $K_{d_t}$.

\subsection{The Lower Bound}

Here we justify Theorem \ref{t : main result} from Theorem \ref{t : round trees everywhere} with the application of Theorem \ref{t : undistorted round trees lower bound}.

\begin{proof}[Proof of Theorem \ref{t : main result}]
Let $1/8 \leq d < 1/2$ and $\epsilon = 1/2 - d$, Theorem \ref{t : round trees everywhere} says the property of containing an undistorted round tree $A((2m-1)^{\beta \eta l}, H)$ with $\beta = 10^{-7}\epsilon^2$, $\eta = 4^{-1} 10^{-8} \epsilon^3$, $ H = 40 \cdot 10^{14} \epsilon^{-5}$ occurs with overwhelming probability in $\mathcal{G}^{d}_{m,l}$.
The following logarithms are to the base $2m-1$.
By Theorem \ref{t : undistorted round trees lower bound},

$$ \text{Confdim}(\partial_\infty G) \geq 1 + \frac{\log(V)}{\log(H)} = 1 + \frac{\beta \eta l}{\log(H)}.$$

\noindent
for $G \sim \mathcal{G}^{d}_{m,l}$ with overwhelming probability in $\mathcal{G}^{d}_{m,l}$.
Underlying this is $d_t = d$ and $d_s = 10^{-7} \epsilon^3 $ (fitting in with Section \ref{s : build round trees all densities} and choices for the probabilities to work out).
So,

$$\text{Confdim}(\partial_ \infty G) > 1 + \frac{(1/2 - d)^5}{C\log\frac{1}{(1/2 - d)^{5}}} l$$ 

\noindent
where $ C = 10^{17} \geq 1$, say,  with overwhelming probability in $\mathcal{G}^{d}_{m,l}$.
To get the exact formulation stated in the introduction one takes a larger $C$.
\end{proof}

\bibliographystyle{alpha}

\begin{thebibliography}{dLdlS21}

\bibitem[ALS15]{Antoniuk-Luczak-Swicatkowski-random-triangular-groups-at-a-third}
Sylwia Antoniuk, Tomasz Luczak, and Jacek \'{S}wi\c{a}tkowski.
\newblock Random triangular groups at density 1/3.
\newblock {\em Compos. Math.}, 151(1):167--178, 2015.

\bibitem[BF20]{Bishop-Ferov-density-of-metric-small-cancellation-in-finitiely-presented-groups}
Alex Bishop and Michal Ferov.
\newblock Density of metric small cancellation in finitely presented groups.
\newblock {\em J. Groups Complex. Cryptol.}, 12(2):Paper No. 1, 18, 2020.

\bibitem[BH99]{Bridson-Haefliger-non-positive-curv-geom}
Martin~R. Bridson and Andr\'{e} Haefliger.
\newblock {\em Metric spaces of non-positive curvature}, volume 319 of {\em
  Grundlehren der Mathematischen Wissenschaften [Fundamental Principles of
  Mathematical Sciences]}.
\newblock Springer-Verlag, Berlin, 1999.

\bibitem[BK13]{Bourdon-Kleiner-combinatorial-modulus-lowener-property-coxeter-groups}
Marc Bourdon and Bruce Kleiner.
\newblock Combinatorial modulus, the combinatorial {L}oewner property, and
  {C}oxeter groups.
\newblock {\em Groups Geom. Dyn.}, 7(1):39--107, 2013.

\bibitem[BK15]{Bourdon-Kleiner-applications-lpcohomology-boundaries-gromov-hyp}
Marc Bourdon and Bruce Kleiner.
\newblock Some applications of {$\ell_p$}-cohomology to boundaries of {G}romov
  hyperbolic spaces.
\newblock {\em Groups Geom. Dyn.}, 9(2):435--478, 2015.

\bibitem[Bou16]{Bourdon-cohomologie-and-isometric-actions-on-Lp-spaces}
Marc Bourdon.
\newblock Cohomologie et actions isom\'{e}triques propres sur les espaces
  {$L_p$}.
\newblock In {\em Geometry, topology, and dynamics in negative curvature},
  volume 425 of {\em London Math. Soc. Lecture Note Ser.}, pages 84--109.
  Cambridge Univ. Press, Cambridge, 2016.

\bibitem[CW15]{Calegari-Waker-surfsub-rand-15}
Danny Calegari and Alden Walker.
\newblock Random groups contain surface subgroups.
\newblock {\em J. Amer. Math. Soc.}, 28(2):383--419, 2015.

\bibitem[DGP11]{Dahmani-Guirardel-Przytycki-random-groups-do-not-split}
Fran\c{c}ois Dahmani, Vincent Guirardel, and Piotr Przytycki.
\newblock Random groups do not split.
\newblock {\em Math. Ann.}, 349(3):657--673, 2011.

\bibitem[dLdlS21]{Laat-Salle-banach-space-actions-L2-spectral-gap}
Tim de~Laat and Mikael de~la Salle.
\newblock Banach space actions and {$L^2$}-spectral gap.
\newblock {\em Anal. PDE}, 14(1):45--76, 2021.

\bibitem[DM19]{Drutu-Mackay-random-groups-random-graphs-eigenvalies-Lp-laplacians}
Cornelia Dru\c{t}u and John~M. Mackay.
\newblock Random groups, random graphs and eigenvalues of {$p$}-{L}aplacians.
\newblock {\em Adv. Math.}, 341:188--254, 2019.

\bibitem[Gro93]{Gromov-invariants}
M.~Gromov.
\newblock Asymptotic invariants of infinite groups.
\newblock In {\em Geometric group theory, {V}ol. 2 ({S}ussex, 1991)}, volume
  182 of {\em London Math. Soc. Lecture Note Ser.}, pages 1--295. Cambridge
  Univ. Press, Cambridge, 1993.

\bibitem[Gro03]{Gromov-random-walk-in-random-groups}
M.~Gromov.
\newblock Random walk in random groups.
\newblock {\em Geom. Funct. Anal.}, 13(1):73--146, 2003.

\bibitem[GS19]{Gouezel-Shchur-corrected-quantitative-version-of-the-morse-lemma}
S\'{e}bastien Gou\"{e}zel and Vladimir Shchur.
\newblock Corrigendum: A corrected quantitative version of the morse lemma.
\newblock {\em J. Funct. Anal.}, 277(4):1258--1268, 2019.

\bibitem[KK13]{Kotowski-Kotowski-random-groups-and-propertyT-Zuk-revisited}
Marcin Kotowski and Michal Kotowski.
\newblock Random groups and property {$(T)$}: \.{Z}uk's theorem revisited.
\newblock {\em J. Lond. Math. Soc. (2)}, 88(2):396--416, 2013.

\bibitem[KS08]{Kapovich-Schupp-group-random-mods}
Ilya Kapovich and Paul Schupp.
\newblock On group-theoretic models of randomness and genericity.
\newblock {\em Groups Geom. Dyn.}, 2(3):383--404, 2008.

\bibitem[Mac12]{Mackay-conf-rand-12}
John~M. Mackay.
\newblock Conformal dimension and random groups.
\newblock {\em Geom. Funct. Anal.}, 22(1):213--239, 2012.

\bibitem[Mac16]{Mackay-conf-rand-16}
John~M. Mackay.
\newblock Conformal dimension via subcomplexes for small cancellation and
  random groups.
\newblock {\em Math. Ann.}, 364(3-4):937--982, 2016.

\bibitem[MT10]{Mackay-Tyson-survey-conformal-dimension}
John~M. Mackay and Jeremy~T. Tyson.
\newblock {\em Conformal dimension}, volume~54 of {\em University Lecture
  Series}.
\newblock American Mathematical Society, Providence, RI, 2010.
\newblock Theory and application.

\bibitem[MW02]{McCammond-Wise-small-cancellation-fans-ladders}
Jonathan~P. McCammond and Daniel~T. Wise.
\newblock Fans and ladders in small cancellation theory.
\newblock {\em Proc. London Math. Soc. (3)}, 84(3):599--644, 2002.

\bibitem[Oll05]{ollivier-survey}
Yann Ollivier.
\newblock {\em A January 2005 invitation to random groups}, volume~10 of {\em
  Ensaios Matem\'aticos [Mathematical Surveys]}.
\newblock Sociedade Brasileira de Matem\'atica, Rio de Janeiro, 2005.

\bibitem[Opp21]{Oppenheim-banach-fixed-point-theorems-groups}
Izhar Oppenheim.
\newblock Banach {Z}uk's criterion for partite complexes with application to
  random groups.
\newblock {\em arXiv:2112.02929 [math.GR]}, 2021.

\bibitem[OW11]{Ollivier-Wise-cubulating-low-density-random-groups}
Yann Ollivier and Daniel~T. Wise.
\newblock Cubulating random groups at density less than {$1/6$}.
\newblock {\em Trans. Amer. Math. Soc.}, 363(9):4701--4733, 2011.

\bibitem[Pan89]{Pansu-conformal-dimension}
Pierre Pansu.
\newblock M\'{e}triques de {C}arnot-{C}arath\'{e}odory et quasiisom\'{e}tries
  des espaces sym\'{e}triques de rang un.
\newblock {\em Ann. of Math. (2)}, 129(1):1--60, 1989.

\bibitem[Riv10]{Rivin-growth-free-groups}
Igor Rivin.
\newblock Growth in free groups (and other stories)---twelve years later.
\newblock {\em Illinois J. Math.}, 54(1):327--370, 2010.

\bibitem[Str90]{Strebel-small-cancellation}
Ralph Strebel.
\newblock Appendix. {S}mall cancellation groups.
\newblock In {\em Sur les groupes hyperboliques d'apr\`es {M}ikhael {G}romov
  ({B}ern, 1988)}, volume~83 of {\em Progr. Math.}, pages 227--273.
  Birkh\"{a}user Boston, Boston, MA, 1990.

\bibitem[Tsa21]{Tsai-density-random-subsets-applications-to-group-theory}
Tsung-Hsuan Tsai.
\newblock Density of random subsets and applications to group theory.
\newblock {\em arXiv:2104.09192 [math.GR]}, 2021.

\bibitem[Wis04]{Wise-cubulating-small-cancellation}
D.~T. Wise.
\newblock Cubulating small cancellation groups.
\newblock {\em Geom. Funct. Anal.}, 14(1):150--214, 2004.

\bibitem[\.{Z}03]{Zuk-propertyT-kazhdan-constants-for-discrete-groups}
A.~\.{Z}uk.
\newblock Property ({T}) and {K}azhdan constants for discrete groups.
\newblock {\em Geom. Funct. Anal.}, 13(3):643--670, 2003.

\end{thebibliography}

\end{document}